\documentclass[10pt,a4paper,reqno]{amsart}
\title{Spherical DG-functors}
\author{Rina Anno}
\email{anno@pitt.edu}
\address{Department of Mathematics\\
University of Pittsburgh\\
301 Thackeray Hall\\
Pittsburgh, PA 15260\\
USA}
\author{Timothy Logvinenko} 
\email{LogvinenkoT@cardiff.ac.uk} 
\address{School of Mathematics\\ 
Cardiff University\\
Senghennydd Road,\\
Cardiff, CF24 4AG\\
UK}
\usepackage{amsmath,amsfonts,amssymb,amsthm,epsfig,amscd,latexsym,comment}
\usepackage{caption}
\usepackage{graphicx}
\usepackage{array}
\usepackage{subfigure}
\usepackage{leftidx}
\usepackage[all]{xy}
\input xy
\xyoption{all}
\newdir^{ (}{{}*!/-5pt/\dir^{(}}

\usepackage[colorlinks=true, pdfpagemode=none, pdfmenubar=false, pdfstartview=FitH, linkcolor=blue, citecolor=blue, urlcolor=blue, pdffitwindow=false]{hyperref}

\DeclareMathOperator{\homm}{Hom}
\DeclareMathOperator{\shhomm}{{\it\mathcal{H}om\rm}}
\DeclareMathOperator{\eend}{End}

\DeclareMathOperator{\autm}{Aut}

\DeclareMathOperator{\picr}{Pic}
\DeclareMathOperator{\tot}{tot}

\DeclareMathOperator{\cl}{Cl}

\DeclareMathOperator{\spec}{Spec\;}

\DeclareMathOperator{\ext}{Ext}
\DeclareMathOperator{\Ext}{Ext}

\DeclareMathOperator{\ev}{ev}
\DeclareMathOperator{\trace}{tr}

\DeclareMathOperator{\action}{act}

\DeclareMathOperator{\modd}{\bf Mod}
\DeclareMathOperator{\lder}{\bf L}
\DeclareMathOperator{\rder}{\bf R}
\DeclareMathOperator{\ldertimes}{\overset{\lder}{\otimes}}
\DeclareMathOperator{\rderhom}{\rder\homm}
\DeclareMathOperator{\rdershom}{\rder\shhomm}

\DeclareMathOperator{\id}{Id}

\DeclareMathOperator{\cone}{Cone}
\DeclareMathOperator{\opp}{{opp}}
\DeclareMathOperator{\fg}{{\it fg}}
\DeclareMathOperator{\qrep}{\it \mathcal{Q}r}
\DeclareMathOperator{\hproj}{\mathcal{P}}
\DeclareMathOperator{\bhproj}{\bar{\mathcal{P}}}
\DeclareMathOperator{\acyc}{\it \mathcal{A}c}
\DeclareMathOperator{\intmod}{{\it \mathcal{I}nt}}
\DeclareMathOperator{\semifree}{\mathcal{S}\mathcal{F}}
\DeclareMathOperator{\sffg}{\mathcal{S}\mathcal{F}_{\fg}}
\DeclareMathOperator{\perf}{{\it \mathcal{P}erf}}
\DeclareMathOperator{\hmtpy}{{Ho}}
\DeclareMathOperator{\Morita}{{Mrt}}

\DeclareMathOperator{\pretriag}{{Pre\text{-}Tr}}
\DeclareMathOperator{\DGFun}{{DGFun}}
\DeclareMathOperator{\Fun}{{Fun}}
\DeclareMathOperator{\exfun}{{ExFun}}
\DeclareMathOperator{\cts}{{cts}}
\DeclareMathOperator{\convol}{{}}
\begin{document}

\def\bv{\mathbf{v}}
\def\kgc_{K^*_G(\mathbb{C}^n)}
\def\kgchi_{K^*_\chi(\mathbb{C}^n)}
\def\kgcf_{K_G(\mathbb{C}^n)}
\def\kgchif_{K_\chi(\mathbb{C}^n)}
\def\gpic_{G\text{-}\picr}
\def\gcl_{G\text{-}\cl}
\def\trch_{{\chi_{0}}}
\def\regring{{R}}
\def\regrep{{V_{\text{reg}}}}
\def\givrep{{V_{\text{giv}}}}
\def\lbar{{(\mathbb{Z}^n)^\vee}}
\def\genpx_{{p_X}}
\def\genpy_{{p_Y}}
\def\genpcn_{p_{\mathbb{C}^n}}
\def\gnat{gnat}
\def\twalg{{\regring \rtimes G}}
\def\L{{\mathcal{L}}}
\def\O{{\mathcal{O}}}
\def\gcd{\mbox{gcd}}
\def\lcm{\mbox{lcm}}
\def\tf{{\tilde{f}}}
\def\tD{{\tilde{D}}}
\def\A{{\mathcal{A}}}
\def\B{{\mathcal{B}}}
\def\C{{\mathcal{C}}}
\def\D{{\mathcal{D}}}
\def\barA{{\bar{\mathcal{A}}}}
\def\barAi{{\bar{\mathcal{A}}_1}}
\def\barAj{{\bar{\mathcal{A}}_2}}
\def\barB{{\bar{\mathcal{B}}}}
\def\barC{{\bar{\mathcal{C}}}}
\def\M{{\mathcal{M}}}
\def\Aopp{{\A^{\opp}}}
\def\Bopp{{\B^{\opp}}}
\def\Copp{{\C^{\opp}}}
\def\aA{\leftidx{_{a}}{\A}}
\def\bA{\leftidx{_{b}}{\A}}
\def\Aa{{\A_a}}
\def\Ea{E_a}
\def\aE{\leftidx{_{a}}{E}{}}
\def\Eb{E_b}
\def\bE{\leftidx{_{b}}{E}{}}
\def\Fa{F_a}
\def\aF{\leftidx{_{a}}{F}{}}
\def\Fb{F_b}
\def\bF{\leftidx{_{b}}{F}{}}
\def\aM{\leftidx{_{a}}{M}{}}
\def\aMb{\leftidx{_{a}}{M}{_{b}}}
\def\Ma{{M_a}}
\def\modk{{\modd\text{-}k}}
\def\modA{{\modd\text{-}\A}}
\def\modB{{\modd\text{-}\B}}
\def\sfA{{\semifree(\A)}}
\def\sfB{{\semifree(\B)}}
\def\sffgA{{\sffg(\A)}}
\def\sffgB{{\sffg(\B)}}
\def\hprojA{{\hproj(\A)}}
\def\hprojB{{\hproj(\B)}}
\def\qrepA{{\qrep(\A)}}
\def\qrepB{{\qrep(\B)}}
\def\opp{{\text{opp}}}
\def\perfsf{{\semifree^{\perf}}}
\def\prfhpr{{\hproj^{\scriptscriptstyle\perf}}}
\def\prfhprA{{\prfhpr(\A)}}
\def\prfhprB{{\prfhpr(\B)}}
\def\prfhprAopp{{\prfhpr(\Aopp)}}
\def\prfhprBopp{{\prfhpr(\Bopp)}}
\def\perfsfA{{\perfsf(\A)}}
\def\perfsfB{{\perfsf(\B)}}
\def\qrhpr{{\hproj^{qr}}}
\def\qrhprA{{\qrhpr(\A)}}
\def\qrhprB{{\qrhpr(\B)}}
\def\qrsf{{\semifree^{qr}}}
\def\qrsf{{\semifree^{qr}}}
\def\qrsfA{{\qrsf(\A)}}
\def\qrsfB{{\qrsf(\B)}}
\def\Aperfsf{{\semifree^{\A\text{-}\perf}(\AbimB)}}
\def\Bperfsf{{\semifree^{\B\text{-}\perf}(\AbimB)}}
\def\Aprfhpr{{\hproj^{\A\text{-}\perf}(\AbimB)}}
\def\Bprfhpr{{\hproj^{\B\text{-}\perf}(\AbimB)}}
\def\Aqrhpr{{\hproj^{\A\text{-}qr}(\AbimB)}}
\def\Bqrhpr{{\hproj^{\B\text{-}qr}(\AbimB)}}
\def\Aqrsf{{\semifree^{\A\text{-}qr}(\AbimB)}}
\def\Bqrsf{{\semifree^{\B\text{-}qr}(\AbimB)}}
\def\modAopp{{\modd\text{-}\Aopp}}
\def\modBopp{{\modd\text{-}\Bopp}}
\def\AmodA{{\A\text{-}\modd\text{-}\A}}
\def\AmodB{{\A\text{-}\modd\text{-}\B}}
\def\AmodC{{\A\text{-}\modd\text{-}\C}}
\def\BmodA{{\B\text{-}\modd\text{-}\A}}
\def\BmodB{{\B\text{-}\modd\text{-}\B}}
\def\BmodC{{\B\text{-}\modd\text{-}\C}}
\def\CmodA{{\C\text{-}\modd\text{-}\A}}
\def\CmodB{{\C\text{-}\modd\text{-}\B}}
\def\CmodC{{\C\text{-}\modd\text{-}\C}}
\def\AbimA{{\A\text{-}\A}}
\def\AbimC{{\A\text{-}\C}}
\def\BbimA{{\B\text{-}\A}}
\def\BbimB{{\B\text{-}\B}}
\def\BbimC{{\B\text{-}\C}}
\def\CbimA{{\C\text{-}\A}}
\def\CbimB{{\C\text{-}\B}}
\def\CbimC{{\C\text{-}\C}}
\def\AhprA{{\hproj\left(\AbimA\right)}}
\def\BhprB{{\hproj\left(\BbimB\right)}}
\def\AhprB{{\hproj\left(\AbimB\right)}}
\def\BhprA{{\hproj\left(\BbimA\right)}}
\def\QAbimB{{Q\A\text{-}\B}}
\def\AbimB{{\A\text{-}\B}}
\def\AonebimB{{\A_1\text{-}\B}}
\def\AtwobimB{{\A_2\text{-}\B}}
\def\BbimA{{\B\text{-}\A}}
\def\Aperf{{\A\text{-}\perf}}
\def\Bperf{{\B\text{-}\perf}}
\def\MddA{{M^{\tilde{\A}}}}
\def\MddB{{M^{\tilde{\B}}}}
\def\MhdA{{M^{h\A}}}
\def\MhdB{{M^{h\B}}}
\def\NhdB{{N^{h\B}}}
\def\Cat{{Cat}}
\def\DGCat{{DG\text{-}Cat}}
\def\HoDGCat{{\hmtpy(\DGCat)}}
\def\HoDGCatV{{\hmtpy(\DGCat_\mathbb{V})}}
\def\tr{{tr}}
\def\pretr{{pretr}}
\def\kctr{{kctr}}
\def\PreTrCat{{\DGCat^\pretr}}
\def\KcTrCat{{\DGCat^\kctr}}
\def\HoPretrCat{{\hmtpy(\PreTrCat)}}
\def\HoKcTrCat{{\hmtpy(\KcTrCat)}}
\def\Aquasirep{{\A\text{-}qr}}
\def\QAquasirep{{Q\A\text{-}qr}}
\def\Bquasirep{{\B\text{-}qr}} 
\def\lderA{{\tilde{\A}}} 
\def\lderB{{\tilde{\B}}} 
\def\adjunit{{\text{adj.unit}}}
\def\adjcounit{{\text{adj.counit}}}
\def\degzero{{\text{deg.0}}}
\def\degone{{\text{deg.1}}}
\def\degminusone{{\text{deg.-$1$}}}
\def\bareta{{\overline{\eta}}}
\def\barzeta{{\overline{\zeta}}}
\def\Ract{{R {\action}}}
\def\barRact{{\overline{\Ract}}}
\def\actL{{{\action} L}}
\def\baractL{{\overline{\actL}}}

\theoremstyle{definition}
\newtheorem{defn}{Definition}[section]
\newtheorem*{defn*}{Definition}
\newtheorem{exmpl}[defn]{Example}
\newtheorem*{exmpl*}{Example}
\newtheorem{exrc}[defn]{Exercise}
\newtheorem*{exrc*}{Exercise}
\newtheorem*{chk*}{Check}
\newtheorem*{remarks*}{Remarks}
\theoremstyle{plain}
\newtheorem{theorem}{Theorem}[section]
\newtheorem*{theorem*}{Theorem}
\newtheorem{conj}[defn]{Conjecture}
\newtheorem*{conj*}{Conjecture}
\newtheorem{prps}[defn]{Proposition}
\newtheorem*{prps*}{Proposition}
\newtheorem{cor}[defn]{Corollary}
\newtheorem*{cor*}{Corollary}
\newtheorem{lemma}[defn]{Lemma}
\newtheorem*{claim*}{Claim}
\newtheorem{Specialthm}{Theorem}
\renewcommand\theSpecialthm{\Alph{Specialthm}}
\numberwithin{equation}{section}
\renewcommand{\textfraction}{0.001}
\renewcommand{\topfraction}{0.999}
\renewcommand{\bottomfraction}{0.999}
\renewcommand{\floatpagefraction}{0.9}
\setlength{\textfloatsep}{5pt}
\setlength{\floatsep}{0pt}
\setlength{\abovecaptionskip}{2pt}
\setlength{\belowcaptionskip}{2pt}
\begin{abstract}
For two DG-categories $\A$ and $\B$
we define the notion of a \em spherical \rm Morita quasi-functor 
$\A \rightarrow \B$. We construct its associated autoequivalences: 
the \em twist \rm $T \in \autm D(\B)$ and 
the \em co-twist \rm $F \in \autm D(\A)$. 
We give sufficiency criteria for a quasi-functor to 
be spherical and for the twists associated to a collection 
of spherical quasi-functors to braid. Using the framework 
of DG-enhanced triangulated categories, we translate 
all of the above to Fourier-Mukai transforms between the derived
categories of algebraic varieties. This is a broad 
generalisation of the results on spherical objects in
\cite{SeidelThomas-BraidGroupActionsOnDerivedCategoriesOfCoherentSheaves} 
and on spherical functors in \cite{Anno-SphericalFunctors}. In fact, 
this paper replaces \cite{Anno-SphericalFunctors}, which has
a fatal gap in the proof of its main theorem. Though conceptually
correct, the proof was impossible to fix within the framework 
of triangulated categories. 
\end{abstract}

\maketitle

\section{Introduction} \label{section-intro}

Let $X$ be a smooth projective variety over an algebraically closed
field $k$ of characteristic $0$. 
Let $D(X)$ be the bounded derived category of coherent sheaves on $X$.
In \cite{SeidelThomas-BraidGroupActionsOnDerivedCategoriesOfCoherentSheaves}
Seidel and Thomas introduced the notion of a \em spherical object \rm 
in $D(X)$. These objects are defined in terms of certain 
cohomological properties and they are mirror-symmetric 
analogues of Lagrangian spheres on a symplectic manifold. 
Given a Lagrangian sphere we can associate to it a symplectic 
automorphism called the \em generalised Dehn twist\rm. Correspondingly:
\begin{theorem*}[\cite{SeidelThomas-BraidGroupActionsOnDerivedCategoriesOfCoherentSheaves}]
Let $E \in D(X)$. The twist functor $T_E$ 
is a certain functorial cone of the natural transformation 
$E \otimes_k \rder\homm_X(E, -) 
\xrightarrow{\mathrm{eval}}  \id_{D(X)}$. 
If $E$ is spherical, then $T_E$ is an autoequivalence of $D(X)$. 
\end{theorem*}
Moreover, in
\cite[Theorem 2.17]{SeidelThomas-BraidGroupActionsOnDerivedCategoriesOfCoherentSheaves} 
Seidel and Thomas give simple criteria on a set $E_1$, \dots, $E_n$ 
of spherical objects in $D(X)$ sufficient to ensure that the
corresponding spherical twists $T_1$, \dots, $T_n$ 
represent the braid group $B_n$. In other words, that we have:
$$ T_{i} T_{j} T_{i} \simeq T_{j} T_{i} T_{j} \quad \quad |i - j| = 1, $$
$$ T_{i} T_{j} \simeq T_{j} T_{i} \quad \quad |i - j| \geq 2. $$

Spherical objects and twists quickly became an essential tool 
in studying derived categories of algebraic varieties as well 
as more classical areas of algebraic geometry
\cite{Mukai-OnTheModuliSpaceOfBundlesOnK3SurfacesI}, 
\cite{Bridgeland-StabilityConditionsOnK3Surfaces},
\cite{Bridgeland-StabilityConditionsAndKleinianSingularities}, 
\cite{IshiiUehara-AutoequivalencesOfDerivedCategoriesOnTheMinimalResolutionsOfA_nSingularitiesOnSurfaces},    
\cite{BroomheadPloog-AutoequivalencesOfToricSurfaces}.
For some time now it was understood by specialists
that the notion of a spherical object should generalise to 
the notion of a \em spherical functor \rm $D(Z) \xrightarrow{s} D(X)$
where $Z$ is some other variety. Such functor should produce 
two auto-equivalences --- the \em twist \rm $t \in \autm D(X)$ and 
the \em co-twist \rm $f \in \autm D(Z)$. More generally, there 
should be a notion of a spherical functor between two abstract 
triangulated categories. Limited special cases of this appear in 
\cite{Horja-DerivedCategoryAutomorphismsFromMirrorSymmetry}, 
\cite{Rouquier-CategorificationOfTheBraidGroups}
\cite{Szendroi-ArtinGroupActionsOnDerivedCategoriesOfThreefolds}, 
\cite{Toda-OnACertainGeneralizationOfSphericalTwists}, 
\cite{KhovanovThomas-BraidCobordismsTriangulatedCategoriesAndFlagVarieties}, 
but general treatment was obstructed by well-known imperfections 
of the axioms of triangulated categories such as non-functoriality 
of the cone construction and non-uniqueness of the data supplied 
by the octahedral axiom. 

In this paper, we are able, at last, to give a fully general and rigorous
treatment of spherical functors and to prove an ideal statement
about their associated auto-equivalences. Due to increased 
prominence of spherical twist autoequivalences in studying 
derived categories of algebraic varieties our results have 
been anticipated and made use of even as the paper was being 
written. The works which already apply the results of this paper include
\cite{Addington-NewDerivedSymmetriesOfSomeHyperkaehlerVarieties},
\cite{DonovanWemyss-NoncommutativeDeformationsAndFlops},
\cite{HalpernLeistnerShipman-AutoequivalencesOfDerivedCategoriesViaGeometricInvariantTheory}, \cite{BroomheadPauksztelloPloog-DiscreteDerivedCategoriesIHomomorphismsAutoequivalencesAndTStructures},
\cite{DonovanSegal-MixedBraidGroupActionsFromDeformationsOfSurfaceSingularities}.

A previous attempt at this general treatment was made 
in \cite{Anno-SphericalFunctors}. 
Conceptually sound, it was brought low by the octahedral axiom. 
The proof of its main theorem \cite[Prop.1]{Anno-SphericalFunctors} 
contained a fatal gap which is impossible to fix 
within the axioms of triangulated categories. Nonetheless, 
it was clear that its ideas could work if we had 
an extra level of control over what the octahedral 
axiom provides us with.  

We gain this extra control by passing to differential graded (DG) categories. 
The axioms of triangulated categories were developed in 
\cite{Verdier-DesCategoriesDeriveesDesCategoriesAbeliennes}
to describe the derived categories of algebraic varieties, which 
are cohomological truncations of certain natural DG-categories. The
imperfections of these axioms can now clearly be seen as  
artefacts of the truncation. Working in the original DG-category 
provides us precisely with the layer of control that was missing. 
This allows us not only to fix the results in 
\cite{Anno-SphericalFunctors}, but to significantly improve
upon them. It allows us to do something more ---
to provide for a collection of spherical functors, as
\cite{SeidelThomas-BraidGroupActionsOnDerivedCategoriesOfCoherentSheaves}
did for spherical objects, a set of straightforward criteria sufficient 
for braid relations to occur between their twists.
For some years now the first author was well-aware of what these criteria 
should be, but proving them on the level of 
triangulated categories was hopeless. 

We first state our results in the language of triangulated categories. 
Let $A$ and $B$ be two Karoubi closed triangulated categories and 
let $s$ be an exact functor $A \rightarrow B$ which has 
left and right adjoints $l$ and $r$. 
Suppose that we can construct a preferred functorial 
exact triangle for each of the four adjunction units and 
co-units involved. Use these triangles to define 
the \em twist \rm $t$ of $s$ by the exact triangle    
\begin{align}
\label{eqn-intro-twist-exact-triangle}
sr \xrightarrow{\adjcounit} &\id_{B} \rightarrow t, 
\end{align}
the \em dual twist \rm $t'$ of $s$ by the exact triangle
\begin{align}
\label{eqn-intro-dual-twist-exact-triangle}
t' \rightarrow &\id_{B} \xrightarrow{\adjunit} sl,
\end{align}
the \em cotwist \rm $f$ of $s$ by the exact triangle
\begin{align}
\label{eqn-intro-cotwist-exact-triangle}
f \rightarrow &\id_{A} \xrightarrow{\adjunit} rs,
\end{align}
and the \em dual cotwist \rm $f'$ of $s$ by the exact triangle
\begin{align}
\label{eqn-intro-dual-cotwist-exact-triangle}
ls \xrightarrow{\adjcounit} &\id_{A} \rightarrow f'.
\end{align}
Define also two natural transformations 
\begin{align}
\label{eqn-intro-lt(-1)-to-r}
lt[-1] \xrightarrow{\eqref{eqn-intro-twist-exact-triangle}}  
lsr 
\xrightarrow{\adjcounit}
r
\\
\label{eqn-intro-r-to-fl(1)}
r \xrightarrow{\adjunit}  
rsl 
\xrightarrow{\eqref{eqn-intro-cotwist-exact-triangle}}
fl[1]. 
\end{align}

\begin{defn}
\label{defn-exact-functor-is-spherical}
The functor $s$ is \em spherical \rm if all of 
the following holds:
\begin{enumerate}
\item
\label{eqn-intro-spher-defn-twist-is-an-autoequivalence}
 $t$ and $t'$ are quasi-inverse autoequivalences of $B$
\item 
\label{eqn-intro-spher-defn-cotwist-is-an-autoequivalence}
$f$ and $f'$ are quasi-inverse autoequivalences of $A$
\item 
\label{eqn-intro-spher-defn-twist-identifies-adjoints}
$lt[-1] \xrightarrow{\eqref{eqn-intro-lt(-1)-to-r}} r$ is an
isomorphism of functors (``the twist identifies the adjoints'').  
\item
\label{eqn-intro-spher-defn-cotwist-identifies-adjoints}
$r \xrightarrow{\eqref{eqn-intro-r-to-fl(1)}} fl[1]$ is an
isomorphism of functors (``the co-twist identifies the adjoints'').  
\end{enumerate}
\end{defn}
The main obstruction is the lack of canonical functorial 
exact triangles 
\eqref{eqn-intro-twist-exact-triangle}-\eqref{eqn-intro-dual-cotwist-exact-triangle}
defining $t$, $t'$, $f$ and $f'$. 
What \cite{Anno-SphericalFunctors} tried to do was
to assume that \em some \rm functorial exact triangles as above 
exist, define $s$ to be spherical if 
\eqref{eqn-intro-spher-defn-cotwist-is-an-autoequivalence}
and \eqref{eqn-intro-spher-defn-cotwist-identifies-adjoints}
hold, and then prove that for any spherical $s$ the condition
\eqref{eqn-intro-spher-defn-twist-is-an-autoequivalence} 
also holds. 
In this paper, as explained in more detail below, 
we assume that
\begin{enumerate}
\item 
$A$ and $B$ admit DG-enhancements
\item 
$s$, $r$ and $l$ descend from DG-functors $S$, $R$ and $L$ 
between some enhancements of $A$ and $B$
\end{enumerate}
and prove that there is a canonical construction of the exact triangles
\eqref{eqn-intro-twist-exact-triangle}-\eqref{eqn-intro-dual-cotwist-exact-triangle} determined
by a certain equivalence class of $S$ such that any two of the conditions in
Defn. \ref{defn-exact-functor-is-spherical}
imply that all four of them hold and $s$ is spherical.
This is the ideal statement mentioned above. 

Let us be more precise. Let $A$ be a triangulated category. 
Traditionally, a \em DG enhancement \rm of $A$ is a DG-category $\A$ 
together with an isomorphism $H^0(\A) \simeq A$. 
A more useful notion for us is that of a \em Morita enhancement\rm, 
which is a DG-category $\A$ together with an isomorphism 
$D_c(\A) \simeq A$. Here $D_c(\A)$ is the full subcategory
of the derived category $D(\A)$ consisting of the compact objects. 
A \em Morita equivalence \rm is a DG-functor $\A \xrightarrow{f} \B$
whose induced functor $D_c(\A) \xrightarrow{\lder f^*} D_c(\B)$ is an
equivalence of categories. This is the right notion of equivalence 
for Morita enhancements. Thus we are led to work in the 
Morita homotopy category $\Morita(\DGCat)$, which is the localisation
of the category $\DGCat$ of all DG-categories by Morita
equivalences. The objects of $\Morita(\DGCat)$ should be thought
of as enhanced Karoubi closed triangulated categories with 
a fixed equivalence class of enhancements. The morphisms in 
$\Morita(\DGCat)$ are called \em Morita quasi-functors\rm. 
Each Morita quasi-functor $\A \rightarrow \B$ induces
a genuine exact functor $D_c(\A) \rightarrow D_c(\B)$. 

Let $\A$ and $\B$ be Morita enhancements of triangulated 
categories $A$ and $B$. A fundamental result of To{\"e}n 
\cite[Theorem 7.2]{Toen-TheHomotopyTheoryOfDGCategoriesAndDerivedMoritaTheory} 
implies that the Morita quasi-functors $\A \rightarrow \B$ 
are in $1$-to-$1$ correspondence with the isomorphism classes in $D(\AbimB)$
of the $\AbimB$-bimodules which are \em $\B$-perfect\rm, i.e. 
$\aM \in D_c(\B)$ for all $a \in \A$. Given  
$M \in D^{\Bperf}(\AbimB)$ the derived tensor product functor 
$$ (-) \ldertimes_\A M\colon D_c(\A) \rightarrow D_c(\B) $$ 
is the exact functor underlying the corresponding Morita
quasi-functor. Thus, we think of $D^{\Bperf}(\AbimB)$ as 
of a triangulated category structure on 
the set $\homm_{\Morita(\DGCat)}(\A, \B)$ and of morphisms
in it as morphisms of Morita quasi-functors. This packages up 
into a $2$-category structure on $\Morita(\DGCat)$ with a functor 
to the $2$-category of Karoubi closed triangulated categories. 
See Section \ref{section-DG-enhancements} for a brief survey
on DG-enhancements. 

We now describe our results. In the body of the
paper they are stated in a slightly more flexible language 
of DG-bimodules. Here we state them in the language of 
Morita quasi-functors, which gives a more intuitive picture.  
Let $\A \xrightarrow{S} \B$ be a Morita quasi-functor and 
let $A \xrightarrow{s} B$ be the underlying exact functor. 
Assume that $s$ has left and right adjoints $B \xrightarrow{l,r} A$ 
which also descend from Morita quasi-functors.
The derived $\A$- and $\B$- duals of $S$ in $D(\BbimA)$ 
are then $\A$-perfect and hence define Morita 
quasi-functors $\B \xrightarrow{L,R} \A$. In Section 
\ref{section-duals-and-adjoints} we construct
\em derived trace \rm and \em action \rm maps 
\begin{align}
SR \xrightarrow{\trace} \id_\B  \quad &\text{ and } \quad  
LS \xrightarrow{\trace} \id_\A  \\
\id_\B \xrightarrow{\action} SL \quad &\text{ and } \quad 
\id_\A \xrightarrow{\action} RS.  
\end{align}
and prove that the exact functors underlying $L$ and $R$ are precisely
$l$ and $r$ and that the derived trace and action 
maps above induce the units and co-units of the adjunctions of $s$,
$l$ and $r$. Then, working in the DG-enhancements, we construct natural 
exact triangles of Morita quasi-functors 
\begin{align}
\label{eqn-intro-twist-exact-triangle-DG}
SR \xrightarrow{\trace} &\id_\B \rightarrow T, \\
\label{eqn-intro-dual-twist-exact-triangle-DG}
T' \rightarrow & \id_\B \xrightarrow{\action} SL, \\
\label{eqn-intro-cotwist-exact-triangle-DG}
F \rightarrow & \id_\A \xrightarrow{\action} RS, \\
\label{eqn-intro-dual-cotwist-exact-triangle-DG}
LS \xrightarrow{\trace} &\id_\A \rightarrow F',
\end{align}
which define the \em twist \rm $T$, the \em dual twist \rm $T'$, 
the \em co-twist \rm $F$ and the \em dual co-twist \rm $F'$ of $S$. 
Thus we obtain a natural choice of functorial exact triangles 
\eqref{eqn-intro-twist-exact-triangle}-\eqref{eqn-intro-dual-cotwist-exact-triangle}
defining $t$, $t'$, $f$ and $f'$. We then prove that $t'$ and $f'$ 
are left adjoint to $t$ and $f$, respectively. All the above 
constructions are readily seen to be Morita-invariant, i.e. 
they are preserved if we replace $\A$ or $\B$ by 
a Morita-equivalent DG-category. Hence they only depend 
on Morita equivalence classes of $\A$ and $\B$ 
and on $S \in \homm_{\Morita(\DGCat)}(\A, \B)$. 

The following is the main result of this paper:
\begin{theorem}[see Theorem \ref{theorem-main-theorem}]
\label{theorem-main-theorem-intro}
If any two of the following conditions hold:
\begin{enumerate}
\item 
\label{item-intro-main-theorem-the-twist-is-an-equivalence}
$t$ is an autoequivalence of $B$
(``the twist is an equivalence'').  
\item 
\label{item-intro-main-theorem-the-cotwist-is-an-equivalence}
$f$ is an equivalence of $A$
(``the cotwist is an equivalence'').  
\item 
\label{item-intro-main-theorem-the-twist-identifies-adjoints}
$lt[-1] \xrightarrow{\eqref{eqn-lt(-1)-to-r}} r$ is an
isomorphism of functors (``the twist identifies the adjoints'').  
\item 
\label{item-intro-main-theorem-the-cotwist-identifies-adjoints}
$r \xrightarrow{\eqref{eqn-r-to-fl(1)}} fl[1]$ is an
isomorphism of functors (``the cotwist identifies the adjoints'').  
\end{enumerate}
then all four hold and $S$ is said 
to be a spherical quasi-functor. 
\end{theorem}

Finally, we give the braiding criteria for spherical quasi-functors. 
These have a natural interpretation in geometrical context that
is the subject of a future paper
\cite{AnnoLogvinenko-OnBraidingCriteriaForSphericalTwistsByFlatFibrations}. 
An example of these criteria being satisfied can be seen in 
a construction by Khovanov and Thomas in
\cite{KhovanovThomas-BraidCobordismsTriangulatedCategoriesAndFlagVarieties}. 
 
Let $A_1, \dots, A_n, B$ be triangulated categories 
with Morita enhancements $\A_1, \dots, \A_n, \B$. 
Let $\A_i \xrightarrow{S_i} \B$ be spherical Morita quasi-functors. 
For any $i\ne j$ trace maps $S_i R_i \xrightarrow{\trace} \id_\B$ 
and $S_j R_j \xrightarrow{\trace} \id_\B$ define a map
\begin{align}
\label{eqn-intro-maps-SiRiSjRj-SiRi-plus-SjRj}
S_i R_i S_j R_j \xrightarrow{S_iR_i \trace \oplus \trace S_j R_j } 
S_i R_i \oplus S_j R_j.
\end{align}

Next, for any $i\ne j$ define a Morita quasi-functor 
$\A_i \xrightarrow{O_i} \A_i$ by
\begin{align}
\label{eqn-intro-definition-of-O_i}
O_{ij}=F_i\cone \left( 
L_iS_jR_jS_i \xrightarrow{\trace\circ(L_i \trace S_i)} \id_{\A_i} \right). 
\end{align}
As $S_i$ is spherical we have $R_i[-1] \simeq F_i L_i$, so 
$S_i R_i \xrightarrow{\trace} \id_\B$ and 
$S_j R_j \xrightarrow{\trace} \id_\B$ define (cf. 
Section \ref{section-criterion-for-braiding}) a map
\begin{align}
\label{eqn-intro-maps-SiOijRi-SiRiSjRj-plus-SjRjSiRi}
S_i O_{ij} R_i \rightarrow S_i R_i S_j R_j \oplus S_j R_j S_i R_i.
\end{align}

\begin{theorem}[Theorems 
\ref{theorem-commutation-criterion}-\ref{theorem-braiding-criterion}]
\label{theorem-intro-braid-group-action}
Suppose that for all $i,j \in 1, \dots, n$ the following holds: 
\begin{enumerate}
\item If $|i - j| > 1$ there exists an isomorphism
$$ S_i R_i S_j R_j \simeq S_j R_j S_i R_i $$
which commutes with the maps 
\eqref{eqn-intro-maps-SiRiSjRj-SiRi-plus-SjRj}. 
\item If $|i - j| = 1$, there exists an isomorphism 
$$ S_i O_{ij} R_i \simeq S_j O_{ji} R_j $$
which commutes with the maps
\eqref{eqn-intro-maps-SiOijRi-SiRiSjRj-plus-SjRjSiRi}
\end{enumerate}
Then the twists $T_1, \dots, T_n$ generate a categorical 
action of the braid group $\B_n$ on $B$. 
\end{theorem}

Finally, we interpret the above in the context of algebraic
geometry. Let $Z$ and $X$ be separated schemes of finite type
over $k$. Let $D_{qc}(Z)$ and $D_{qc}(X)$ be the 
derived categories of quasi-coherent sheaves and
$D(Z)$ and $D(X)$ be the bounded derived categories of 
coherent sheaves on $Z$ and $X$. Let $\A$ and $\B$ be 
the standard DG-enhancements of $D(Z)$ and $D(X)$. These
are given by the 
DG-categories of $h$-injective complexes of sheaves on $Z$ and $X$, 
respectively. 
In Example \ref{example-quasi-functors-between-enhancements-of-DBCoh}
we prove an analogue for the bounded coherent derived categories of
the famous result of To{\"e}n 
\cite[Theorem 8.9]{Toen-TheHomotopyTheoryOfDGCategoriesAndDerivedMoritaTheory}
for the unbounded quasi-coherent ones. We prove that
the exact functors $D(Z) \rightarrow D(X)$ which descend from 
the Morita quasi-functors $\A \rightarrow \B$ are precisely the
Fourier-Mukai transforms. Given an object $E \in D(Z \times X)$ 
the Fourier-Mukai transform $\Phi_E$ is apriori 
a functor $D_{qc}(Z) \rightarrow D_{qc}(X)$. 
In Example \ref{example-quasi-functors-between-enhancements-of-DBCoh}
we identify $\homm_{\Morita(\DGCat)}(\A, \B)$ with 
the full subcategory of $D(Z \times X)$ consisting 
the objects $E$ such that $\Phi_E$ restricts 
to $D(Z) \rightarrow D(X)$. Under 
this identification, each Morita quasi-functor $\A \xrightarrow{S} \B$ 
goes to such object $E \in D(Z \times X)$ 
that $D(Z) \xrightarrow{\Phi_E} D(X)$ is
the exact functor $s$ underlying $S$. 

The above results for Morita quasi-functors can then all be interpreted
for the Fourier-Mukai transforms. 
Let $E \in D(Z \times X)$ be such that $\Phi_E$ restricts to 
a functor $D(Z) \xrightarrow{s} D(X)$ and this restriction
has a left adjoint which is also a Fourier-Mukai transform. 
E.g. it is sufficient to assume that $E$ is 
proper over $Z$ and $X$ and perfect over $Z$ and $X$. Our
results for Morita quasi-functors provide natural constructions
on the level of Fourier-Mukai kernels of the right and left adjoints 
$r$ and $l$ and of all four adjunctions units and co-units involved.
We conjecture that these coincide with the explicit formulas proved
independently 
in \cite{AnnoLogvinenko-OnTakingTwistsOfFourierMukaiFunctors} and 
\cite{AnnoLogvinenko-OrthogonallySphericalObjectsAndSphericalFibrations}.  
Regardless of whether this holds or not, 
the functorial exact triangles
\eqref{eqn-intro-twist-exact-triangle}-\eqref{eqn-intro-dual-cotwist-exact-triangle}
defining the twists and co-twists $t$, $t'$, $f$ and $f'$ are well-defined
and depend only on $E \in D(Z \times X)$. We say that $E$ is \em
spherical over $Z$ \rm if the four conditions of 
the Definition \ref{defn-exact-functor-is-spherical} are satisfied. 
Our main theorem then applies to show that, in fact, it suffices
to only verify any two of these four conditions. The braiding criteria
above translate similarly to the language of Fourier-Mukai kernels. 
It is worth noting that if we set $Z = \spec k$ then the natural
isomorphism $Z \times X \simeq X$ identifies 
$D(Z \times X)$ with $D(X)$
and our results imply immediately the results in 
\cite{SeidelThomas-BraidGroupActionsOnDerivedCategoriesOfCoherentSheaves}. 

Finally, we also describe in 
Section \ref{section-applications-to-algebraic-geometry} 
a variation on all of the above. It uses a slightly 
different enhancement framework which allows 
one to work with the unbounded derived categories $D_{qc}(Z)$ and $D_{qc}(X)$. 
The penalty is a strong smoothness condition. We can only work with  
$E \in D_{qc}(X \times Y)$ such that $\Phi_E$ has a left adjoint 
which is also a Fourier-Mukai transform and they both 
take compact objects to compact objects.  

About the structure of this paper: in Section 
\ref{section-DG-categories-and-DGmodules-and-bimodules}
we give an overview of the facts we need on DG-categories and DG-modules over
them. In Section \ref{section-duals-and-adjoints} we define the 
dualizing functors for DG-modules and DG-bimodules. 
We then construct and study \em trace \rm and \em action \rm maps and 
show them to be units and co-units of homotopy adjunctions between 
an $\AbimB$-bimodule $M$ and its $\A$- and $\B$-duals $M^{\A}$ and $M^{\B}$. 
In Section \ref{section-subsection-twisted-complexes} we give an
overview on twisted complexes over a DG-category and prove 
explicit formulas for taking a tensor product and for dualizing
on the level of twisted complexes. Section 
\ref{section-pre-triangulated-categories} summarizes the facts we need
about pre-triangulated categories. 
In Section \ref{section-twisted-cubes} we develop 
a theory of \em twisted cubes\rm, which acts as
a ``higher'' octahedral axiom for the world of 
pretriangulated categories. In Section \ref{section-DG-enhancements}
we explain the framework of DG-enhancements of triangulated
categories and its applications to algebraic geometry. In 
Section \ref{section-spherical-bimodules} we 
constructs twists and co-twists of a DG-bimodule, define a notion
of a spherical DG-bimodule and prove our main theorem on 
the level of DG-bimodules. In
Section \ref{section-applications-to-algebraic-geometry} we 
interpret this for Fourier-Mukai transforms between the derived
categories of algebraic varieties via the framework introduced in 
Section \ref{section-DG-enhancements}. 
In Section \ref{section-braiding-criteria} we state and prove the
braiding criteria for spherical DG-bimodules. Finally, the Appendix
\ref{section-on-homotopy-equivalences-of-twisted-complexes} contains
some technical results we need in Section \ref{section-braiding-criteria}  
on constructing homotopy equivalences between 
twisted complexes. There the authors have to resort
to using $A_\infty$-categories, $A_\infty$-functors 
and the interpretation of DG quasi-functors as 
strictly unital $A_\infty$-functors between the corresponding
DG-categories. It is something they quite happily avoided doing 
throughout the rest of the paper. 

\em Acknowledgements: \rm We would like to thank Alexander Efimov, 
Alexei Bondal, Valery Lunts and Alexander Kuznetsov for useful
discussions in the course of writing this paper. We would like to
thank Pawel Sosna for a careful readthrough of the paper. 
The first author would like to thank the University of Pittsburgh 
for providing a
perfect environment for completing the work on this paper.  
The second author would like to offer similar thanks to 
Cardiff University. Both authors would like 
to thank the annual Lunts dacha seminar for being a helpful 
and stimulating environment for research. 

\section{Preliminaries}
\label{section-preliminaries}

Some proofs in this paper rely on explicit computations where 
matching up the signs becomes important. As there are different sign 
conventions present in the literature for the material in this
section, we make our choices explicit at the cost of restating some 
very well-known definitions. We aim to enable our reader to verify 
all the computations which are ``left to the reader''. 

\em Notation: \rm Throughout the paper all schemes are defined and 
all DG-categories are considered over the same base field we denote
by $k$. 

Let $X$ be a scheme. We denote by $D_{\text{qc}}(X)$, resp. $D(X)$, 
the full subcategory of the derived category of
$\mathcal{O}_X$-$\modd$ consisting of complexes with 
quasi-coherent, resp. bounded and coherent, cohomology.
%

\subsection{DG categories, modules and bimodules}
\label{section-DG-categories-and-DGmodules-and-bimodules}
Throughout this section $k$ is a commutative ring. 
\subsubsection{DG categories}
Let $E$ and $F$ be complexes of $k$-modules. 
Define $E \otimes F$ to be the complex of $k$-modules 
\begin{align}
\nonumber
(E \otimes F)_n  = \bigoplus_{i + j = n} E_i \otimes F_j  \\
d(e \otimes f) = de \otimes f + (-1)^{\deg(e)}e \otimes df. 
\end{align}
We have the standard sign-twisting 
isomorphism $E \otimes F \xrightarrow{\sim} F \otimes E$ given by 
\begin{align}
\label{eqn-E-otimes-F-iso-F-otimes-E}
e \otimes f \mapsto (-1)^{\deg(e)\deg(f)} f \otimes e.
\end{align}
Define $\homm_k(E,F)$ to be the complex of $k$-modules 
\begin{align}
\label{eqn-hom-complex-of-dg-k-modules}
\nonumber
\homm^n_k(E,F) = \bigoplus_{j-i=n} \homm_k(E_i,F_j)  \\
df = d_F \circ f - (-1)^{\deg(f)} f \circ d_E. 
\end{align}

A DG-category over $k$ is 
a category $\A$ whose morphism spaces $\homm_{\A}(a,b)$ 
are complexes of $k$-modules and 
whose composition maps 
$$ \homm_{\A}(b,c) \otimes \homm_{\A}(a,b) 
\rightarrow \homm_{\A}(a,c) $$ are 
closed degree $0$ maps of complexes of $k$-modules. 
The homotopy category $H^0(\A)$ has same objects 
as $\A$ and its morphisms spaces are $0$-th cohomologies of 
their counterparts in $\A$. 
Let $\modk$ be the DG-category of complexes of $k$-modules 
with morphism spaces defined by \eqref{eqn-hom-complex-of-dg-k-modules}
and the composition $(f \circ g)(s) = f(g(s))$. 
See \cite[\S1.1-1.2]{Keller-DerivingDGCategories} or 
\cite{Toen-LecturesOnDGCategories} for details.  

Given a DG-category $\A$ denote by $\Aopp$ the \em opposite
DG-category \rm of $\A$. Its objects are the same as those of $\A$ and 
for all $a,b \in \Aopp$ we have $\homm_{\Aopp}(a,b) = \homm_{\A}(b,a)$. 
The composition is defined by
composing the sign-twisting isomorphism 
$\eqref{eqn-E-otimes-F-iso-F-otimes-E}$ with the composition map of $\A$. 
In other words, we set
$$ \beta \circ_{\Aopp} \alpha = (-1)^{\deg(\alpha)\deg(\beta)} \alpha
\circ_\A \beta $$ 
for all $\alpha \in \homm_{\Aopp}(a,b)$, $\beta \in \homm_{\Aopp}(b,c)$. 

Let $\A$ and $\B$ be two DG-categories. A DG-functor 
$\A \rightarrow B$ is a $k$-linear functor which preserves 
the grading and the differential on morphisms.
Wherever the context permits we omit ``DG-'' and simply say 
``functor''. A degree $n$ natural transformation of 
DG-functors $\Phi \xrightarrow{t} \Psi$ 
is a collection 
$$\bigl\{ t(a) \in \homm^n_\B(\Phi(a), \Psi(a)) \bigr\}_{a \in \A}$$
where
$t(a') \circ \Phi(\alpha) = (-1)^{nm} \Psi(\alpha) \circ t(a)$ for every
$\alpha \in \homm^m_\A(a,a')$. Define the DG-category $\DGFun(\A,\B)$
as follows. Its objects are DG-functors $\A \rightarrow \B$. 
Its morphism complexes $\homm_{\DGFun(\A,\B)}^\bullet(\Phi, \Psi)$
consist of natural transformations $\Phi \xrightarrow{t} \Psi$
graded by degree and with differentials defined levelwise 
by those of $\B$, i.e. $dt(a) = d_\B(t(a))$ for each $a \in \A$. 
The composition maps are also defined levelwise by those of $\B$. 

We denote by $\DGCat$ the category whose objects are all small
DG-categories over $k$ and whose morphisms are DG-functors between
them. 

\subsubsection{Closed symmetrical monoidal structure on $\DGCat$}

Let $\A$ and $\B$ be DG-categories. We define $\A \otimes_k \B$ to be 
the DG-category whose objects are pairs $(a,b)$ with $a \in \A, b \in \B$,
whose morphism complexes are
$$\homm_{\A \otimes \B}\left(a \otimes b, a' \otimes b')\right) = 
\homm_{\A}(a \otimes a') \otimes \homm_{ \B}(b\otimes b')$$ 
and whose composition is defined by
$$ (\alpha' \otimes \beta') \circ (\alpha \otimes \beta)
= (-1)^{\deg(\beta')\deg(\alpha)}
(\alpha' \circ \alpha) \otimes (\beta' \circ \beta). $$ 
This construction is bifunctorial in $\A$ and $\B$ and 
defines a monoidal structure on $\DGCat$ whose unit is $k$.

The monoidal structure $(\otimes_k, k)$ is symmetric 
via a natural isomorphism 
\begin{align}
\label{eqn-B-x-A-isomorphism-to-A-x-B}
\B \otimes_k \A \xrightarrow{\sim} \A \otimes_k \B
\end{align}
defined on objects by $ b \otimes a \mapsto a \otimes b$
and on morphisms by 
$\beta \otimes \alpha \mapsto (-1)^{\deg(\alpha)\deg(\beta)} \alpha
\otimes \beta$. 

The monoidal structure $(\otimes_k, k)$ is, moreover, closed with
the internal $\homm$ given by $\DGFun(-,-)$. Explicitly, 
for any DG-categories $\A, \B, \C$ we have a natural isomorphism 
\begin{align}
\label{eqn-left-tensor-and-DGFun-adjunction-enhanced}
\DGFun\left(\A \otimes_k \B, \C\right) \xrightarrow{\sim}    
\DGFun\left(\B, \DGFun\left(\A, \C\right) \right) 
\end{align}
which takes any $\A \otimes_k \B \xrightarrow{\Phi} \C$ to 
the functor
$$ \forall\; b \in \B \quad b \mapsto \Phi(- \otimes b), 
\quad 
\forall\;\beta \in \homm_{\B}(b,b')
\quad \beta \mapsto \Phi(\id \otimes \beta)$$
and any $\Phi \xrightarrow{f} \Psi$ to the natural transformation  
$\bigl\{ \Phi(- \otimes b) \xrightarrow{f} \Psi(- \otimes b) \bigr\}_{b
\in \B}$.
The object set of $\DGFun(-,-)$ is the set $\homm_{\DGCat}(-,-)$, 
so the isomorphism \eqref{eqn-left-tensor-and-DGFun-adjunction-enhanced} 
induces an adjunction isomorphism between $(-) \otimes_k$ and $
\DGFun(\A,-)$ which makes $\DGFun$ the internal $\homm$ in
$(\DGCat, \otimes_k, k)$. 

For any two DG-categories $\A$ and $\B$ we have tautological
categorical isomorphisms 
\begin{align}
(\A \otimes_k \B)^\opp &\simeq \Aopp \otimes_k \B^\opp, \\
\DGFun(\A,\B)^\opp &\simeq \DGFun(\Aopp, \Bopp). 
\end{align}
The former isomorphism sends any pair of objects or morphisms in 
$\A \otimes_k \B$ to themselves considered as elements of 
$\Aopp \otimes_k \B^\opp$. The latter sends any functor or a natural
transformation in $\DGFun(\A,\B)$ to itself considered as an element
of $\DGFun(\Aopp, \Bopp)$. 

Finally, for any four DG-categories $\A,\B, \C$ and $\D$ we have 
the simultaneous evaluation functor
\begin{align}
\label{eqn-simulatenous-evaluation-functor}
\DGFun(\A,\C) \otimes_k \DGFun(\B,\D) \rightarrow 
\DGFun(\A \otimes_k \B, \C \otimes_k \D)
\end{align}
which sends any pair of functors $\A \rightarrow \C$ and $\B \rightarrow \D$
to the functor of simultaneously evaluating them on any pair of
objects or morphisms in $\A \otimes_k \B$. Similarly for 
natural transformations of such pairs of functors. Note that 
there are no sign twists involved. 

\subsubsection{DG modules}
\label{section-DG-modules}

A (right) $\A$-module is a functor from $\Aopp$ to $\modk$. 
Denote by $\modA$ the DG-category $\DGFun(\Aopp, \modk)$. 
For the reasons of brevity and to mimic the notation used for DG-algebras, 
for any two $E,F \in \modA$ we write $\homm_\A(E,F)$ for 
$\homm_{\modA}(E,F)$.
The category $H^0(\modA)$ admits natural structure of a triangulated 
category which 
is defined levelwise by the usual triangulated structure 
on $H^0(\modk)$, cf. \cite[\S 2.2]{Keller-DerivingDGCategories}. 

For any $E \in \modA$ and $a \in \A$ we write $E_{a}$ for the complex 
of $k$-modules $E(a)$. We write $v \in E$ if $v \in E_{a}$ for some 
$a \in \A$. The Yoneda embedding $\A \hookrightarrow \modA$ is 
the fully faithful functor defined on the objects by
$$ a\mapsto \homm_{\A}(-, a) \quad\quad  \forall\; a \in \A$$
and on the morphisms by composition. 
For each $a \in \A$ denote by $\aA$ its image under the Yoneda embedding, 
these are the \em representable \rm objects of $\modA$. 
Note, that for all $a,b \in \A$ we have $\aA_b = \homm_{\A}(b,a)$. 
For any $E \in \modA$ trivially $\homm_{\A}(\aA, E) = E_a $. For 
each $s \in E_{a}$ and $\alpha \in \aA_{b}$ we write $s \cdot \alpha$ 
for the element $(-1)^{\deg(s)\deg(\alpha)} E(\alpha)(s) \in E_{b}$. 
We have
$$ (s \cdot \alpha) \cdot \beta = s \cdot (\alpha \circ \beta). $$
In other words, we can think of the data defining an $\A$-module $E$
as of collection of fibers $E_a \in C(k)$ for each $a \in \A$
with a right action of (the $\homm$-spaces of) $\A$ on them,
such that $\aA_b$ acts on $E_a$ and maps it to $E_b$. Similarly,  
a morphism of right $\A$-modules $E \xrightarrow{t} F$ can be thought of 
as a collection of maps $E_a \xrightarrow{t} F_a$ in $\modk$ which 
commute with the $\A$-action: 
$t(s\cdot\alpha) = t(s) \cdot \alpha$ for any $s \in E_a$ and 
$\alpha \in \aA$. 

A left $\A$-module is a right $\Aopp$-module, i.e. 
a functor $\A \rightarrow \modk$. To facilitate the treatment of
bimodules, it is often useful to treat right $\Aopp$-modules as left 
$\A$-modules and employ for them the following notation. 
For any $F \in \modAopp$ and $a \in \A$ we write $\aF$ 
(instead of $\Fa$) for the complex $F(a)$.
For each $a \in \A$ write $\Aa$ for the image of $a$ under 
the Yoneda embedding of $\Aopp$, i.e. for the 
functor $\homm_{\A}(a,-)$. Set 
$\alpha \cdot s = F(\alpha)(s)$ for each $s \in \aF$ 
and $\alpha \in \Aa$, it is a left action of $\A$ on $F$. 
A morphism of left $\A$-modules $E \xrightarrow{t} F$
can be thought of as a collection of maps 
$\aE \xrightarrow{t} \aF$
which skew-commute with the $\A$-action:
$t(\alpha \cdot s) = (-1)^{\deg(t)\deg(\alpha)} \alpha \cdot t(s)$
for any $s \in \aE$ and $\alpha \in \aA$.

\subsubsection{Tensor and Hom}
Let $\A$ be a DG-category and let 
$E$ and $F$ be a right and a left $\A$-module. Define 
the tensor product $E \otimes_\A F \in \modk$ to be the quotient of 
$\bigoplus_{a \in A} E_a \otimes \leftidx{_a}{F} \; \in \; \modk $
by the $\A$-action relations
\begin{align}
\label{eqn-A-action-relations-for-tensor-product}
(s \cdot \alpha) \otimes t = s \otimes (\alpha \cdot t)
\quad\quad\quad
\forall\; \alpha \in \bA_{a},\; s \in \Eb,\; t \in \aF.
\end{align}

We extend this to the functor
\begin{align}
\label{eqn-modk-valued-tensor-product}
(-)\otimes_{\A}(-)\colon \quad \modA \otimes_k \modAopp \rightarrow \modk
\end{align}
by defining the tensor product $\lambda \otimes \mu$ of 
two maps $E \xrightarrow{\lambda} E'$ and $F \xrightarrow{\mu} F'$
to be the following. The map  
\begin{align*}
\bigoplus_{a \in \A} \Ea \otimes \aF 
& \quad \rightarrow \quad
\bigoplus_{a \in \A} \Ea' \otimes \aF' 
\\
e \otimes f 
& \quad\mapsto\quad
(-1)^{\deg(e) \deg(\mu)} \lambda(e) \otimes \mu(f) 
\end{align*}
preserves $\A$-action relations and we define
$\lambda \otimes \mu$ to be the induced map
$E \otimes_\A F \rightarrow E' \otimes_\A F'$.  

Similarly, we define the functor
\begin{align}
\label{eqn-modk-valued-Hom}
\homm_\A\left(-,-\right)\colon \quad \modA \otimes_k (\modA)^\opp \rightarrow \modk
\end{align}
on objects by $(E,F) \mapsto \homm_\A(F,E)$ and on 
morphisms as follows. For any pair of maps $E \xrightarrow{\lambda} E'$ 
and $F' \xrightarrow{\mu} F$ in $\modA$ we define the composition map 
\begin{align*}
\homm_\A(F,E) & \xrightarrow{\lambda \circ (-) \circ \mu}
\homm_\A(F',E') \\
\alpha &\mapsto (-1)^{\deg(\alpha) \deg(\mu)} \lambda \circ \alpha \circ \mu.
\end{align*}

\subsubsection{DG bimodules}

An $\A\text{-}\B$ bimodule is an $\Aopp \otimes \B$-module. 
We write $\AmodB$ for 
$$ \DGFun(\A \otimes \Bopp, \modk)
\simeq \DGFun(\A, \modB) \simeq \DGFun(\Bopp, \modAopp)$$ considered
as the DG category of all $\AbimB$-bimodules. 
Let $M \in \AmodB $. For any $a \in \A, b \in \B$ we write $\aMb$ for 
$M(a,b) \in \modk$, write 
$\aM$ for the $\B$-module $M(a,-)$, and write $M_{b}$
for the $\Aopp$-module $M(-,b)$. The functor $\mathcal{A}
\rightarrow \modB$ which corresponds to $M$ maps
$a$ to $\aM$. We can extend it to the functor
$$(-) \otimes_{\mathcal{A}} M\colon \modA \rightarrow \modB,$$
where for any $E \in \modA$ and $b \in \B$
we set $(E \otimes_{\A} M)_b = E \otimes_{\A} M_b$ and
have $\B$ act via $M_b$. We can further extend this 
to a lift of the tensor bifunctor \eqref{eqn-modk-valued-tensor-product}
from $\modAopp$ to $\AmodB$ in the second argument. 
This admits a more general description. 

Let $\A$, $\B$ and $\C$ be any DG-categories. Since 
$\CmodA$ and $\CmodB$ are equivalent to $\DGFun(\C, \modA)$ and
$\DGFun(\C, \modB)$, the composition functor
$$ \DGFun(\modA,\modB) \otimes_k \DGFun(\C,\modA)
\xrightarrow{(-)\circ(-)}  
\DGFun(\C,\modB)$$
induces via the adjunction the functor
\begin{align}
\label{eqn-fiberwise-lift-of-a-functor}
\DGFun\left(\modA, \modB\right) \rightarrow 
\DGFun\left(\CmodA, \CmodB\right)
\end{align}
best described as the functor of ``defining fiberwise over $\C$''. It
takes a functor $\modA \xrightarrow{\Phi} \modB$ and defines a functor
$\CmodA \rightarrow \CmodB$ which takes any $\CbimA$ bimodule $E$
to the $\CbimB$ bimodule whose fiber over each $c \in \C$ is
$\Phi(\leftidx{_c}{E})$, and similarly for morphisms. 

We can apply a similar procedure to the functors whose domain is a tensor
product of module categories via the simultaneous evaluation functor 
\eqref{eqn-simulatenous-evaluation-functor}. We define the functor
\begin{align}
\label{eqn-bimodule-valued-tensor-product} 
(-) \otimes_\A (-) \colon\quad \CmodA \otimes_k \AmodB \rightarrow \CmodB
\end{align}
as the composition 
\begin{small}
\begin{align*}
\xymatrix{
\DGFun(C,\modA) \otimes_k \DGFun(\Bopp, \modAopp) 
\ar[d]^{\eqref{eqn-simulatenous-evaluation-functor}} \\
\DGFun(C \otimes_k \Bopp, \modA \otimes_k \modAopp)
\ar[d]^{\eqref{eqn-modk-valued-tensor-product} \circ (-)}\\
\DGFun(C \otimes_k \Bopp, \modk)
} 
\end{align*}
\end{small}
Similarly, we use $\modk$ valued $\homm$ functor 
$\eqref{eqn-modk-valued-Hom}$ to define the functors
\begin{align}
\label{eqn-bimodule-hom-as-right-modules} 
&\homm_{\B}(-,-) \colon\quad \CmodB \otimes_k (\AmodB)^\opp \rightarrow \CmodA \\
\label{eqn-bimodule-hom-as-left-modules} 
&\homm_{\A}(-,-) \colon\quad \AmodC \otimes_k (\AmodB)^\opp \rightarrow \BmodC.
\end{align}

For any $\AbimB$-bimodule $M$ we have the usual Tensor-Hom
adjunction: for any DG-category $\C$ 
$$ (-) \otimes_\A M\colon \CmodA \rightarrow \CmodB $$
is left adjoint to 
$$ \homm_\B(M,-) \colon \CmodB \rightarrow \CmodA. $$
Its adjunction co-unit 
\begin{align}
\label{eqn-adjunction-counit-for-M} 
\homm_\B\bigl(M,-\bigr) \otimes_\A M \rightarrow \id 
\end{align}
is given by the composition map 
$$ \homm_{\B}(M, -) \otimes_{\A} \homm_{\B}(\B, M) \rightarrow 
\homm_{\B}(\B, -), $$
and its adjunction unit 
\begin{align}
\label{eqn-adjunction-unit-for-M} 
\id \rightarrow \homm_\B\bigl(M,(-) \otimes_\A M\bigr) 
\end{align}
is defined by 
$$ s \mapsto 
\bigl(\forall\; t \in \aM, \quad t \mapsto s \otimes t \bigr) 
\quad \quad 
\forall\; c \in \C,\; a \in \A,\; s \in \leftidx{_c}{(-)}_a.
$$
Similarly,
$$ M \otimes_{\B} (-) \colon \BmodC \rightarrow \AmodC $$
is left adjoint to 
$$ \homm_\Aopp(M,-) \colon \AmodC \rightarrow \BmodC $$
with analogous adjunction unit
\begin{align}
\label{eqn-adjunction-unit-for-M-otimes-(-)} 
\id \rightarrow \homm_\Aopp\bigl(M,M\otimes_\B(-)\bigr) 
\end{align}
and counit
\begin{align}
\label{eqn-adjunction-counit-for-M-otimes-(-)} 
M \otimes_\B \homm_\Aopp\bigl(M,-\bigr) \rightarrow \id. 
\end{align}

\subsubsection{Derived category}

A module $C \in \modA$ is \em acyclic \rm if for each $a \in \A$
the complex of $k$-modules $C_a$ is acyclic. A module $P \in \modA$ is 
\em $h$-projective \rm if $\homm_{H^0(\modA)}(P,C) = 0$ 
for every acyclic $C \in \modA$. Denote by $\hproj(\A)$ the
corresponding full subcategory of $\modA$. 
A morphism $E \rightarrow F$ of $\A$-modules is a \em quasi-isomorphism \rm 
if for each $a \in \A$ the induced morphism $E_a \rightarrow F_a$ 
is a quasi-isomorphism. 
Let $\mathcal{M}' \subset \modA$ be a full DG-subcategory, 
then a left (resp.\ right) 
resolution of $E \in \A$ by $E' \in \mathcal{M}'$ 
is a quasi-isomorphism $E' \rightarrow E$ (resp.\ $E \rightarrow E'$). 
The \em derived category $D(\A)$ \rm is the localisation of $H^0(\modA)$ by 
the class of all quasi-isomorphisms. 
It can be understood explicitly as follows. By definition of
acyclicity $D(\A)$ is the Verdier quotient of $H^0(\modA)$ by 
$H^0(\acyc(\A))$. By definition of $h$-projectivity
$H^0(\hproj(\A))$ is left orthogonal to $H^0(\acyc(\A))$. Since 
left resolutions by $h$-projectives exist in $\modA$ we have in fact 
a semi-orthogonal decomposition 
$$H^0(\modA) = \bigl<\;H^0\left(\acyc(\A)\right),\; H^0\left(\hproj(\A)\right)\;\bigr>.$$
This canonically identifies $D(\A) = H^0(\modA)/H^0(\acyc(\A))$ with 
$H^0(\hproj(\A))$. In practice, we can use for resolutions
a smaller full subcategory $\sfA$ of the \em semifree \rm 
modules in $\modA$. These are the modules $E \in \modA$ which 
admit an exhaustive filtration 
$0 = F_0 \subseteq F_1 \subseteq F_{2} \subseteq \dots \subseteq E$ whose 
quotients $F_{i}/F_{i-1}$ are direct sums of shifts of representable modules. 
Any semifree module is $h$-projective and any $\A$-module can be resolved 
by a semifree module \cite[\S C.8]{Drinfeld-DGQuotientsOfDGCategories}.
When $k$ is a field, we have a functorial h-projective resolution of
$\A$-modules provided by the \em bar-resolution \rm $\bar{\A}$ of the diagonal
$\AbimA$-bimodule $\A$ \cite[\S 6.6]{Keller-DerivingDGCategories}. 

Another way to understand $D(\A)$ is via either of the two natural model 
category structures induced on $\modA$ from $\modk$. In particular, 
in the \em projective \rm model category structure on $\modk$ the
weak equivalences and the fibrations are the quasi-isomorphisms and 
the termwise surjections of complexes. In the corresponding 
model category structure on $\modA$ we define the equivalences and 
the fibrations levelwise in $\modk$, i.e. a morphism $A \rightarrow B$ 
is an equivalence (resp. fibration) if for every $a \in \A$ morphism 
$A_a \rightarrow B_a$ is an equivalence (resp. fibration) in 
$\modk$ \cite[\S 3]{Toen-TheHomotopyTheoryOfDGCategoriesAndDerivedMoritaTheory}. It follows that every $\A$-module is fibrant, while the cofibrant modules
are precisely the direct summands of semifree modules. 
We denote the full subcategory of $\modA$ consisting of cofibrant 
objects by $\intmod(\A)$. It is the Karoubi completion of
$\sfA$. 

Summarizing, we have a chain of full subcategories
\begin{align}
\sfA \hookrightarrow \intmod(\A) \hookrightarrow \hproj(\A) 
\end{align}
of $\modA$ which, after applying $H^0$  becomes a chain of equivalent
full triangulated subcategories  
\begin{align}
H^0(\sfA) 
\xrightarrow{\sim}
H^0(\intmod(\A))
\xrightarrow{\sim}
H^0(\hproj(\A))
\end{align}
of $H^0(\modA)$. The natural functor $H^0(\modA) \rightarrow D(\A)$
induces an equivalence of these with $D(\A)$. In the 
language of Section \ref{section-DG-enhancements},  
$\sfA$, $\intmod(\A)$ and $\hproj(\A)$ are 
quasi-equivalent DG-enhancements of $D(\A)$.  

An $\A$-module $E$ is \em quasi-representable \rm if it is
quasi-isomorphic to a representable module. 
We denote by $\qrepA$ and $\qrhprA$ 
the corresponding full subcategories of $\modA$ and of $\hprojA$. 
A semi-free $\A$-module $E$ 
is \em finitely-generated \rm if the filtration 
$F_0 \subseteq F_1 \subseteq F_{2} \subseteq \dots \subset E$ can be taken 
to be finite with quotients $F_i/F_{i-1}$ finite direct sums of shifts 
of representables. Denote by $\sffgA$ the corresponding full 
subcategory of $\sfA$. Its homotopy category $H^0(\sffgA)$ is 
the \em triangulated hull \rm of $H^0(\A)$ in $H^0(\modA)$, i.e. 
it is the smallest full triangulated subcategory of $H^0(\modA)$ 
containing $H^0(\A)$. An $\A$-module $E$ is \em perfect \rm if
its image in $D(\A)$ lies in the full subcategory $D_c(\A)$
of \em compact \rm objects, i.e. if $\homm_{D(\A)}\left(E,-\right)$ 
commutes with infinite direct sums. We denote the full subcategories 
of perfect modules in $\modA$ and $\hprojA$ by $\perf(\A)$ and 
$\prfhprA$, respectively.
 
In any category, an object $E$ is a \em retract \rm of an object $F$ 
if there exist morphisms $E \rightarrow F \rightarrow E$ whose 
composition is the identity. For $E, F \in \modA$  
we say that $E$ is a \em homotopy retract \rm of $F$ if there 
exist $E \rightarrow F \rightarrow E$ whose composition is 
homotopic to identity. In other words, $E$ is
a retract of $F$ in $H^0(\modA)$. In additive categories 
the notion of a retract is the same as that of a direct summand. 
The category $D_c(\A)$ is the Karoubi completion of $H^0(\sffgA)$ 
inside $D(\A)$ \cite[\S5]{Keller-DerivingDGCategories}. 
Thus $\prfhprA$ coincides with the full subcategory in $\modA$ of 
homotopy retracts of elements of $\sffgA$.

Let $(-) \ldertimes_\A M$,  $M \ldertimes_{\mathcal{B}} (-)$, 
$\rderhom_\B(M,-)$ and $\rderhom_{\Aopp}(M, -)$ be 
the corresponding derived functors. 
Whenever these functors are mentioned, unless made clear otherwise, 
$\C$ is assumed to be $k$.

We say that an $\AbimB$-bimodule $M$ is:
\begin{itemize}
\item \em $\A$-perfect \rm if $M_b$ is a perfect $\Aopp$-module 
for each $b \in \B$.
\item \em $\B$-perfect \rm if $\aM$ is a perfect $\B$-module for
each $a \in \A$.
\end{itemize}
We define similarly the notions of $\A$- and $\B$-
quasi-representability, $h$-projectivity, etc. 

Since acyclicity is defined levelwise in $\modk$,  
$\homm_{\B}(M,-)$ takes acyclic modules to acyclic for 
any $\B$-$h$-projective $M$. The same is true for
$M\otimes_{\B}(-)$, since it is trivially true for any 
$\B$-representable $M$. Thus to compute $\rder\homm_{\B}(M,-)$ 
and $M \ldertimes_{\B}(-)$ it suffices to take a $\B$-$h$-projective
resolution of $M$. Similarly, if $M$ is $\A$-$h$-projective
then $(-)\otimes_{\A}M$ and $\homm_{\Aopp}(M,-)$ compute 
$(-)\ldertimes_{\A}M$ and $\rder\homm_{\Aopp}(M,-)$.
If $k$ is a field\footnote{If it is not, one should take cofibrant
replacements of $\A$ and $\B$ 
\cite[Prop 3.3]{Toen-TheHomotopyTheoryOfDGCategoriesAndDerivedMoritaTheory}.}
then any $h$-projective $\AbimB$-bimodule is both $\A$- and
$\B$-$h$-projective \cite[\S6.1]{Keller-DerivingDGCategories}, 
and hence the derived functors above can
be computed by taking an $h$-projective resolution of $M$. 

It follows from the above, that $M$ is
\begin{itemize}
\item \em $\A$-perfect \rm if and only if $M \ldertimes_\B (-)$ restricts to 
$D_{c}(\Bopp) \rightarrow D_{c}(\Aopp)$.
\item \em $\B$-perfect \rm if and only if $(-) \ldertimes_\A M$ restricts to 
$D_{c}(\A) \rightarrow D_{c}(\B)$.
\end{itemize}

If $k$ is a field we can be more precise. Let $M \in \hproj(\AbimB)$. 
The functors $(-) \otimes_\A M$ and $M \otimes_\B (-)$ restrict to 
$\A \rightarrow \hproj(\B)$ and $\Bopp \rightarrow \hproj(\Aopp)$
and
\begin{itemize}
\item $\A$-perfect if and only if $M \otimes_{\mathcal{B}} (-)$
restricts to a functor $\prfhprBopp \rightarrow \prfhprAopp$. 
\item $\B$-perfect if and only if 
$(-) \otimes_{\mathcal{A}} M$ restricts to a functor 
$\prfhprA \rightarrow \prfhprB$.

\end{itemize}

\subsection{Duals and adjoints}
\label{section-duals-and-adjoints}

As before, let $\A$ be a DG-category. Define \em the diagonal $\A\text{-}\A$ 
bimodule $\A$ \rm by setting 
$\aA_b = \homm_{\A}(b,a)$ for any $a, b \in \A$. 
Then $\aA$ and $\A_a$ are precisely the representable modules
$\homm_\A(-,a)$ and $\homm_\A(a,-)$ in $\modA$ and $\modAopp$. 
This coincides with the notation introduced in \S\ref{section-DG-modules}.

The diagonal bimodule corresponds to the functor 
$\A \rightarrow \modA$ which sends $a \mapsto \aA$.
We have natural functorial isomorphisms
\begin{align}
\label{eqn-natural-isomorphisms-for-hom(A,-)-and-(-)-otimes-A}
\homm_{\A}\left(\A, -\right) \simeq
\id_\modA \simeq  (-) \otimes_{\A} \A 
\end{align}
given for any $\A$-module $M$ explicitly by
\begin{align*}
M \rightarrow M \otimes_\A \A \colon \quad &
s \mapsto s \otimes \id_a \quad \quad &\forall\; a \in \A, s \in M_a \\
M \otimes_\A \A \rightarrow M \colon \quad &
s \otimes \alpha \mapsto s.\alpha \quad \quad &\forall\; a,b \in \A, s
\in M_b, \alpha \in \bA_a \\
M \rightarrow \homm_\A\left(\A,M\right) \colon \quad &
s \mapsto \;
\left( \alpha \mapsto s.\alpha \quad \forall\; b \in \A, \alpha
\in \aA_b \right) \quad \quad &\forall\; a \in \A, s \in M_a \\
\homm_\A\left(\A,M\right) \rightarrow M\colon\quad &
\alpha \mapsto \alpha(\id_a) &
a \in \A, \alpha \in \homm_\A\left(\aA,M\right).
\end{align*}
We use these isomorphisms implicitly throughout the paper. 

On the other hand, $\homm_{\A}\left(-, \A\right)$ is \em the dualizing
functor \rm 
$$(-)^\A \colon (\modA)^{\opp} \rightarrow \modAopp. $$ 
Explicitly, for any $C \in \modA$ its \em dual module $C^{\A}$ \rm 
is the $\Aopp$-module $a \mapsto \homm_{\A}(C, \leftidx{_a}{\A})$. 
For any morphism $C \xrightarrow{\alpha} D$ in $\modA$
\em the dual morphism $\alpha^\A$ \rm is defined with a sign twist: 
for each $a \in \A$ define 
the requisite morphism $\homm_{\A}(D,\leftidx{_a}{\A}) \rightarrow   
\homm_{\A}(C,\leftidx{_a}{\A})$ by  
$$\beta \mapsto (-1)^{\deg(\beta) \deg(\alpha)} \beta \circ \alpha.$$ 

Tautologically, $(-)^{\A}$ restricts to $\id$ on
the Yoneda embedded subcategories 
$\Aopp \hookrightarrow (\modA)^{\opp}$ and 
$\Aopp \hookrightarrow \modAopp$. 
Therefore it induces an equivalence 
$$\sffg\left(\A\right)^{\opp} \xrightarrow{\sim} 
\sffg(\Aopp)$$ 
and a quasi-equivalence 
$$\prfhprA^{\opp} \rightarrow \prfhprAopp,$$ 
whose induced maps on morphism complexes are homotopy
equivalences. By abuse of notation, we also use $(-)^{\A}$ to refer 
to the dualizing functor for $\Aopp$. The double dualizing functor 
$(-)^{\mathcal{A}\mathcal{A}}\colon \modA \rightarrow \modA$ 
is isomorphic to the identity on $\sffg(\mathcal{A})$ and homotopic
to the identity on $\prfhprA$. An analogous claim holds for
$(-)^{\mathcal{A}\mathcal{A}}\colon \modAopp \rightarrow \modAopp$.

Let $C \in \modB$, $D \in \modA$ and let $M$ be an $\AbimB$-bimodule. 
There is a natural map of DG $k$-modules
\begin{align}
\label{eqn-tensor-product-and-hom-map} 
D \otimes_\A \homm_\B\left(C,M\right) 
\rightarrow 
\homm_\B\left(C, D \otimes_\A M\right) 
\end{align}
defined by setting for any $a \in \A$
$$ s \otimes \gamma \mapsto 
\bigl(\forall\; t \in C, \quad t \mapsto s \otimes \gamma(t) \bigr) 
\quad \quad 
s \in D_a, \gamma \in \homm_\B\left(C,\aM\right).
$$
This map is clearly an isomorphism when
either $C$ or $D$ are representable. 
It follows that it is an isomorphism
when either $C$ or $D$ lie in $\sffg(\A)$ and 
a homotopy equivalence when either $C$ or $D$ lie in $\prfhprA$.

If in \eqref{eqn-tensor-product-and-hom-map} we set $\B = \A$
and let $M$ be the diagonal bimodule $\A$ 
we obtain the \em evaluation map \rm 
\begin{align} \label{eqn-evaluation-map}
D \otimes_{\A} C^{\A} \xrightarrow{\ev} \homm_{\A}(C, D).
\end{align}
It is the same map of DG $k$-modules as the composition map 
$$ \homm_{\A}(\A, D) \otimes_{\A} \homm_{\A}(C, \A) 
\rightarrow 
\homm_{\A}(C,D). $$

Let $M$ be an $\AbimB$ bimodule. 
We define $M^\A$, \em the dual of $M$ with respect to
$\A$\em, to be the $\BbimA$-bimodule $\homm_{\Aopp}(M, \A)$. 
In other words, $M^\A$ corresponds 
to the functor $\B \rightarrow \modA$ which maps $b \mapsto (M_b)^\A$. 
Similarly, we define $M^\B$, \em the dual of $M$ with respect to
$\B$\em, to be the $\BbimA$-bimodule $\homm_{\B}(M, \B)$, 
which corresponds to
the functor $\mathcal{A}^{\opp} \rightarrow \modBopp$
which maps $a \mapsto (\leftidx{_a}{M})^\B$. 
More generally, define the functor 
$(-)^\A \colon (\AmodB)^\opp \rightarrow \BmodA$ 
fiberwise over $\B$ by the dualising functor of $\modA$, and 
define $(-)^\B$ similarly.  
Denote by $(-)^{\tilde{\A}}$ and $(-)^{\tilde{\B}}$ their derived functors 
$D(\AbimB)^\opp \rightarrow D(\BbimA)$. Since $(-)^\A$
is defined fiberwise over $\B$ it sends
$\A$-$h$-projective acyclic bimodules to acyclic ones. 
It follows that if $M$ is $\A$-$h$-projective then $M^\lderA \simeq (M)^\A$
in $D(\BbimA)$. 
Similarly, if $M$ is $\B$-$h$-projective, 
then $M^\lderB \simeq (M)^\B$ in $D(\BbimA)$. 

The evaluation map \eqref{eqn-evaluation-map}
induces a morphism of functors $\modB \rightarrow \modA$
\begin{align} \label{eqn-otimesM^B-to-hom-M-*-map}
(-) \otimes_{\B} M^\B \rightarrow \homm_{\B}(M, -)  
\end{align}
It follows from the above that for any $M$ the map 
\eqref{eqn-otimesM^B-to-hom-M-*-map} is an isomorphism on all of $\sffg(\B)$ 
and a homotopy equivalence on all of $\prfhprB$. On the other hand, 
if $M$ is $\B$-$h$-projective and $\B$-perfect, 
then for any $N \in \modB$ the morphism \eqref{eqn-otimesM^B-to-hom-M-*-map} is 
a quasi-isomorphism. This is because 
all $\leftidx{_a}{M}$ lie in $\prfhprB$ and thus 
\eqref{eqn-otimesM^B-to-hom-M-*-map} 
is a homotopy equivalence levelwise in $\modk$. Similarly, we obtain 
a morphism of functors 
\begin{align} \label{eqn-otimesM^A-to-hom-M-*-map}
M^\A \otimes_{\A} (-)  \rightarrow \homm_{\A}(M, -)  
\end{align}
which is a quasi-isomorphism on all of $\modAopp$ whenever $M$
is $\A$-$h$-projective and $\A$-perfect. 

Consider the map 
$$ M^\B \otimes_\A M \rightarrow \B$$
given by the Tensor-Hom adjunction counit 
\ref{eqn-adjunction-counit-for-M} evaluated at the diagonal bimodule 
$\B$. Taking its right adjoint with respect to $M^\B \otimes_\A (-)$
yields a map $M \rightarrow M^{\B\B}$. The induced natural transformation
\begin{align}
\label{eqn-double-B-dual-transformation}
\id \rightarrow (-)^{\B\B}
\end{align}
is a quasi-isomorphism for any $\B$-$h$-projective and $\B$-perfect
$M$, since it is then a homotopy equivalence levelwise in $\modB$. 
We similarly define a natural transformation 
\begin{align}
\label{eqn-double-A-dual-transformation}
\id \rightarrow (-)^{\A\A}
\end{align}
which is a quasi-isomorphism for any $\A$-$h$-projective and
$\A$-perfect $M$. 

The above properties of natural transformations
\eqref{eqn-otimesM^B-to-hom-M-*-map}-\eqref{eqn-double-A-dual-transformation}
imply the following:
\begin{lemma}
\label{lemma-hom-to-tensor-with-dual-and-double-duals}
\begin{enumerate}
\item
\label{item-MtildeA-to-RHom}
 For any $M \in D_{\Aperf}(\AbimB)$ we have an isomorphism 
of functors $D(\Aopp) \rightarrow D(\Bopp)$: 
\begin{align}
\label{eqn-MtildeA-to-RHom-iso}
M^{\tilde{\A}} \ldertimes_\A (-) \simeq 
\rder\homm_\A(M,-).
\end{align}
\item 
For any $M \in D_{\Bperf}(\AbimB)$ we have an isomorphism 
of functors $D(\B) \rightarrow D(\A)$: 
\begin{align}
\label{eqn-MtildeB-to-RHom-iso}
(-) \ldertimes_\B M^{\tilde{\B}} \xrightarrow{\sim}
\rder\homm_\B(M,-).
\end{align}
\item 
We have an isomorphism of endofunctors of $D_{\Aperf}(\AbimB)$:
\begin{align}
\label{eqn-identity-to-M-double-A-dual-iso}
\id \xrightarrow{\sim} (-)^{\tilde{\A}\tilde{\A}}.
\end{align}
\item 
We have an isomorphism of endofunctors of $D_{\Bperf}(\AbimB)$:
\begin{align}
\label{eqn-identity-to-M-double-B-dual-iso}
\id \xrightarrow{\sim} (-)^{\tilde{\B}\tilde{\B}}.
\end{align}
\end{enumerate}
\end{lemma}

In view of Tensor-Hom adjunction we then have:

\begin{cor}
\label{cor-derived-M-and-M^A-and-M^B-adjunction}
\begin{enumerate}
\item
\label{item-MddA-and-M-adjunction}
For any $M \in D_{\Aperf}(\AbimB)$ the functor 
$$ (-) \ldertimes_\B M^{\tilde{\A}} \colon \quad D(\B) \rightarrow D(\A)$$
is left adjoint to the functor
$$ (-) \ldertimes_\A M \colon \quad D(\A) \rightarrow D(\B).$$
\item 
\label{item-M-and-MddB-adjunction}
For any $M \in D_{\Bperf}(\AbimB)$ the functor 
$$ (-) \ldertimes_\B M^{\tilde{\B}} \colon \quad D(\B) \rightarrow D(\A)$$
is right adjoint to the functor
$$ (-) \ldertimes_\A M \colon \quad D(\A) \rightarrow D(\B).$$
\end{enumerate}
\end{cor}

\begin{proof}
\begin{enumerate}
\item By Lemma \ref{lemma-hom-to-tensor-with-dual-and-double-duals} 
we have the following isomorphism of functors
\begin{align}
\label{eqn-iso-of-tensor-M-to-RHom-MddA}
(-) \ldertimes_\A M 
\xrightarrow{\eqref{eqn-identity-to-M-double-A-dual-iso}}
(-) \ldertimes_\A M^{\tilde{\A}\tilde{\A}}  
\xrightarrow{\eqref{eqn-MtildeA-to-RHom-iso}}
\rder\homm_{\A}(\MddA,-). 
\end{align}
This isomorphism transforms the derived Tensor-Hom adjunction 
$$(-) \ldertimes_\B M^{\tilde{\A}} \quad \longleftrightarrow \quad \rder\homm_{\A}(\MddA,-) $$
with its unit \eqref{eqn-adjunction-unit-for-M} and 
counit \eqref{eqn-adjunction-counit-for-M} into the desired adjunction 
$$(-) \ldertimes_\B M^{\tilde{\A}} \quad \longleftrightarrow \quad (-) \ldertimes_\A M.$$
\item Similarly, by 
Lemma \ref{lemma-hom-to-tensor-with-dual-and-double-duals}
we have an isomorphism 
\begin{align}
\label{eqn-iso-of-tensor-MddB-to-RHom-M}
(-) \ldertimes_\B \MddB 
\xrightarrow{\eqref{eqn-MtildeB-to-RHom-iso}}
\rder\homm_{\B}(M,-) 
\end{align}
which produces the desired adjunction out of the Tensor-Hom adjunction 
$$(-) \ldertimes_\A M \quad \longleftrightarrow \quad
\rder\homm_{\B}(\MddB,-). $$
\end{enumerate}
\end{proof}

It is helpful to have the units and counits 
of the adjunctions in Cor. \ref{cor-derived-M-and-M^A-and-M^B-adjunction} 
written down explicitly in terms of the maps between corresponding
bimodules:

\begin{defn}
Let $M \in \AmodB$. The \em $\B$-trace map \rm 
\begin{align}
\label{eqn-trace-map-B}
M^\B \otimes_{\A} M \xrightarrow{tr} \B. 
\end{align}
in $\BmodB$ is the co-unit \eqref{eqn-adjunction-counit-for-M} of 
the Tensor-Hom adjunction evaluated at the diagonal bimodule $\B$.\\ 
The \em derived $\B$-trace map \rm is the induced map 
$\MddB \ldertimes_{\A} M \xrightarrow{tr} \B$ in $D(\BbimB)$. 

The \em $\A$-trace map \rm 
\begin{align}
\label{eqn-trace-map-A}
M \otimes_{\B} M^{\A} \xrightarrow{\trace} \A
\end{align}
and its derived version are defined similarly. 
\end{defn}

For $\B$- and $\A$-perfect $M$ the associativity of the composition map 
implies that the natural transformations
\begin{align*}
(-) \ldertimes_\A M \ldertimes_\B \MddA \xrightarrow{\trace} \id_{D(\A)} \\
(-) \ldertimes_\B \MddB \ldertimes_\A M \xrightarrow{\trace} \id_{D(\B)} 
\end{align*}
coincide with the counits of adjunctions in 
Cor. \ref{cor-derived-M-and-M^A-and-M^B-adjunction} 
\eqref{item-MddA-and-M-adjunction}-\eqref{item-M-and-MddB-adjunction}. 

\begin{defn}
Let $M \in \AmodB$. 
The \em $\A$-action map \rm 
\begin{align}
\label{eqn-A-action-map}
\A \xrightarrow{\action} \homm_\B(M, M)
\end{align}
in $\AmodA$ is the unit \eqref{eqn-adjunction-unit-for-M} of 
the Tensor-Hom adjunction evaluated at $\A$.
The \em derived $\A$-action map \rm is the induced map 
$\A \xrightarrow{\action} \rder\homm_\B(M, M)$ in $D(\AbimA)$. 
When $M$ is $\B$-perfect we also use this term 
for the corresponding map $\A \xrightarrow{\action} M \ldertimes_\B
\MddB$ 
obtained via the isomorphism 
$\eqref{eqn-MtildeB-to-RHom-iso}$. 

The \em $\B$-action map \rm 
\begin{align}
\label{eqn-B-action-map}
\B \xrightarrow{\action} \homm_{\Aopp}(M,M)
\end{align}
and its derived versions are defined similarly. 
\end{defn} 
For $\B$- and $\A$-perfect $M$ the induced natural transformations
\begin{align*}
\id_{D(\B)} \xrightarrow{\action} (-) \ldertimes_\B \MddA \ldertimes_\A M \\
\id_{D(\A)} \xrightarrow{\action} (-) \ldertimes_\A M \ldertimes_\B \MddB 
\end{align*}
coincide with the units of adjunctions in 
Cor. \ref{cor-derived-M-and-M^A-and-M^B-adjunction}
\eqref{item-MddA-and-M-adjunction}-\eqref{item-M-and-MddB-adjunction}. 
Showing this amounts to checking
that
\begin{align*}
\vcenter{
\xymatrix{
(-) \otimes_\A \A 
\ar[rr]^<<<<<<<<<<{\id \otimes \action}
\ar[d]^{\simeq}
& &
(-) \otimes_\A \homm_{\B}(M,M)
\ar[d]^{\eqref{eqn-tensor-product-and-hom-map}}
\\
(-)
\ar[rr]_<<<<<<<<<<<<{\eqref{eqn-adjunction-unit-for-M}}
& &
\homm_{\B}(M, (-) \otimes_\A M)
}
}
\end{align*}
commutes. It is a straightforward exercise we leave to the reader. 

We would now like to lift these derived adjunctions to homotopy ones.
That is, given an $\A$- and $\B$-perfect $M \in \AmodB$ we
would like to write down $h$-projective resolutions of $M$, $\MddA$ and 
$\MddB$ and
four maps which induce in the homotopy category the units and 
counits of the two adjunctions in 
Cor. \ref{cor-derived-M-and-M^A-and-M^B-adjunction}.

We use very specific resolutions of $M$, $\MddA$ and $\MddB$ 
obtained via the bar-construction, 
cf. \cite[\S6.6]{Keller-DerivingDGCategories}. We briefly recall 
the essentials. Let $\barA \rightarrow \A$ and $\barB \rightarrow \B$ be 
the bar-resolutions of the diagonal bimodules in $\AmodA$ and $\BmodB$. 
These are quasi-isomorphisms with $\barA$ and $\barB$ semifree. 
The induced natural transformations 
$ (-) \otimes_\A \barA \rightarrow \id_{\modA}$ and 
$ (-) \otimes_\B \barB \rightarrow \id_{\modB}$
are functorial $h$-projective resolutions for $\A$- and $\B$-modules, 
respectively. This can be seen via the following useful fact:
\begin{prps}
\label{prps-C-hproj-times-hproj-equals-hproj}
Let $\A$, $\B$ and $\C$ be DG-categories. 
Let $M$ be an $\AbimB$-bimodule and $N$ be a $\BbimC$-bimodule. 
If either of the following holds
\begin{enumerate}
\item $M$ is $h$-projective and $N$ is a $\C$-$h$-projective.
\item $M$ is $\A$-$h$-projective and $N$ is $h$-projective.
\end{enumerate}
then $M \otimes_B N$ is an $h$-projective $\AbimC$-bimodule.  
\end{prps}
\begin{proof}
Suppose $M$ is $h$-projective and $N$ is $\C$-$h$-projective
and let $Q$ be any acyclic $\AbimC$-bimodule. 
By the adjunction of $(-) \otimes_\B N$ and 
$\homm_{\C}(N, -)$ done over $\A$ we have a natural 
isomorphism
\begin{align*}
\homm_{\AbimC}
\left(M \otimes_B N, Q \right) 
\simeq
\homm_{\AbimB}
\left(M , \homm_\C\left(N, Q \right) \right).
\end{align*}
For any $a \in \A$ and $b \in \B$, $\leftidx{_b}{N}$ and $Q_a$ are
an $h$-projective and an acyclic $\C$-modules. It follows that 
$\homm_\C(N,Q)$ is an acyclic $\AbimB$-bimodule, and hence
$\homm_{\AbimB}\left(M , \homm_\C\left(N, Q \right) \right)$ is acyclic. 
We have now shown $\homm_{\AbimC} \left(M \otimes_B N, Q \right)$ to
be acyclic for any acyclic $Q$, whereby $M \otimes_B N$ is 
an $h$-projective $\AbimC$-bimodule. 

The case of $M$ being $\A$-$h$-projective and $N$ being
$h$-projective is treated similarly. 
\end{proof}

Similarly, for $\AbimB$ bimodules an $h$-projective resolution could
be obtained by tensoring with $\overline{\A^\opp \otimes \B}$. However, 
there is another resolution more suited to our needs:

\begin{cor}
\label{cor-barA-M-barB-is-h-proj-resolution-of-M}
Let $M \in \AmodB$. Then 
$\barA \otimes_\A M \otimes_\B \barB \rightarrow M$ 
is an $h$-projective resolution of $M$. 
\end{cor}

\begin{proof}
By Prop. \ref{prps-C-hproj-times-hproj-equals-hproj} the bimodule
$\barA \otimes_\A M$ is $\A$-$h$-projective, and then by Prop.
\ref{prps-C-hproj-times-hproj-equals-hproj} again 
$\left(\barA \otimes_\A M\right) \otimes_\B \barB$ is
$h$-projective. 
\end{proof}

\begin{defn}
Define $\bhproj(\AbimB)$ to be the full subcategory of
$\hproj(\AbimB)$ consisting of all bimodules of form $\barA \otimes_\A
M \otimes_\B \barB$ for some $M \in \AmodB$. 
\end{defn}

Note that by Cor. \ref{cor-barA-M-barB-is-h-proj-resolution-of-M}
we have a canonical identification $H^0(\bhproj(\AbimB)) \simeq D(\AbimB)$. 

Let $N$ be any $\AmodB$ bimodule. The quasi-isomorphisms $\barA \rightarrow \A$ and $\barB \rightarrow \B$ and functorial isomorphisms \eqref{eqn-natural-isomorphisms-for-hom(A,-)-and-(-)-otimes-A} yield functorial quasi-isomorphisms
\begin{align}
\label{eqn-homotopy-equivalences-with-tensoring-by-bar-complex}
\barA \otimes_{\A} N \xrightarrow{\sim} N \xleftarrow{\sim} N \otimes_{\B} \barB.  
\end{align}

If $N \in \hproj(\AbimB)$ then so are $\barA \otimes_\A N$ and $N \otimes_\B
\barB$, and the two quasi-isomorphisms 
in \eqref{eqn-homotopy-equivalences-with-tensoring-by-bar-complex} are
actually homotopy equivalences. If moreover $N \in \bhproj(\AbimB)$, 
we have canonical homotopy inverses of 
\eqref{eqn-homotopy-equivalences-with-tensoring-by-bar-complex}
\begin{align}
\label{eqn-inverses-of-homotopy-equivalences-with-tensoring-by-bar-complex}
\barA \otimes_{\A} N \xleftarrow{\sim} N 
\xrightarrow{\sim} N \otimes_{\B} \barB.  
\end{align}
induced by the comultiplication maps $\barA \rightarrow \barA
\otimes_\A \barA$ and $\barB \rightarrow \barB \otimes_\B \barB$
defined in \cite[\S6.6]{Keller-DerivingDGCategories}. Moreover, 
these are genuine right inverses -- the following compositions 
are not merely homotopic but equal to $\id$:
$$ N 
\xrightarrow{\eqref{eqn-inverses-of-homotopy-equivalences-with-tensoring-by-bar-complex}}
\barA \otimes_\A N 
\xrightarrow{\eqref{eqn-homotopy-equivalences-with-tensoring-by-bar-complex}}
N $$
$$ N 
\xrightarrow{\eqref{eqn-inverses-of-homotopy-equivalences-with-tensoring-by-bar-complex}}
N \otimes_\B \barB 
\xrightarrow{\eqref{eqn-homotopy-equivalences-with-tensoring-by-bar-complex}}
N. $$
This is our
main reason for introducing $\bhproj(\AbimB)$: it makes
a number of diagrams commute genuinely and not up to homotopy.  
Throughout the rest of the paper, where necessary,
we implicitly identify any $N \in \bhproj(\AbimB)$ with 
$\barA \otimes_\A N$ and $N \otimes_\B \barB$ via 
\eqref{eqn-homotopy-equivalences-with-tensoring-by-bar-complex} 
and \eqref{eqn-inverses-of-homotopy-equivalences-with-tensoring-by-bar-complex}.  

The dualisation functors $(-)^\A$ and $(-)^\B$ do not restrict to functors 
$\bhproj(\AbimB) \rightarrow \bhproj(\BbimA)$. We thus define:

\begin{defn}
\label{defn-homotopy-duals} 
Let $M \in \AmodB$. Define $\MhdA$ and $\MhdB$ to be the bimodules 
$\barB \otimes_B M^\A \otimes_\A \barA$ and 
$\barB \otimes_B M^\B \otimes_\A \barA$, respectively. 
\end{defn}

These are our chosen $h$-projective resolutions of the derived duals
$\MddA$ and $\MddB$. 
We now proceed to define the unit and counit maps of our homotopy adjunctions. 

\begin{defn}
Let $M \in \bhproj(\AbimB)$. 
The \em homotopy $\A$-trace map $M \otimes_\B \MhdA 
\xrightarrow{\trace} \barA$ \rm is the composition  
\begin{align}
\label{eqn-homotopy-A-trace-map}
M \otimes_\B \MhdA \xrightarrow{\barB \rightarrow \B}
M \otimes_\B M^\A \otimes_\A \barA \xrightarrow{\trace \otimes \id} \barA. 
\end{align}
Similarly, the \em homotopy $\B$-trace map 
$\MhdB \otimes_\A M \xrightarrow{\trace} \barB$ \rm is the composition
\begin{align}
\label{eqn-homotopy-B-trace-map}
\MhdB \otimes_\A M \xrightarrow{\barA \rightarrow \A}
\barB \otimes_\B M^\B \otimes_\A M \xrightarrow{\id \otimes \trace} \barB. 
\end{align}
\end{defn}
\begin{defn}
\label{defn-homotopy-action-maps}
Let $M \in \bhproj^{\Bperf}\!(\AbimB)$. The map 
\begin{align}
\label{eqn-otimes-M^B-to-hom-M-*-map-tensor-barA}
M \otimes_\B \MhdB \xrightarrow{\barB \rightarrow \B} 
M \otimes_\B M^\B \otimes_\A \barA 
\xrightarrow{\eqref{eqn-otimesM^B-to-hom-M-*-map} \otimes \id } 
\homm_\B(M,M) \otimes_\A \barA
\end{align}
is a quasi-isomorphism since 
\eqref{eqn-otimesM^B-to-hom-M-*-map} is one. 
Thus there exists a homotopy lift of 
$\barA \xrightarrow{\action \otimes \id }  
\homm_\B(M,M) \otimes_\A \barA$ along 
\eqref{eqn-otimes-M^B-to-hom-M-*-map-tensor-barA}. Choose 
once and for all such a lift and call it the 
\em homotopy $\A$-action map\rm
\begin{align}
\label{eqn-homotopy-A-action-map}
\barA \xrightarrow{\action} M \otimes_\B \MhdB.
\end{align}

Let $M \in \bhproj^{\Aperf}\!(\AbimB)$. 
We define similarly the \em homotopy $\B$-action map \rm
\begin{align}
\label{eqn-homotopy-B-action-map}
\barB \xrightarrow{\action} \MhdA \otimes_\A M. 
\end{align}
\end{defn}

\begin{prps}
\label{prps-M^B-and-M^A-are-homotopy-adjoints-of-M-via-S-SRS-S-maps}
Let $\A$ and $\B$ be DG-categories and $M \in \bhproj(\AbimB)$.

If $M$ is $\B$-perfect
$(-) \otimes_\B \MhdB$ is homotopy right adjoint to 
$(-) \otimes_\A M$ with the unit and the counit being the 
homotopy $\B$-action and $\B$-trace maps. That is, the compositions 
\begin{align}
\label{eqn-mb-mb_m_mb-mb-homotopic-to-the-identity-map}
\MhdB
\xrightarrow{\id \otimes \action} 
\MhdB \otimes_\A M \otimes_\B \MhdB 
\xrightarrow{\trace \otimes \id} 
\MhdB
\end{align}
\begin{align}
\label{eqn-m-m_mb_m-m-homotopic-to-the-identity-map}
M 
\xrightarrow{\action \otimes \id} 
M \otimes_\B \MhdB \otimes_\A M
\xrightarrow{\id \otimes \trace} 
M
\end{align}
are homotopic to the identity maps. 

If $M$ is $\A$-perfect then 
$(-) \otimes_\B \MhdA$ is homotopy left adjoint to 
$(-) \otimes_\A M$ with the unit and the counit being the 
homotopy $\A$-action and $\A$-trace maps. That is, the compositions 
\begin{align}
\label{eqn-ma-ma_m_ma-ma-homotopic-to-the-identity-map}
\MhdA
\xrightarrow{\action \otimes \id} 
\MhdA \otimes_\A M \otimes_\B \MhdA 
\xrightarrow{\trace \otimes \id} 
\MhdA
\end{align}
\begin{align}
\label{eqn-m-m_ma_m-m-homotopic-to-the-identity-map}
M 
\xrightarrow{\action \otimes \id} 
M \otimes_\B \MhdA \otimes_\A M
\xrightarrow{\trace \otimes \id} 
M
\end{align}
are homotopic to the identity maps. 
\end{prps}
\begin{proof}
The compositions 
\begin{align*}
M^{\tilde{\B}} 
\xrightarrow{\id \otimes \action} 
M^{\tilde{\B}} 
\ldertimes_\A M 
\otimes_\B 
M^{\tilde{\B}} 
\xrightarrow{\trace \otimes \id} 
M^{\tilde{\B}} \\
M 
\xrightarrow{\action \otimes \id} 
M \ldertimes_\B M^{\tilde{\B}} \otimes_\A M
\xrightarrow{\trace \otimes \id} 
M 
\end{align*}
are equal to $\id_{D(\BbimA)}$ and $\id_{D(\AbimB)}$. 
This is because the derived $\B$-action
and $\B$-trace maps are the unit and the counit of 
a genuine adjunction between $(-) \ldertimes_\B M^{\tilde{\B}}$
and $(-) \ldertimes_\A M$. 

By construction, the images of homotopy $\B$-trace and $\B$-action maps
in $D(\BbimB)$ are identified with their derived counterparts  
by the isomorphism $\MddB \simeq \MhdB$. 
It follows that the images of
\eqref{eqn-mb-mb_m_mb-mb-homotopic-to-the-identity-map} and 
\eqref{eqn-m-m_mb_m-m-homotopic-to-the-identity-map} in  
$D(\BbimA)$ and $D(\AbimB)$ are conjugate (and thus equal) 
to $\id_{D(\BbimA)}$ and $\id_{D(\AbimB)}$. Hence 
\eqref{eqn-mb-mb_m_mb-mb-homotopic-to-the-identity-map} and 
\eqref{eqn-m-m_mb_m-m-homotopic-to-the-identity-map} themselves 
are homotopic to $\id_{\BmodA}$ and $\id_{\AmodB}$.  

The second assertion is proved analogously. 
\end{proof}

Let $\A$, $\B$ and $\C$ be DG-categories. Let $M \in \AmodB$ and $N
\in \BmodC$. Consider the composition 
\begin{align}
\label{eqn-M-otimes-N-right-dual}
N^\C \otimes_\B M^\B \xrightarrow{\ev} 
\homm_\B \left(M, N^\C\right) = 
\homm_\B \left(M, \homm_\C(N,\C)\right) 
\underset{\sim}{\xrightarrow{\text{adj.}}}
\homm_\B \left(M \otimes_B N, \C\right) 
= \left(M \otimes_B N\right)^\C
\end{align}
The first map is the evaluation map \eqref{eqn-evaluation-map}, 
it is a quasi-isomorphism if $M$ is $\B$-perfect and
$\B$-$h$-projective. The second map is the adjunction isomorphism 
for $(-) \otimes_\B N$ and $\homm_\C(N,-)$. Similarly
\begin{align}
\label{eqn-M-otimes-N-left-dual}
N^\B \otimes_\B M^\A \xrightarrow{\ev} 
\homm_\Bopp \left(N, \homm_\Aopp(M,\A)\right) 
\underset{\sim}{\xrightarrow{\text{adj.}}}
\left(M \otimes_B N\right)^\A, 
\end{align}
is a quasi-isomorphism if $N$ is $\A$-perfect and
$\A$-$h$-projective. We have thus:
\begin{lemma}
Let $\A$, $\B$ and $\C$ be DG-categories. Let 
$M$ and $N$ be $\AbimB$- and $\BbimC$-bimodules. 

If $M$ is $B$-perfect we have an isomorphism in $D(\CbimA)$:
\begin{align}
N^{\tilde{\C}} \ldertimes_\B M^{\tilde{\B}}
\xrightarrow{\eqref{eqn-M-otimes-N-right-dual}} 
\left(M \ldertimes_B N\right)^{\tilde{\C}}.
\end{align}
 
If $N$ is $\B$-perfect we have an isomorphism in $D(\CbimA)$:
\begin{align}
N^{\tilde{\B}} \ldertimes_\B M^{\tilde{\A}} 
\xrightarrow{\eqref{eqn-M-otimes-N-left-dual}} 
\left(M \ldertimes_B N\right)^{\tilde{\A}}. 
\end{align}
\end{lemma}

More is true:
\begin{lemma}
\label{lemma-derived-trace-and-action-maps-are-isomorphic}
Let $\A$ and $\B$ be DG-categories and $M
\in D(\AbimB)$ be $\A$- and $\B$-perfect. 
Then $\B \xrightarrow{\action} M^{\tilde{\A}} \ldertimes_\A M$ 
is isomorphic in $D(\BbimB)$ to
\begin{align*}
\left(M^{\tilde{\B}} \ldertimes_\A M \xrightarrow{\trace} \B
\right)^{l\tilde{\B}}.
\end{align*}
Similarly, $A \xrightarrow{\action} M \ldertimes_B M^{\tilde{\B}}$ 
is isomorphic in $D(\AbimA)$ to 
\begin{align*}
\left(M \ldertimes_\B M^{\tilde{\A}} \xrightarrow{\trace} \A
\right)^{r\tilde{\A}}.
\end{align*}
Here by $(-)^{l\tilde{B}}$ we mean dualising an $\BbimB$-bimodule 
as a left $\B$-module. Similarly for $(-)^{r\tilde{\A}}$, etc.
\end{lemma}

\begin{proof}
We only prove the first assertion, the second assertion is proved
similarly. 
Replace $M$ by an $h$-projective resolution. Then in $D(\BbimB)$
the map $\B \xrightarrow{\action} M^{\tilde{\A}} \ldertimes_\A M$ is
isomorphic  to $\B \xrightarrow{\action} \homm_\Aopp(M,M)$ 
and $M^{\tilde{\B}} \ldertimes_\A M \xrightarrow{\trace} \B$
is isomorphic to $M^\B \otimes_\A M \xrightarrow{\trace} \B$. 
It now suffices to show that in $\BmodB$ the diagram 
\begin{align}
\label{eqn-dual-trace-action-map-commutative-square}
\vcenter{
\xymatrix{
\B 
\ar[dd]_{\action}
\ar[r]^<<<<<<<<<<<<<{\action}
& 
\homm_{\Aopp}\left(M,M\right)
\ar[d]^{\id \otimes \eqref{eqn-double-B-dual-transformation}}
\\
&
\homm_{\Aopp}\left(M, M^{\B\B}\right)
\ar[d]^{\text{adjunction}}
\\
\homm_\Bopp\left(\B, \B\right)
\ar[r]_<<<<{\trace}
&
\homm_\Bopp\left(M^\B \otimes_\A M, \B\right)
} 
}
\end{align}
is commutative, since its left column is an isomorphism 
and its right column a quasi-isomorphism.  
The diagram \eqref{eqn-dual-trace-action-map-commutative-square} commutes
because both its halves can be readily seen to compose into the element of 
$$
\homm_{\BbimB}\left(\B, \homm_{\Bopp}\left(M^\B \otimes_\A M, \B\right)\right)
$$
which is adjoint to the trace map 
$M^\B \otimes_\A M \xrightarrow{\trace}\B$ in 
$$
\homm_{\BbimB}\left(M^\B \otimes_\A M, \B\right)
$$
under the adjunction of $M^\B \otimes_\A M \otimes_\B (-)$
and $\homm_{\Bopp}\left(M^\B \otimes_\A M, -\right)$. 
\end{proof}


Finally, we have the following analogue of 
Prps.~\ref{prps-C-hproj-times-hproj-equals-hproj} with 
$h$-projectivity replaced by perfection:

\begin{prps}
\label{prps-C-perfect-plus-perfect-equals-perfect}
Let $\A$, $\B$ and $\C$ be DG-categories. Let $M$ be a perfect
$\AbimB$-bimodule and $N$ be a $\C$-perfect $\BbimC$-bimodule. 
Then $M \ldertimes_B N$ is a perfect $\AbimC$-bimodule.  
\end{prps}
\begin{proof} 
Let $\bigoplus Q_i$ be an infinite direct sum of 
$\AbimC$-bimodules. We have a chain of natural isomorphisms:
\begin{align}
\label{item-otimes-N-adjunction}
\homm_{D(\AbimC)}
\left(M \ldertimes_B N, \bigoplus Q_i \right) 
\simeq
\homm_{D(\AbimB)}
\left(M , \rderhom_\C\left(N, \bigoplus Q_i\right) \right) \\
\label{item-N-C-perfect}
\homm_{D(\AbimB)}
\left(M , \rderhom_\C\left(N, \bigoplus Q_i\right) \right) 
\simeq
\homm_{D(\AbimB)}
\left(M , \bigoplus \rderhom_\C\left(N, Q_i\right) \right) \\
\label{item-M-perfect}
\homm_{D(\AbimB)}
\left(M , \bigoplus \rderhom_\C\left(N, Q_i\right) \right) 
\simeq
\bigoplus \homm_{D(\AbimB)}
\left(M , \rderhom_\C\left(N, Q_i\right) \right) \\
\label{item-rhom-N-adjunction}
\bigoplus \homm_{D(\AbimB)}
\left(M , \rderhom_\C\left(N, Q_i\right) \right) 
\simeq
\bigoplus \homm_{D(\AbimC)}
\left(M \ldertimes_B N , Q_i \right).
\end{align}
The isomorphisms \eqref{item-otimes-N-adjunction} and 
\eqref{item-rhom-N-adjunction} are due to the adjunction of 
$(-) \ldertimes_\B N$ and $\rderhom_{\C}(N, -)$ done over $\A$, 
\eqref{item-N-C-perfect} is due to $N$ being $\C$-perfect 
and \eqref{item-M-perfect} is due to $M$ being perfect. 

Thus $\homm_{D(\AbimC)}\left(M \ldertimes_\B N, -\right)$ 
commutes with infinite direct sums, i.e. $M \ldertimes_\B N$ 
is perfect. 
\end{proof}

Recall that a DG-category $\A$ is called \em smooth \rm if
the diagonal bimodule $\A$ is a perfect $\AbimA$-bimodule. 

\begin{cor}
\label{cor-for-smooth-A-perfect-means-perfect}
Let $\A$ be a smooth DG category and $\B$ be any DG-category. 
Then any $\B$-perfect $\AbimB$-bimodule $N$ is perfect.   
\end{cor}
\begin{proof}
By definition, $\A$ being smooth means that $\A$ is 
a perfect $\AbimA$-bimodule. We then apply Lemma 
\ref{prps-C-perfect-plus-perfect-equals-perfect} 
to conclude that $N \simeq \A \otimes_\A N$ is perfect. 
\end{proof}

\section{Twisted complexes and twisted cubes}
\label{section-twisted-complexes}

\subsection{Twisted complexes}
\label{section-subsection-twisted-complexes}

The notion of a twisted complex was introduced in
\cite{BondalKapranov-EnhancedTriangulatedCategories}. 
There exist at present two different conventions for writing 
down twisted complexes: the original one introduced in 
\cite{BondalKapranov-EnhancedTriangulatedCategories} and a
slightly different one introduced in 
\cite{BondalLarsenLunts-GrothendieckRingofPretriangulatedCategories}
where all the objects in a twisted complex are shifted
so as to ensure that all the twisted maps have degree $1$. 
Abstractly, this latter convention is more natural as these
shifts are precisely what one has to do when taking the convolution
of a twisted complex.

However, all the twisted complexes we work with in this paper are
lifts of genuine complexes in the homotopy category, and hence 
they exist naturally in the convention of
\cite{BondalKapranov-EnhancedTriangulatedCategories}. For this 
reason we are going to present the material in this section, such as
the formulas for dualizing and tensoring twisted complexes, in 
the notation of \cite{BondalKapranov-EnhancedTriangulatedCategories}. 
The authors are well aware that the signs in these formulas are 
much simpler in the notation of
\cite{BondalLarsenLunts-GrothendieckRingofPretriangulatedCategories}. 
However, to actually apply any formula in 
\cite{BondalLarsenLunts-GrothendieckRingofPretriangulatedCategories}
convention to the twisted complexes we work with throughout the paper, 
we'd first have to shift everything to make all the twisted maps 
have degree $1$, then apply the formula, and then shift everything back 
to relate the answer to what we are working with. This would introduce
back all the complicated signs, and it is therefore better to write
down the formulas 
in \cite{BondalKapranov-EnhancedTriangulatedCategories} convention
from the start.

The definitions in the published version
of \cite{BondalKapranov-EnhancedTriangulatedCategories} contain sign errors. 
For reader's convenience we give below the corrected versions of 
these definitions:

\begin{defn}
\label{defn-twisted-complex} 
Let $\A$ be a DG-category. A \em twisted complex \rm over $\A$ is
a collection 
$$ \bigl\{ (E_i)_{i \in \mathbb{Z}}, \; \alpha_{ij} \colon E_i \rightarrow E_j  \bigr\} $$
where $E_i$ are objects in $\A$ with $E_i = 0$ for all but finite number 
of $i$, and $\alpha_{ij}$ are morphisms in $\A$ of degree $i - j + 1$
satisfying the condition
$$ (-1)^j d\alpha_{ij} + \sum_k \alpha_{kj} \circ \alpha_{ik} = 0. $$
\end{defn}
A twisted complex is called \em one-sided \rm if $\alpha_{ij} = 0$ for all $i
\geq j$. 

We adopt the following convention: to write down a twisted complex we 
write down two expressions separated by a comma. First expression is the 
$i$-th graded part of the twisted complex. The second expression is 
the twisted map from $i$th to $j$th graded parts of the twisted complex. 
E.g. $(E_i, \alpha_{ij})$ is a twisted complex whose $i$-th graded part 
is $E_i$ and whose twisted map from $E_i$ to $E_j$ is $\alpha_{ij}$. 

To make twisted complexes over $\A$ into a DG-category we define
the $\homm$-complex from a twisted complex $(E_i, \alpha_{ij})$ to 
a twisted complex $(F_i, \beta_{ij})$ to be the complex of $k$-modules
whose degree $p$ part is  
$$ \coprod_{p = q + l - k} \homm^q_{\A}(E_k, F_l) $$
with the differential defined by setting, 
for each $\gamma \in  \homm^q_{\A}(E_k, F_l)$,
$$ d\gamma = (-1)^l d_{\A} \gamma + \sum_{m \in
\mathbb{Z}}\left( \beta_{lm} \circ \gamma - (-1)^{q + l -k} \gamma
\circ \alpha_{mk} \right), $$ 
where $d_\A$ is the differential on morphisms in $\A$. 

The signs and indices in the definitions above are set up precisely
so as to ensure that the following notion of \em convolution \rm 
extends naturally to a fully faithful functor from 
the DG-category of twisted complexes over $\A$ to the DG-category $\modA$.
But first we need to define the notion of a \em shift \rm of an $\A$-module. 
We do it levelwise in $\modk$ and, since we are dealing with right
modules, we do not twist the $\A$-action, that is:

\begin{defn} Let $M$ be an $\A$-module. For any $n \in \mathbb{Z}$ 
define the $\A$-module $M[n]$ by setting 
$$(M[n])_a = M_a[n] \quad\quad \forall\;a \in \A$$
and having $\A$ act via its action on $M$. That is, for
any $\alpha \in \leftidx{_a}{A_b}$ and any $s \in (M[n])_a$ 
we set $s\cdot_{M[n]}\alpha \in (M[n])_b$
to be $s \cdot_{M} \alpha$.  
\end{defn}

\begin{defn}
Let $\A$ be a DG-category and let $(E_i, \alpha_{ij})$ be a twisted
complex over $\A$. Let $\bigoplus_i E_i[-i]$ be the $\A$-module
where we use the Yoneda embedding to embed each $E_i$ into $\modA$. 
The \em convolution \rm of $(E_i, \alpha_{ij})$
is the $\A$-module obtained by taking $\bigoplus_i E_i[-i]$
and endowing it with a new differential 
$d + \sum_{i,j} \alpha_{ij}$, where $d$ is 
the natural differential of $\bigoplus_i E_i[-i]$. 

We use curly brackets to denote taking the convolution of the twisted
complex, e.g. $\bigl\{ E_i, \alpha_{ij} \bigr\}$. 
\end{defn}

The most time-consuming part of proving the results below
is in getting the signs to agree. 
Recall the definitions of the bimodule-valued tensor product
and $\homm$ functors 
\eqref{eqn-bimodule-valued-tensor-product}-\eqref{eqn-bimodule-hom-as-left-modules}.
In particular, for any maps $E \xrightarrow{\alpha} E'$ in $\BmodA$ 
and $F \xrightarrow{\beta} F'$ in $\AmodC$ the product map 
$$ E \otimes_\A F \xrightarrow{\alpha \otimes \beta } E' \otimes F' $$
in $\BmodC$ is given for every $(b,c) \in \B \otimes \Copp$ and $a \in \A$ by
$$ e \otimes f \mapsto (-1)^{\deg(e) \deg(\beta)} \alpha(e) \otimes \beta(f) 
\quad\quad e \in \leftidx{_b}{E}{_a}, f \in \leftidx{_a}{F}{_c}.$$ 
Similarly, for any maps $E' \xrightarrow{\alpha} E$ in $\AmodB$ and 
$F \xrightarrow{\beta} F'$ in $\AmodC$ the composition map 
$${\beta \circ (-) \circ \alpha}\colon \quad \homm_\A(E,F) \rightarrow
\homm_\A(E',F')$$
in $\BmodC$ is given for every $(b,c) \in \B \otimes \Copp$ by
$$f \mapsto (-1)^{\deg(f)\deg(\alpha)} \beta \circ f \circ \alpha 
\quad\quad f \in \homm_\A(E_b,F_c).$$
The formula for the other (``as right modules'') bimodule $\homm$ functor 
\eqref{eqn-bimodule-hom-as-right-modules} is identical. 

\begin{lemma}
\label{lemma-tensor-and-hom-of-twisted-complexes} 
Let $\A$, $\B$ and $\C$ be DG-categories and let $(E_i, \alpha_{ij})$
be a twisted complex over $\AmodB$. 

\begin{enumerate}
\item
Let $(F_i, \beta_{ij})$ be a twisted complex over $\BmodC$ then
\begin{small}
\begin{align}
\label{eqn-tensor-product-of-twisted-complexes}
\bigl\{ E_i, \alpha_{ij} \bigr\}
\otimes_\B 
\bigl\{ F_i, \beta_{ij} \bigr\}
\simeq 
\bigl\{
\bigoplus_{k + l = i} E_k \otimes_\B F_l, 
\sum_{l + m = j}
(-1)^{l(k-m+1)}
\alpha_{km} \otimes \id_l 
+  
\sum_{k + n  = j} (-1)^k \id_k \otimes \beta_{ln}
\bigr\}. 
\end{align}
\end{small}
\item
\label{item-hom-complex-of-twisted-complexes}
Let $(F_i, \beta_{ij})$ be a twisted complex over $\CmodB$ then
\begin{scriptsize}
\begin{align}
\label{eqn-right-hom-complex-of-twisted-complexes}
\homm_\B 
\left(
\bigl\{ E_i, \alpha_{ij} \bigr\}
,
\bigl\{ F_i, \beta_{ij} \bigr\}
\right)
\simeq 
\bigl\{
\bigoplus_{l - k = i} \homm_\B(E_k,F_l), 
\sum_{l - m = j}
(-1)^{m(m-k)+l+1}
(-) \circ \alpha_{mk}  
+  
\sum_{n - k  = j} (-1)^{(l-n+1)k} \beta_{ln} \circ (-)
\bigr\}. 
\end{align}
\end{scriptsize}
Similarly, if $(F_i, \beta_{ij})$ is a twisted complex over $\AmodC$
then 
\begin{scriptsize}
\begin{align}
\label{eqn-left-hom-complex-of-twisted-complexes}
\homm_\A 
\left(
\bigl\{ E_i, \alpha_{ij} \bigr\}
,
\bigl\{ F_i, \beta_{ij} \bigr\}
\right)
\simeq 
\bigl\{
\bigoplus_{l - k = i} \homm_\A(E_k,F_l), 
\sum_{l - m = j}
(-1)^{m(m-k)+l+1}
(-) \circ \alpha_{mk}  
+  
\sum_{n - k  = j} (-1)^{(l-n+1)k} \beta_{ln} \circ (-)
\bigr\}. 
\end{align}
\end{scriptsize}
\end{enumerate}
\end{lemma}
\begin{proof}
\begin{enumerate}
\item 
An isomorphism of DG-modules is an isomorphism of the underlying graded  
modules which respects the differential. As a graded $\AbimC$-bimodule,
i.e. forgetting the differential, the LHS of 
\eqref{eqn-tensor-product-of-twisted-complexes} is isomorphic to 
\begin{align*}
\left( \bigoplus_{k \in \mathbb{Z}} E_k[-k] \right) 
\otimes_\B
\left( \bigoplus_{l \in \mathbb{Z}} F_l[-l] \right),
\end{align*}
while the RHS is isomorphic to 
\begin{align*}
\bigoplus_{k,l \in \mathbb{Z}}
\bigl( E_k \otimes_\B F_l[-k-l] \bigr)
\end{align*}
There is a tautological isomorphism between the two
\begin{align*}
e \otimes f \mapsto (-1)^{k' l} e \otimes f \quad \quad
\forall\; a \in \Aopp, b \in B, c \in C;\; k,l,k',l' \in \mathbb{Z};\;
e \in \left(\leftidx{_a}{\left(E_{k}\right)}{_b}\right)_{k'};\;
f \in \left(\leftidx{_b}{\left(E_{l}\right)}{_c}\right)_{l'}.
\end{align*}
which needs its sign twist to respect the
$\B$-action relations of the corresponding tensor products. 
This isomorphism can be readily seen to also respect the
differentials 
$$ d(e \otimes f) = (-1)^{k} d_{E_k}e \otimes f + \sum_{m}
\alpha_{km}(e) \otimes f + (-1)^{k + k' + l} e \otimes d_{F_l}f
+ \sum_{n} (-1)^{k + k'} e \otimes \beta_{ln}(f), $$
$$ d(e \otimes f) = (-1)^{k + l} d_{E_k}e \otimes f + 
(-1)^{k + l + k'} e \otimes d_{F_l}f +
\sum_{m} (-1)^{l(k-m+1)} \alpha_{km}(e) \otimes f + 
\sum_{n} (-1)^{k + (l-n+1)k'} e \otimes \beta_{ln}(f)$$
on the LHS and the RHS of
\eqref{eqn-tensor-product-of-twisted-complexes}. 

\item We only prove the first statement, the second is proved
identically. As a graded $\CbimA$-bimodule, 
the LHS of \eqref{eqn-right-hom-complex-of-twisted-complexes} is isomorphic to
\begin{align}
\label{eqn-right-hom-of-twisted-complexes-as-a-graded-bimodule}
\homm_\B\left(\bigoplus_k E_k[-k], \bigoplus_l F_l[-l]\right), 
\end{align}
while the RHS is isomorphic to 
\begin{align}
\label{eqn-twisted-complex-of-right-homs-as-a-graded-bimodule}
\bigoplus_{k,l} \homm_\B\left(E_k, F_l\right)[-(l-k)].
\end{align}
Since all the direct sums are finite the obvious 
natural map from
\eqref{eqn-twisted-complex-of-right-homs-as-a-graded-bimodule}
to 
\eqref{eqn-right-hom-of-twisted-complexes-as-a-graded-bimodule}
is an isomorphism of graded $\CbimA$-bimodules. 

It doesn't respect the differentials given 
for any $a \in \A$, $c \in \C$ and 
$f \in \homm^q_\B\left(\leftidx{_a}{(E_k)}{}, \leftidx{_c}{(F_l)}{}\right)$
by
\begin{align}
df = (-1)^{l} d_{\B}f + \sum_m (-1)^{q+l-k+1} f \circ \alpha_{mk} +
\sum_n \beta_{ln} 
\end{align}
on the LHS of \eqref{eqn-right-hom-complex-of-twisted-complexes} 
and by 
\begin{align}
df = (-1)^{l - k} d_{\B}f + \sum_m (-1)^{q(m-k+1) + m(m-k)+l+1 } f \circ \alpha_{mk} +
\sum_n (-1)^{(l-n+1)k} \beta_{ln} 
\end{align}
on the RHS of \eqref{eqn-right-hom-complex-of-twisted-complexes}.  
One can now readily check that the composition of 
the natural isomorphism above with the automorphism of  
\eqref{eqn-twisted-complex-of-right-homs-as-a-graded-bimodule}
which multiplies each $\homm^q_\B\left(E_k, F_l\right)$ by 
$(-1)^{(q-1)k}$ respects the differentials and thus yields
the desired isomorphism of DG bimodules. 
\end{enumerate}
\end{proof}

\begin{lemma}
\label{lemma-convolutions-and-the-dualizing-functor} 
Let $\A$ and $\B$ be DG-categories and 
let $(E_i, \alpha_{ij})$ be a twisted complex of $\AbimB$-bimodules. 
Let $E$ be its convolution $\bigl\{E_i, \alpha_{ij}\bigr\}$. 
\begin{enumerate}
\item
\label{item-dualizing-twisted-complexes}
Then 
\begin{align}
\label{eqn-right-dual-of-a-twisted-complex}
 E^\B \simeq 
\bigl\{E_{-i}^{\B}, (-1)^{j^2 + ij + 1} \alpha_{(-j)(-i)}^{\B} \bigr\} 
\\
\label{eqn-left-dual-of-a-twisted-complex}
 E^{\A} \simeq 
\bigl\{E_{-i}^{\A}, (-1)^{j^2 + ij + 1} \alpha_{(-j)(-i)}^{\A} \bigr\}
\end{align}
in $\AmodB$. 

\item 
\label{item-trace-and-action-maps-via-twisted-complexes}
The $\A$-trace map 
$E \otimes_\B E^\A \xrightarrow{\trace} \A$ is isomorphic 
to the image in $\AmodA$ of the map
\begin{align}
\label{eqn-trace-for-twisted-complexes}
\left(\bigoplus_{k-l=i} E_{k} \otimes_\B E^{\A}_{l}, 
\sum_{m-l=j} (-1)^{l(k-m+1)}\alpha_{km}\otimes \id + 
\sum_{k-n=j} (-1)^{k+n^2+nl+1}\id\otimes\alpha^{\A}_{nl}\right)
\rightarrow \A  
\end{align}
which consists of a single degree $0$ map 
$\bigoplus_{k} E_{k} \otimes_\B E^{\A}_{k}
\xrightarrow{\sum \trace} \A$.

The $\A$-action map $\A \xrightarrow{\action} \homm_{B}(E,E)$ 
is isomorphic to the image in $\AmodA$ of the map 
\begin{align}
\label{eqn-action-for-twisted-complexes}
\A \rightarrow \left(\bigoplus_{l-k=i} \homm_B\left(E_{k}, E_{l}\right), 
\sum_{l-m=j} (-1)^{m(m-k) + l + 1} (-) \circ \alpha_{mk} + 
\sum_{n-k=j} (-1)^{(l-n+1)k}\alpha_{ln} \circ (-) \right)
\end{align}
which consists of a single degree $0$ map.  
$ \A \xrightarrow{\sum \action}
\bigoplus_{k} \homm_\B \left(E_{k},E_{k}\right)$.
 
Analogous statements hold for $\B$-trace and $\B$-action maps.

\item 
\label{item-homotopy-trace-and-action-maps-via-twisted-complexes} 
Suppose each $E_i$ is $h$-projective and $\B$-perfect. 

The map $E \otimes_\B \barB \otimes_\B E^\A \xrightarrow{\trace} \barA$ is
homotopy equivalent to the image in $\AmodA$ of the map
\begin{align}
\label{eqn-homotopy-trace-for-twisted-complexes}
\left(\bigoplus_{k-l=i} E_{k} \otimes_{\B} \barB \otimes_\B E^{\A}_{l}, 
\sum_{m-l=j} (-1)^{l(k-m+1)}\alpha_{km}\otimes \id + 
\sum_{k-n=j} (-1)^{k+n^2+nl+1}\id\otimes\alpha^{\A}_{nl}\right)
\rightarrow \barA  
\end{align}
which consists of a single degree $0$ map 
$\bigoplus_{k} E_{k}\otimes_{\B} \barB \otimes_\B E^{\A}_{k}
\xrightarrow{\sum \trace}
\barA$.

The map $\barA \xrightarrow{\action} E \otimes_\B E^\B \otimes_\A \barA$ 
is homotopy equivalent to the image in $\AmodA$ of the map 
\begin{align}
\label{eqn-homotopy-action-for-twisted-complexes}
\barA \rightarrow \left(\bigoplus_{k-l=i} E_{k} \otimes_{\B}
E^{\B}_{l} \otimes_\A \barA, 
\sum_{m-l=j} (-1)^{-l(k-m+1)}\alpha_{km}\otimes \id + 
\sum_{k-n=j} (-1)^{k+n^2+nl+1}\id\otimes\alpha^{\B}_{nl}\right)
\end{align}
which consists of a single degree $0$ map 
$ 
\barA \xrightarrow{\sum \action} 
\bigoplus_{k} E_{k}\otimes_{\B} E^{\B}_{k} \otimes_\A \barA
$.
 
Analogous statements hold for $\B$-trace and $\B$-action maps when 
$E_i$ are $h$-projective and $\A$-perfect. 
\end{enumerate}
\end{lemma}
\begin{proof}
\em \eqref{item-dualizing-twisted-complexes}: \rm
Both statements follow immediately from Lemma 
\ref{lemma-tensor-and-hom-of-twisted-complexes}\eqref{item-hom-complex-of-twisted-complexes}
by setting the twisted complex $\left\{F_i, \beta_{ij}\right\}$
to be the corresponding diagonal bimodule concentrated in degree $0$. 

\em \eqref{item-trace-and-action-maps-via-twisted-complexes}: \rm
For the $\A$-trace map claim, the 
isomorphisms \eqref{eqn-left-dual-of-a-twisted-complex} and 
\eqref{eqn-tensor-product-of-twisted-complexes} compose to 
an isomorphism from $E \otimes_\B E^\A$ to the convolution of 
the LHS of \eqref{eqn-trace-for-twisted-complexes}. 
We claim that this isomorphism composes with the image of 
\eqref{eqn-trace-for-twisted-complexes} in $\AmodA$ to
give $E \otimes_\B E^{\A} \xrightarrow{\trace} \A$. 
When checking maps to be equal it suffices to only consider them 
as maps of graded modules. Thus we are reduced to checking that 
the trace map of a finite direct sum of graded modules equals 
the sum of the trace maps of the individual modules. This is
straightforward. 

For the $\A$-action map, we are similarly reduced to checking
that the action map $\A \rightarrow \homm_\B(E,E)$ is also 
compatible with finite direct sums. 

\em \eqref{item-homotopy-trace-and-action-maps-via-twisted-complexes}: \rm 
We only prove the first claim. 
The natural maps $\barA \rightarrow \A$ and $\barB \rightarrow \B$
induce isomorphisms in $D(\AbimA)$ between 
$E \otimes_\B \barB \otimes_\B E^\A \xrightarrow{\trace} \barA$
and $E \otimes_\B E^\A \xrightarrow{\trace} \A$ and also between 
\eqref{eqn-homotopy-trace-for-twisted-complexes} and 
\eqref{eqn-trace-for-twisted-complexes}. It follows from 
\eqref{item-trace-and-action-maps-via-twisted-complexes}
that $E \otimes_\B \barB \otimes_\B E^\A \xrightarrow{\trace} \barA$
and \eqref{eqn-homotopy-trace-for-twisted-complexes} are isomorphic
in $D(\AbimA)$. Since all the bimodules involved are $h$-projective 
the two are furthermore isomorphic in $H^0(\modA)$, 
as required. 
\end{proof}

\subsection{Pre-triangulated categories}
\label{section-pre-triangulated-categories}

Let $\A$ and $\B$ be two DG-categories. 
A functor $\A \xrightarrow{f} \B$ is a \em quasi-equivalence \rm 
if $f$ induces quasi-isomorphisms on morphism complexes and
if $H^0(\A) \xrightarrow{H^0(f)} H^0(\B)$ is an equivalence
of categories. 

A DG-category $\A$ is \em pretriangulated \rm if $H^0(\A)$ is a triangulated 
subcategory of $H^0(\modA)$ under the Yoneda embedding. 
The \em pretriangulated hull $\pretriag(A)$ \rm of $\A$ is the category of 
\em one-sided \rm twisted complexes in $\A$, these are the twisted
complexes $(E_i, q_{ij})$ where $q_{ij} = 0$ if $i \geq j$. 
The convolution functor gives a fully faithful embedding 
$\pretriag(\A) \hookrightarrow \modA$ whose composition with 
$\A \rightarrow \pretriag(\A)$ is the Yoneda embedding and 
whose image in $\modA$ is equivalent to $\sffg(\A)$
\cite[\S 2.4]{Drinfeld-DGQuotientsOfDGCategories}. Hence $H^0(\pretriag(\A))$ 
coincides with the triangulated hull of $H^0(\A)$ in $H^0(\modA)$. 
Therefore $\A$ is pretriangulated if and only if the natural 
embedding $\A \hookrightarrow \pretriag(\A)$ is a quasi-equivalence. We 
say that $\A$ is \em strongly pretriangulated \rm if 
$\A \hookrightarrow \pretriag(\A)$ is, in fact, an equivalence. 
In other words, if it has a quasi-inverse 
$\pretriag(\A) \xrightarrow{T} \A$.
Note, that in such case the convolution functor 
filters through the Yoneda embedding, i.e.
$$\pretriag(\A) \xrightarrow{T} \A \hookrightarrow \modA$$
is the convolution functor. For strongly pre-triangulated categories, 
by abuse of notation, 
we mainly use the term \em convolution functor \rm to mean $T$.   

Let $\A$ be any DG-category. It is known that $\pretriag(\A)$ 
is strongly pretriangulated 
\cite{BondalKapranov-EnhancedTriangulatedCategories}.
Also $\shhomm(\A,\C)$ is strongly pretriangulated for any 
strongly pretriangulated $\C$ since we can define convolutions of 
twisted complexes levelwise in $\C$. In particular, $\modA$ is 
strongly pretriangulated since $\modk$ is. Finally, 
a full subcategory of $\modA$ (or any other strongly pretriangulated 
DG-category) which is itself pretriangulated, e.g. it descends 
to a triangulated subcategory of $H^0(\modA)$, and closed under 
homotopy equivalences is strongly pretriangulated. 
Therefore $\hprojA$ and $\prfhprA$ are 
strongly pretriangulated and, for any other DG-category $\B$,
$\Aprfhpr$ and $\Bprfhpr$ are also strongly pretriangulated. 
If $\A$ itself is pretriangulated, then $\qrhprA$ and $\Aqrhpr$
are strongly pretriangulated. If, on the other hand, 
$\B$ is pretriangulated, then $\Bqrhpr$ is strongly pretriangulated. 
 
\subsection{Twisted cubes}
\label{section-twisted-cubes}
 
One of the chief technical tools we employ in this paper is 
a notion of a twisted cube over a pre-triangulated category. This 
seemingly trivial extension of a notion of a twisted complex has some  
far-reaching consequences that we exploit. To the authors' knowledge, 
the material below is original to this paper. 

We employ the following notation: let $I=\{-1,0\}^n$ enumerate vertices 
of an $n$-cube\footnote{We use $\{-1,0\}^n$ rather than $\{0,1\}^n$ as 
our indexing set since we want the arrows in the cube to go from lower 
to higher degree vertices and we want the terminal end of the cube 
to have degree $0$. This ensures that for a $1$-cube diagram, i.e. 
a single morphism, the corresponding twisted complex
coincides naturally with the cone of this morphism, with no shifts
involved.} For $\bar{i}, \bar{j}\in I$ with $\bar{i}=(i_1,\ldots, i_n)$ 
and $\bar{j}=(j_1,\ldots, j_n)$ we say that $\bar{j}>\bar{i}$ if $j_m\geq i_m$ 
for all $m$ and $\bar{i}\ne\bar{j}$. For any $\bar{i} \in I$ we denote by 
$|\bar{i}|$ its degree $\sum i_m$. 

Let $\C$ be a pre-triangulated category. A \em twisted $n$-cube over $\C$ \rm 
is 
\begin{enumerate}
\item  a set $\left\{ X_{\bar{i}} \right\}_{i \in I}$ of objects of $C$.  
\item  a set $\left\{q_{\bar{i}\bar{j}}\right\}_{\bar{i}, \bar{j} \in
I, \bar{i} < \bar{j}}$ of morphisms in $C$, such that
$q_{\bar{i}\bar{j}}$ is a morphism 
$X_{\bar{i}} \rightarrow X_{\bar{j}}$ of degree $|\bar{i}| - |\bar{j}| + 1$ 
which satisfies the relation
\begin{equation}
(-1)^{|\bar{j}|}dq_{\bar{i}\bar{j}}+\sum\limits_{\bar{i}<\bar{k}<\bar{j}} q_{\bar{k}\bar{j}}q_{\bar{i}\bar{k}}=0.
\end{equation}
\end{enumerate}

The \em total complex \rm $\tot (X_{\bar{i}}, q_{\bar{i}\bar{j}})$ 
of a twisted $n$-cube $(X_{\bar{i}}, q_{\bar{i}\bar{j}})$ 
is the one-sided twisted complex 
$$\left(
\bigoplus_{\substack{\bar{i} \in I,
\\|\bar{i}|=i}}X_{\bar{i}}, 
\sum_{\substack{\bar{i}, \bar{j} \in I, \\ |\bar{i}| = i, |\bar{j}| =j}} 
q_{\bar{i}\bar{j}}
\right) 
$$
over $\C$. Its convolution is an object of $\C$ which we call 
\em the convolution of the twisted cube $(X_{\bar{i}},
q_{\bar{i}\bar{j}})$\rm. 
 
\begin{lemma}[The Cube Lemma]
\label{lemma-the-cube-lemma} 
Let $X=\left(X_{\bar{i}}, q_{\bar{i}\bar{j}}\right)$ 
be a twisted $n$-cube indexed by $I$ 
over a pre-triangulated category $\C$. Choose $0\leq m\leq n$ and choose any 
$m$ indices in $1,\dots,n$ to define a splitting $I=J\times K$ with
$J=\{-1, 0\}^m$, $K=\{-1, 0\}^{n-m}$. Then 
\begin{enumerate}
\item Fix $\bar{k}\in K$. Then
$$ 
\left( X_{(\bar{i},\bar{k})} \right)_{\bar{i} \in J}
\quad\text{ and }\quad
\left((-1)^{|\bar{k}|} q_{(\bar{i},\bar{k})(\bar{j},\bar{k})} 
\right)_{\bar{i}, \bar{j} \in J}
$$ 
form a twisted $m$-cube indexed by $J$ over $\C$. We denote it 
by $Y^{\bar{k}}$ and call it a ``sign-twisted subcube'' of
$X$, to stress that the morphisms in $Y^{\bar{k}}$ and 
in $X$ differ (possibly) by a sign.  

\item Fix $\bar{k}, \bar{l} \in K$.   
For any $0 \leq i < j \leq m$ let 
$$p^{\bar{k}\bar{l}}_{ij} = 
\sum_{\substack{\bar{i},\bar{j}\in J,\\ |\bar{i}|=i, |\bar{j}|=j}}
q_{(\bar{i},\bar{k})(\bar{j},\bar{l})}. $$
The collection 
$\left(p^{\bar{k}\bar{l}}_{ij}\right)_{i,j}$ 
defines a morphism of twisted complexes 
$$\tot\left(Y^{\bar{k}} \right) \rightarrow 
\tot\left(Y^{\bar{l}} \right)$$
of degree $|\bar{k}|-|\bar{l}|+1$. Denote it by
$p^{\bar{k}\bar{l}}$.
\item The twisted complexes $\tot\left(Y^{\bar{k}}\right)$
and the morphisms $p^{\bar{k}\bar{l}}$ form 
a twisted $(n-m)$-cube over $\pretriag(\C)$ indexed by $K$. 
Let $Z \in \pretriag\left(\pretriag\left(\C\right)\right)$ be 
its total complex. 
\item The (double) convolution of $Z$ is isomorphic in $\C$ to the convolution 
of the original twisted cube $X$. In particular, 
it is independent of $m$ and of the choice of $I=J\times K$. 
\end{enumerate}
\end{lemma}

\begin{proof}
A straightforward verification. 
\end{proof}

Given a twisted $n$-cube $\bar{X}$ over a pre-triangulated category 
$\C$ its image in $H^0(\C)$ is an ordinary $n$-cube shaped diagram 
$X$ which commutes (up to isomorphism). Roughly, the point of 
the Cube Lemma is that $X$ can be canonically extended in $H^0(\C)$ 
to an $n$-cube $X'$ of side $2$ with the following properties\footnote{ By an $n$-cube of side $2$ we
mean an $n$-dimensional cube whose sides are two edges long, i.e. 
it's vertices are enumerated by $\left\{-1,0,1\right\}^n$ instead
of $\left\{-1,0\right\}^n$. }:
\begin{itemize}
\item The vertices of $X'$ are the convolutions of the faces of $X$. 
\item The rows and columns of $X'$ are exact triangles in $H^0(\C)$.  
\item $X'$ commutes (up to isomorphism). 
\end{itemize}
  
This is best understood by looking at some examples. 
Let $\C$ be a pre-triangulated category. 
\begin{enumerate}
\item A twisted $0$-cube over $\C$ is a single object of $C$. 
\item A twisted $1$-cube over $\C$ is a pair of 
objects $A$ and $B$ of $\C$ together with a closed morphism 
$$A \xrightarrow{f_{AB}} B$$ of degree $0$. Write 
$$a \xrightarrow{f_{ab}} b$$
for its image in $H^0(\C)$. Here we denote by $a$ and $b$, and
$f_{ab}$ the classes of $A$, $B$ and $f_{AB}$ in $H^0(\C)$.

There are no non-trivial ways to split this up as a cube of cubes, so the 
Cube Lemma doesn't tell us anything new. However, 
the total complex of this cube is, trivially, 
$$A \xrightarrow{f_{AB}} \underset{\degzero}{B}$$ 
and its convolution fits into a diagram
\begin{align}
\label{eqn-1-cube-convolution}
A \xrightarrow{f_{AB}} B 
\rightarrow 
\left\{ A \xrightarrow{f_{AB}} \underset{\degzero}{B} \right\} \rightarrow
A[1] 
\end{align}
in $\C$ where the two new morphisms are induced by 
the canonical morphisms of twisted complexes 
\begin{align}
\label{eqn-morphism-B-to-A->B}
\left(\underset{\degzero}{B}\right)
\xrightarrow{\id_{B}}
\left(A \xrightarrow{f_{AB}} \underset{\degzero}{B}\right) \\
\label{eqn-morphism-A->B-to-A-shifted-by-1}
\left(A \xrightarrow{f_{AB}} \underset{\degzero}{B}\right) 
\xrightarrow{\id_{A}}
\left(\underset{\degminusone}{A}\right). 
\end{align}
Moreover, the image of \eqref{eqn-1-cube-convolution} in $H^0(\C)$ is
precisely the exact triangle
\begin{align}
\label{eqn-1-cube-induced-exact-triangle}
a \xrightarrow{f_{ab}} b
\rightarrow \cone(f_{ab}) \rightarrow
\end{align}
which was the original point of 
\cite{BondalKapranov-EnhancedTriangulatedCategories}. 

Note that we can also complete 
$A \xrightarrow{f_{AB}} \underset{\degzero}{B}$ 
to the diagram 
\begin{align}
\label{eqn-1-cube-convolution-alternative}
A \xrightarrow{f_{AB}} B 
\rightarrow 
\left\{ A \xrightarrow{-f_{AB}} \underset{\degzero}{B} \right\} \rightarrow
A[1] 
\end{align}
whose image in $H^0(\C)$ is canonically isomorphic to
$\eqref{eqn-1-cube-induced-exact-triangle}$. The two new morphisms
in \eqref{eqn-1-cube-convolution-alternative}
are defined exactly as in \eqref{eqn-morphism-B-to-A->B} and
\eqref{eqn-morphism-A->B-to-A-shifted-by-1}.

Thus, convolving a twisted $1$-cube produces an exact
triangle in $H^0(\C)$. In the language above - 
the image of a twisted $1$-cube in $H^0(\C)$ 
is an ordinary $1$-cube and we can 
canonically complete it to a $1$-cube of side $2$ whose 
single row is an exact triangle.
It is this, together with repeated application of the Cube Lemma,
that produces the desired phenomena for twisted cubes of higher dimension. 

\item A twisted $2$-cube over $\C$ is a diagram  
\begin{align}
\label{eqn-exmpl-twisted-2-cube}
\vcenter{
\xymatrix{
A
\ar[rr]^{f_{AB}}
\ar[dd]_{f_{AC}}
\ar[rrdd]^{f_{AD}}
& &
B
\ar[dd]^{f_{BD}}
\\
\\
C
\ar[rr]_{f_{CD}}
& &
D
}
}
\end{align}
of objects and morphisms in $\C$, where
$f_{AB}$, $f_{AC}$, $f_{BD}$, $f_{CD}$ are closed maps of degree $0$
and $f_{AD}$ is a map of degree $-1$ such that
\begin{align}
\label{eqn-twisted-2-cube-condition}
- df_{AD} = f_{BD}f_{AB} + f_{CD}f_{AC}.
\end{align}
 The image of
\eqref{eqn-exmpl-twisted-2-cube} in $H^0(C)$ is the diagram 
\begin{align}
\label{eqn-exmpl-twisted-2-cube-in-H0}
\vcenter{
\xymatrix{
a
\ar[rr]^{f_{ab}}
\ar[dd]_{f_{ac}}
& &
b
\ar[dd]^{f_{bd}}
\\
\\
c
\ar[rr]_{f_{cd}}
& &
d.
}
}
\end{align}
Note that $f_{AD}$, not being necessarily closed, doesn't apriori
define a morphism in $H^0(\C)$. However the condition 
\eqref{eqn-twisted-2-cube-condition} on $f_{AD}$ ensures that 
we have $f_{bd}f_{ab} + f_{cd}f_{ac} = 0$ in $H^0(\C)$, i.e.
the diagram \eqref{eqn-exmpl-twisted-2-cube-in-H0} commutes
up to the isomorphism $(-1)\id_d$. 

The Cube Lemma tells us that 
$\left(A \xrightarrow{-f_{AB}} \underset{\degzero}{B}\right)$
and 
$\left(C \xrightarrow{f_{CD}} \underset{\degzero}{D}\right)$
are twisted $1$-cubes and that the maps 
$\left(f_{AC}, f_{AD}, f_{BD}\right)$ define a closed morphism 
$f_{ABCD}$ of degree $0$ between their convolutions producing
a twisted $1$-cube:
\begin{align}
\label{eqn-exmpl-twisted-2-cubes-horizontal-cube-of-cubes}
\left\{A\xrightarrow{-f_{AB}} \underset{\degzero}{B}\right\}
\xrightarrow{f_{ABCD}} 
\left\{C\xrightarrow{f_{CD}} \underset{\degzero}{D}\right\}.
\end{align}

Using the argument in the above section on twisted $1$-cubes we
complete \eqref{eqn-exmpl-twisted-2-cube-in-H0} to 
\begin{align}
\label{eqn-exmpl-twisted-2-cube-in-H0-completed-horizontally}
\vcenter{
\xymatrix{
a
\ar[r]^{f_{ab}}
\ar[d]_{f_{ac}}
& 
b
\ar[r]
\ar[d]^{f_{bd}}
& 
\cone(f_{ab})
\ar[r]
\ar[d]^{f_{abcd}}
&
\\
c
\ar[r]_{f_{cd}}
& 
d
\ar[r]
&  
\cone(f_{cd})
\ar[r]
&
}
}\quad 
\end{align}
We then check that each of the squares (including the third `wrap-around'
square) in this diagram commutes (up to an isomorphism). 
We can do this since we have constructed
\eqref{eqn-exmpl-twisted-2-cube-in-H0-completed-horizontally} 
as the image $H^0(\C)$ of an explicit diagram of twisted complexes
in $\pretriag(\C)$ and we can check that, in fact, that diagram itself 
commutes up to an isomorphism.  

Similarly, the Cube Lemma tells us that 
$\left(A \xrightarrow{-f_{AC}} \underset{\degzero}{C}\right)$
and 
$\left(B \xrightarrow{f_{BD}} \underset{\degzero}{D}\right)$
are twisted $1$-cubes and that the maps 
$\left(f_{AB}, f_{AD}, f_{CD}\right)$ define a closed morphism 
$f_{ACBD}$ of degree $0$ between their convolutions producing
a twisted $1$-cube:
\begin{align}
\label{eqn-exmpl-twisted-2-cubes-vertical-cube-of-cubes}
\left\{A\xrightarrow{-f_{AC}} \underset{\degzero}{C}\right\}
\xrightarrow{f_{ACBD}} 
\left\{B\xrightarrow{f_{BD}} \underset{\degzero}{D}\right\}.
\end{align}
We can therefore complete \eqref{eqn-exmpl-twisted-2-cube-in-H0} to 
\begin{align}
\label{eqn-exmpl-twisted-2-cube-in-H0-completed-vertically}
\vcenter{
\xymatrix{
a
\ar[r]^{f_{ab}}
\ar[d]_{f_{ac}}
& 
b
\ar[d]^{f_{bd}}
\\
c
\ar[r]_{f_{cd}}
\ar[d]
& 
d
\ar[d]
\\
\cone(f_{ac})
\ar[r]^{f_{acbd}}
\ar[d]
&  
\cone(f_{bd})
\ar[d]
\\
&  
}
}\quad 
\end{align}
and check that each of the squares in it commutes. 

Finally, the Cube Lemma tells us that the convolutions of the twisted $1$-cubes 
\eqref{eqn-exmpl-twisted-2-cubes-horizontal-cube-of-cubes} and
\eqref{eqn-exmpl-twisted-2-cubes-vertical-cube-of-cubes} are
both isomorphic to the convolution $T$ of the original twisted $2$-cube
\eqref{eqn-exmpl-twisted-2-cube}. We can therefore fit together  
diagrams \eqref{eqn-exmpl-twisted-2-cube-in-H0-completed-vertically}
and \eqref{eqn-exmpl-twisted-2-cube-in-H0-completed-horizontally}
and then complete them to the $2$-cube of side $2$ 
\begin{align}
\label{eqn-exmpl-twisted-2-cube-in-H0-completed-to-a-2-cube-of-side-2}
\vcenter{
\xymatrix{
a
\ar[r]^{f_{ab}}
\ar[d]_{f_{ac}}
& 
b
\ar[d]^{f_{bd}}
\ar[r]
&
\cone(f_{ab}) 
\ar[d]^{f_{abcd}}
\ar[r] 
&
\\
c
\ar[r]_{f_{cd}}
\ar[d]
& 
d
\ar[d]
\ar[r]
&
\cone(f_{cd}) 
\ar[d]
\ar[r]
&
\\
\cone(f_{ac})
\ar[r]^{f_{acbd}}
\ar[d]
&  
\cone(f_{bd})
\ar[d]
\ar[r]
&
t
\ar[r]
\ar[d]
&
\\
& & &
}
}\quad 
\end{align}
where all rows and columns are exact and where 
$$ \cone(f_{acbd}) \simeq t \simeq \cone(f_{abcd}). $$
We then check as above that every square in this diagram (including
the `wrap-around' ones) commutes up to an isomorphism.
\end{enumerate}

\begin{lemma}[The Cube Completion Lemma]
\label{lemma-the-cube-completion-lemma}
Let $I = \left\{-1,0\right\}^n$ and let $X = \left(X_{\bar{i}},
q_{\bar{i}\bar{j}}\right)$ be a twisted $n$-cube over $\C$ indexed by $I$. 
There exists a uniquely defined ``$n$-cube of side $2$'' --- 
a diagram $Z = \left\{Z_{\bar{m}}, r_{\bar{m}\bar{n}}\right\}$  
in $\C$ indexed by $M=\{-1, 0, 1\}^n$ with the following properties:
\begin{enumerate}
\item
\label{item-objects-of-2-cube-Z}
\bf Objects of $Z$. \rm 
Let $\bar{m}$ be any vertex of $M$. 
Define the splitting $I = J \times K$ by choosing for $J$ 
all the indices $\lambda \in \left\{ 1, \dots, n\right\}$ 
where $\bar{m}_\lambda$ equals $1$. Let $\bar{m}'$ be 
the restriction of $\bar{m}$ to $K$. 

The object $Z_{\bar{m}}$ is isomorphic to the convolution of 
the sign-twisted subcube $Y^{\bar{m}'}$ of $X$ constructed 
by the Cube Lemma with respect to the vertex $\bar{m}$ of $K$. 
This cube consists of all the objects $X_{\bar{i}}$ such that 
$\bar{i}$ restricts to $\bar{m}'$ in $K$ and all the morphisms 
between these vertices in $X$ multiplied by $(-1)^{|\bar{m}'|}$. 

Since $\bar{m}$ uniquely determines the twisted cube $Y^{\bar{m}'}$ 
we also refer to this cube simply as $Y^{\bar{m}}$.  

\item
\label{item-morphisms-of-2-cube-Z}
\bf Morphisms of $Z$. \rm
 Let $\bar{l} \rightarrow \bar{m} \rightarrow \bar{n}$ be any row
of $M$, i.e. for some $k \in \{ 1,\dots,n\}$ we have 
$$ 
\begin{cases}
\bar{l}_i = -1,\; \bar{m}_i = 0,\; \bar{n}_i = 1 & \quad i = k \\
\bar{l}_i = \bar{m}_i = \bar{n}_i & \quad i \neq k 
\end{cases}. 
$$

Take the sign-twisted subcube $Y^{\bar{n}}$ of $X$ and split its index 
set into $J' \times K'$ where we choose for $J'$ all the indices 
where $\bar{l}$ and $\bar{m}$ equal $1$ and for $K'$ the single 
remaining index $k$. Apply the Cube Lemma to $Y^{\bar{n}}$ 
with respect to this splitting to construct the twisted $1$-cube 
$$ \left\{Y^{\bar{l}}\right\} 
   \xrightarrow{\alpha}
   \underset{\degzero}{\left\{Y^{\bar{m}}\right\}} $$
whose convolution is $\left\{Y^{\bar{n}}\right\}$.

Then
\begin{align}
\label{eqn-a-row-of-the-completed-cube-of-side-2}
Z_{\bar{l}}
\xrightarrow{r_{\bar{l}\bar{m}}}
Z_{\bar{m}}
\xrightarrow{r_{\bar{m}\bar{n}}}
Z_{\bar{n}}
\xrightarrow{r_{\bar{n}\bar{l}}}
Z_{\bar{l}}[1]
\end{align}
is the image in $\C$ of the diagram 
\begin{align}
\label{eqn-a-row-of-the-completed-cube-of-side-2-preimage}
Y^{\bar{l}} 
\xrightarrow{\alpha}
Y^{\bar{m}} 
\rightarrow 
\left(Y^{\bar{l}} \xrightarrow{\alpha} 
\underset{\degzero}{Y^{\bar{m}}} \right)
\rightarrow
Y^{\bar{l}}[1]
\end{align}
constructed as explained in the section on the completion for twisted
$1$-cubes, cf. \eqref{eqn-1-cube-convolution}.  

\item \label{item-all-other-morphisms-of-2-cube-Z-are-0}
Any morphism in $Z$ which doesn't occur in
\eqref{eqn-a-row-of-the-completed-cube-of-side-2}
for some row $\bar{l} \rightarrow \bar{m} \rightarrow \bar{n}$
of $M$ is $0$.  

\item
\label{item-2-cube-Z-is-recursive}

\bf Recursivity. \rm
Let $I = J \times K$ be a splitting as in the Cube Lemma 
and let $Y$ be the twisted cube of sign-twisted subcubes of $X$ 
constructed by the Cube Lemma with respect to this splitting. 
Then the cube $Z_Y$ of side $2$ in $\C$ defined by $Y$ is 
naturally a subcube of $Z$.  

\item
\label{item-commutativity}

\bf Commutativity. \rm 
The image of the diagram $Z$ in $H^0(\C)$ commutes (up to isomorphism). 
\end{enumerate}
\end{lemma}
\begin{proof}
The first three properties uniquely define the diagram 
$Z = \left\{Z_{\bar{m}}, r_{\bar{m}\bar{n}}\right\}$. 
The recursivity is a straightforward verification. To prove
the commutativity of $Z$ it suffices to prove that every $2$-face 
of $Z$ commutes. This reduces via the recursivity to the case 
of $X$ being a $2$-cube, where it is again a straightforward 
verification. See the section on the completion for twisted 
$2$-cubes.
\end{proof}

\section{DG enhancements}
\label{section-DG-enhancements}

\subsection{On DG-enhancements of triangulated categories}
\label{section-DG-enhancements-of-triangulated-categories}

Let $T$ be a triangulated category. An \em enhancement \rm of $T$ is 
a pretriangulated DG-category $\A$ and an exact equivalence 
$H^0(\A) \xrightarrow{\epsilon} T$. Two enhancements $(\A, \epsilon)$ and 
$(\A', \epsilon')$ are \em equivalent \rm if there exists 
a quasi-equivalence $\A \xrightarrow{f} \A'$. 
If we want to use DG-categories as enhancements of triangulated ones, 
we are led to work in the localisation of $\DGCat$, the category of all small
DG-categories, by quasi-equivalences. We denote this localisation by 
$\HoDGCat$. For any two small DG-categories $\A$ and $\B$ denote by $[\A, \B]$ 
the set of morphisms between $\A$ and $\B$
in $\HoDGCat$. The elements of $[\A, \B]$ are called \em 
quasi-functors\rm. 

Any category quasi-equivalent to a pretrianguated category is itself 
pretriangulated. We denote the full subcategory of $\HoDGCat$
consisting of classes of pretriangulated categories by $\HoPretrCat$. 
We call the elements of $\HoPretrCat$ \em enhanced triangulated
categories \rm and think of them as of small triangulated categories
with a fixed quasi-equivalence class of DG-enhancements. Similarly, 
we can think of a quasi-functor between two enhanced triangulated 
categories as of an exact functor between the triangulated categories
and a fixed choice of a certain equivalence class of DG-functors 
between their enhancements which all descend to this exact functor.
In this sense, exact functors and quasi-functors are precisely
analogous to morphisms between cohomologies of two complexes and
morphisms between their classes in the derived category.

One way to understand the morphism set $[\A, \B]$ in $\HoDGCat$ is via 
the model category structure on $\DGCat$ constructed in 
\cite{Tabuada-UneStructureDeCategorieDeModelesDeQuillenSurLaCategorieDesDG-Categories}. The weak equivalences are the quasi-equivalences,
and the fibrations are defined in such a way that every object is fibrant. 
Therefore, the elements of $[\A, \B]$ can be identified with the functors 
from a fixed cofibrant replacement of $\A$ into $\B$, up to homotopy
relation. Moreover, there exists a cofibrant replacement functor 
$Q \colon \DGCat \rightarrow \DGCat$ equipped with a natural transformation 
$Q \rightarrow \id$ such that $Q\A \rightarrow \A$ is a
quasi-equivalence which is an identity on the sets of objects \cite[Prop.
2.3]{Toen-TheHomotopyTheoryOfDGCategoriesAndDerivedMoritaTheory}. 

The set $[\A, \B]$ can be naturally endowed with a structure of
an element of $\HoDGCat$ as follows.
The tensor product $\otimes = \otimes_k$ of elements of $\DGCat$
can be derived into a bifunctor 
$$ \ldertimes \colon \HoDGCat \times \HoDGCat \rightarrow \HoDGCat $$
giving a symmetric monoidal structure for $\HoDGCat$. We compute 
$\A \ldertimes \B$ as either $Q\A \otimes \B$ or $\A \otimes Q\B$. 
If $k$ is a field, every small $DG$-category is $k$-flat
and $\A \ldertimes \B = \A \otimes \B$. 
The monoidal structure defined by $\ldertimes$ on $\HoDGCat$ is closed
\cite[\S 4.2]{Toen-TheHomotopyTheoryOfDGCategoriesAndDerivedMoritaTheory}, 
i.e. for any $\A, \B \in \HoDGCat$ the functor $[(-) \otimes \A, \B]$
is representable by an object of $\HoDGCat$, defined up to unique isomorphism.
Denoted by $\rdershom(\A, \B)$, it is constructed
as the class in $\HoDGCat$ of $\hproj^{\Bquasirep}(\QAbimB)$ 
\cite[Thrm 6.1]{Toen-TheHomotopyTheoryOfDGCategoriesAndDerivedMoritaTheory}. 
These are the $h$-projective $\QAbimB$-bimodules $M$ where for all
$a \in Q\A$ the $\B$-module $\leftidx{_a}{M}$ is quasi-isomorphic
(and hence homotopic as $\leftidx{_a}{M}$ is $h$-projective 
\cite[Lemma 6.1(c)]{Keller-DerivingDGCategories}) 
to a representable. 
By \cite[Cor 4.8]{Toen-TheHomotopyTheoryOfDGCategoriesAndDerivedMoritaTheory}
the isomorphism classes of $H^0\left(\hproj^{\Bquasirep}(\QAbimB)\right)$ 
are in natural bijection with the elements of $[A,B]$. 
Explicitly, any element of $[\A,\B]$ can be represented by 
a functor $Q\A \rightarrow \B$. Composing this with the Yoneda
embedding $\B \rightarrow \modB$ defines a $\QAbimB$-bimodule
which is even $\B$-representable. Any $h$-projective resolution 
of it defines the desired isomorphism class in
$H^0\left(\hproj^{\Bquasirep}(\QAbimB)\right)$. 
Getting from $M \in \hproj^{\Bquasirep}(\QAbimB)$ to the
corresponding quasi-functor $f \in [\A,\B]$ is more subtle, but 
it is easy to pin down the underlying functor 
$H^0(\A) \rightarrow  H^0(\B)$. Indeed, 
$M$ defines a functor $Q\A \rightarrow \modB$ which maps 
every element of $Q\A$ to something homotopic to a representable 
element of $\modB$. This defines, up to an isomorphism,
the requisite functor $H^0(Q\A) = H^0(\A) \rightarrow H^0(\B)$. 
Indeed, this also shows that any morphism between two elements 
of $H^0\left(\hproj^{\Bquasirep}(\QAbimB)\right)$ induces a natural 
transformation between the underlying functors of the corresponding 
quasi-functors in a way which is compatible with compositions.  

In other words, $\rdershom(\A,\B) =
\hproj^{\Bquasirep}(\QAbimB)$\footnote{ If $k$ is a field, 
then $\hproj^{\Bquasirep}(\QAbimB)$ is quasi-equivalent
to $\Bqrhpr$ and we use the latter instead.} 
is, in a sense, a DG-enhancement of the set $[\A,\B]$. 
Let us therefore enrich $\HoDGCat$ to a $2$-category by setting 
the category of morphisms from $\A$ to $\B$ to be 
$H^0\left(\rdershom(\A,\B)\right)$. By above, each $1$-morphism in 
$\HoDGCat$ corresponds naturally to a quasi-functor from $\A$ to $\B$. 
By abuse of notation, we now refer to the elements of 
$H^0\left(\rdershom(\A,\B)\right)$ also as ``quasi-functors''. 
There is a natural functor
\begin{align}
\Phi\colon H^0\left(\rdershom(\A,\B)\right) \rightarrow \Fun\left(H^0(\A), H^0(\B)\right) 
\end{align}
which sends each quasi-functor to its underlying functor. 
Defining $\Phi$ depends on a choice for each quasi-representable object 
in $\modB$ of a homotopy to a representable one. A different choice would 
produce a different functor canonically isomorphic to $\Phi$. We therefore 
make a particular choice for each $\B$ and consider all functors 
$\Phi$ fixed. Our functors $\Phi$ package up into a $2$-functor
\begin{align}
\Phi\colon \HoDGCat \rightarrow \Cat
\end{align}
into a $2$-category $\Cat$ whose objects are small categories, 
whose $1$-morphisms are functors and whose $2$-morphisms are natural 
transformations. 

By above, if $\A$ and $\B$ lie in $\HoPretrCat$ then so does 
$\rder\shhomm(\A,\B)$. Therefore, in the $2$-category $\HoPretrCat$ 
the morphism categories are themselves enhanced triangulated categories. 
The $2$-functor $\Phi$ sends the triangulated category 
$H^0(\rder\shhomm(\A,\B))$ of quasi-functors to the full subcategory 
in $\Fun(H^0(\A), H^0(\B))$ consisting of exact functors. Moreover, for
any morphism of quasi-functors $\Phi$ sends its cone to a functorial 
cone of the underlying morphism of exact functors. This is exactly 
the situation we want to be in. This paper adheres to the currently 
prevalent philosophy that instead of working with triangulated categories 
$A$ and $B$ and the (non-triangulated) category $\exfun(A,B)$ of exact 
functors between them, one should work with enhancements $\A$ and $\B$ of $A$
and $B$ in $\HoDGCat$ (which are often unique up to isomorphism, cf. 
\cite{LuntsOrlov-UniquenessOfEnhancementForTriangulatedCategories}), 
the enhanced triangulated category $\rder\shhomm(\A, \B)$ and
the functor $H^0(\rder\shhomm(\A, \B)) \xrightarrow{\Phi} \exfun(A,B)$. 
For years now, this was practiced implicitly by all who work 
with Fourier-Mukai kernels of the derived functors
between algebraic varieties, cf. Examples 
\ref{example-quasi-functors-between-enhancements-of-QCoh}. 
and 
\ref{example-quasi-functors-between-enhancements-of-DBCoh}. 

\subsection{Morita enhancements}

The triangulated categories we want to enhance are the derived 
categories of quasi-coherent sheaves and the bounded derived 
categories of coherent sheaves on separated schemes of finite type 
over $k$. All these categories are Karoubi closed. It turns 
out that the full subcategory of $\HoPretrCat$ consisting of those
enhanced triangulated categories whose underlying triangulated
categories are Karoubi closed admits a more natural description.

Define a $DG$-category $\A$ to be \em kc-triangulated \rm if it is 
pre-triangulated and $H^0(\A)$ is Karoubi closed\footnote{Here ``kc''
stands for ``Karoubi closed''. These are simply called ``triangulated 
DG-categories'' in papers of To{\"e}n, however we feel that it didn't
reflect well their main difference from the established 
notion of pretriangulated DG categories.}. It follows that $\A$ 
is kc-triangulated if and only if the Yoneda embedding $\A
\hookrightarrow \prfhprA$ is a quasi-equivalence. 
Denote by $\HoKcTrCat$ the full subcategory of $\HoDGCat$ consisting
of kc-triangulated categories. The following is explained
in detail in \cite[\S 4.4]{Toen-LecturesOnDGCategories}. 
Let $\A \xrightarrow{f} \B$ be a functor between DG-categories. 
The induced functor $f_* \colon \modB \rightarrow
\modA$ preserves acyclicity. Its left adjoint 
$f^*\colon \modA \rightarrow \modB$ preserves, by adjunction, 
$h$-projectivity. 
We say that $f$ is a \em Morita equivalence \rm if 
$D(\B) \xrightarrow{f_*} D(\A)$ is an exact equivalence or,
equivalently, if $\prfhprA \xrightarrow{f^*} \prfhprB$ is a
quasi-equivalence. The functor 
$$\prfhpr(-)\colon \HoDGCat \rightarrow \HoKcTrCat$$ 
is the left adjoint of the natural inclusion 
$\HoKcTrCat \hookrightarrow \HoDGCat$ 
\cite[Prop. 6]{Toen-LecturesOnDGCategories}. It follows, as explained
in \cite[\S 4.4]{Toen-LecturesOnDGCategories}, that 
$\prfhpr(-)$ induces an equivalence 
$\Morita(\DGCat) \xrightarrow{\sim} \HoKcTrCat$, where 
$\Morita(\DGCat)$ is the localisation of $\DGCat$ by Morita
equivalences. We use this to identify Morita equivalence classes 
of small DG categories with the elements of $\HoKcTrCat$. 
In other words, when speaking of the class of a small 
DG-category $\A$ in $\HoKcTrCat$ we mean $\prfhprA$. 

We call the morphisms in $\Morita(\DGCat)$ \em Morita quasi-functors\rm. 
By above Morita quasi-functors $\A \rightarrow \B$ correspond to the
ordinary quasifunctors $\prfhprA \rightarrow \prfhprB$. It follows from 
\cite[Theorem 7.2]{Toen-TheHomotopyTheoryOfDGCategoriesAndDerivedMoritaTheory} 
that $\rder\shhomm\left(\prfhprA,\prfhprB\right)$ 
is quasi-equivalent to $\Bprfhpr$. 
This gives a more natural DG-enhancement of the set 
$\homm_{\Morita(\DGCat)}\left(\A, \B\right)$. In particular, we 
think of the elements of $D^{\Bperf}(\AbimB)$ as of Morita quasifunctors 
$\A \rightarrow \B$. Note that given $M \in D^{\Bperf}(\AbimB)$, 
the exact functor underlying the corresponding Morita quasi-functor is 
$(-) \ldertimes_\A M$. 

This leads to a slightly different notion of DG-enhancement. 
Define a \em Morita enhancement \rm of a small triangulated 
category $A$ to be a small DG-category $\A$ together 
with an isomorphism $D_c(\A) \xrightarrow{\sim} A$. Since $D_c(\A) =
H^0\left(\prfhprA\right)$, $\A$ is a Morita enhancement of $A$ if and only 
if its class in $\HoKcTrCat$ is the usual enhancement of $A$. 
Moreover, we can similarly use small DG categories to enhance 
non-small triangulated categories (i.e. unbounded
derived categories of quasi-coherent sheaves). Define a 
\em large Morita enhancement \rm of a triangulated category $A$ 
to be a small DG-category $\A$ together with an isomorphism 
$D(\A) \xrightarrow{\sim} A$. An advantage of this Morita point of view
is that we use much smaller DG categories to define our enhancements. 
In fact, the derived categories of schemes can be Morita enhanced 
by DG-algebras, cf. Examples 
\ref{example-quasi-functors-between-enhancements-of-QCoh}
and 
\ref{example-quasi-functors-between-enhancements-of-DBCoh}. 

\subsection{Examples}

The following examples illustrate the notions introduced in 
the previous section and explain the framework to which the main 
definitions of Section \ref{section-spherical-DG-functors} rightfully
belong. First is the usual framework of DG-enhancements:

\begin{exmpl}
\label{example-quasi-functors-between-general-enhanced-triag-cat}
Let $\A$ and $\B$ be two elements of $\HoDGCat$. As 
described in Section \ref{section-DG-enhancements}, 
$\rder\shhomm(\A, \B)$ is represented in $\HoDGCat$ by the full 
subcategory $\Bqrhpr$ of $\hproj(\AbimB)$
consisting of $\B$-quasi-representable bimodules. Such bimodules, 
in particular are $\B$-perfect. 

Let $M \in \Bqrhpr$. The functor 
$H^0(\A) \rightarrow H^0(\B)$ defined by the corresponding quasi-functor 
is the restriction of $(-) \ldertimes_{\A} M$ from $D(\A) \rightarrow D(\B)$ 
to $H^0(\A) \rightarrow H^0(\B)$. It follows from 
Section \ref{section-duals-and-adjoints} that 
if $M$ is also $\A$-perfect, then $(-) \ldertimes_{\A} M$, 
as a functor $D(\A) \rightarrow D(\B)$, has left and right 
adjoints $(-) \ldertimes_{\B} M^\A$ and 
$(-) \ldertimes_{\B} M^\B$. If moreover $M^\A$ and $M^\B$ are 
$\A$-quasi-representable, then these adjoints restrict
to functors $H^0(\B) \rightarrow H^0(\A)$. In other words, 
$M^\A$ and $M^\B$ define quasi-functors $\B \rightarrow \A$
whose induced functors $H^0(\B) \rightarrow H^0(\A)$
are left and right adjoint to the functor $H^0(\A) \rightarrow
H^0(\B)$ defined by $M$. 
\end{exmpl}

Next we illustrate Morita enhancements. In the two examples below 
we explain how derived categories of algebraic varieties are Morita 
enhanced by DG algebras and how the quasi-functors between these 
enhancements may be represented as DG-bimodules for these algebras:
\begin{exmpl}
\label{example-quasi-functors-between-enhancements-of-QCoh}
Let $X$ and $Y$ be two quasi-compact, quasi-separated schemes over $k$. 
By \cite[Theorem
3.1.1]{BondalVanDenBergh-GeneratorsAndRepresentabilityOfFunctorsInCommutativeAndNoncommutativeGeometry}
there exist compact generators $E_X$ and $E_Y$ of $D_{qc}(X)$ and $D_{qc}(Y)$. 
We choose $h$-injective resolutions of $E_X$ and $E_Y$ and 
define $\A$ and $\B$ to be their DG-$\eend$-algebras. 
Then $\A$ and $\B$ are the standard large Morita enhancements of 
$D_{qc}(X)$ and $D_{qc}(Y)$, i.e. 
$\hprojA$ and $\hprojB$ are their standard enhancements 
in the usual sense. 

By \cite[Theorem
7.2]{Toen-TheHomotopyTheoryOfDGCategoriesAndDerivedMoritaTheory} 
the pullback along the Yoneda embedding $\A \hookrightarrow \modA$ 
induces an isomorphism 
$$ \rder\shhomm_{\cts}\left(\hprojA, \hprojB\right)
\xrightarrow{\sim}
\rder\shhomm\left(\A, \hprojB\right) $$
in $\HoDGCat$. Here $\rder\shhomm_{\cts}$ stands for 
the full subcategory consisting of \em continuous \rm quasi-functors,
i.e. the quasi-functors $\hprojA \rightarrow \hprojB$ whose 
underlying functors $D(\A) \rightarrow D(\B)$ commute with 
infinite direct sums. 
The universal properties of $\rder\shhomm$ and
\cite[Lemma 6.2]{Toen-TheHomotopyTheoryOfDGCategoriesAndDerivedMoritaTheory} 
imply that $\rder\shhomm\left(\A, \hprojB\right)$ is represented
in $\HoDGCat$ by $\hproj(\AbimB)$. Explicitly, after replacing $\A$ 
by its cofibrant resolution any quasi-functor in 
$H^0(\rder\shhomm\left(\A, \hprojB\right))$ can be represented 
by an actual functor $\A \rightarrow \hprojB$. Taking an
$h$-projective resolution of the corresponding $\AbimB$-bimodule gives 
the desired homotopy class in $\hproj(\AbimB)$.

Thus every continuous quasi-functor $\hprojA \rightarrow \hprojB$
can be represented by an element $M \in \hproj(\AbimB)$. The underlying 
functor $D_{qc}(X) \rightarrow D_{qc}(Y)$ is then precisely 
$(-) \ldertimes_{\A} M$. 
It follows from Section \ref{section-duals-and-adjoints} that if 
$M$ is $\A$- and $\B$-perfect, then $M^\A$ and $M^\B$ define
quasi-functors $\hprojB \rightarrow \hprojA$ such 
that $(-) \ldertimes_{\B} M^\A$ and $(-) \ldertimes_{\B} M^\B$ 
are the left and right adjoints of $(-) \ldertimes_{\A} M$.

It is also shown in 
\cite[Section 8.3]{Toen-TheHomotopyTheoryOfDGCategoriesAndDerivedMoritaTheory}
that $\Aopp \otimes \B$ is the standard large Morita 
enhancement of $D(X \times_k Y)$ via
a natural identification of $D(X \times_k Y)$ with $D(\AbimB)$. 
Combined with the above we obtain an identification 
of $D(X \times_k Y)$ with 
$H^0( \rder\shhomm_{\cts}\left(\hprojA, \hprojB\right))$
which sends each object $E \in D(X \times_k Y)$
to a quasi-functor $\hprojA \rightarrow \hprojB$ whose underlying functor
$D_{qc}(X) \rightarrow D_{qc}(Y)$ is isomorphic to the Fourier-Mukai
transform defined by $E$. 
\end{exmpl}
\begin{exmpl}
\label{example-quasi-functors-between-enhancements-of-DBCoh}
Let $X$ and $Y$ be separated schemes of finite type over $k$. 
By \cite[Theorem 7.39]{Rouquier-DimensionsOfTriangulatedCategories} 
there exist strong generators $F_X$ and $F_Y$ of $D(X)$. 
Choose $h$-injective resolutions of $F_X$ and $F_Y$ and let 
$\A$ and $\B$ be their DG-$\eend$-algebras. Then $\A$ and 
$\B$ are the standard Morita enhancements of $D(X)$ and $D(Y)$, i.e. 
$\prfhprA$ and $\prfhprB$ are their standard enhancements in the usual
sense. It was, 
moreover, proved in 
\cite[Theorem 6.3]{Lunts-CategoricalResolutionOfSingularities}
that for any choice of generators $F_X$ and $F_Y$ the DG-algebras 
$\A$ and $\B$ are smooth. 

By 
\cite[Theorem 7.2]{Toen-TheHomotopyTheoryOfDGCategoriesAndDerivedMoritaTheory} 
the pullback along the Yoneda embedding $\A \hookrightarrow \prfhprA$ 
induces an isomorphism 
$$ \rder\shhomm\left(\prfhprA, \prfhprB \right)
\xrightarrow{\sim}
\rder\shhomm\left(\A, \prfhprB\right) $$
in $\HoDGCat$. Once again, the universal properties of $\rder\shhomm$ and
\cite[Lemma 6.2]{Toen-TheHomotopyTheoryOfDGCategoriesAndDerivedMoritaTheory} 
imply that $\rder\shhomm\left(\A, \prfhprB\right)$ is represented in 
$\HoDGCat$ by $\hproj^{\Bperf}(\AbimB)$, the full subcategory of 
$\hproj(\AbimB)$ consisting of $\B$-perfect bimodules. Explicitly, 
after replacing $\A$ by its cofibrant resolution any quasi-functor in 
$H^0(\rder\shhomm\left(\A, \prfhprB\right))$ can be represented 
by an actual functor $\A \rightarrow \prfhprB$. Taking any
$h$-projective
resolution of the corresponding $\B$-perfect $\AbimB$-bimodule
we obtain the desired homotopy class in $\hproj^{\Bperf}(\AbimB)$. 

Thus any quasi-functor $\prfhprA \rightarrow \prfhprB$ can be
represented by $M \in \hproj^{\Bperf}(\AbimB)$  and
the underlying functor $D(X) \rightarrow D(Y)$ is then 
$(-) \ldertimes_{\A} M$. It follows again from 
Section \ref{section-duals-and-adjoints} that if $M$ is also $\B$-perfect, 
then $M^\A$ and $M^\B$ define quasi-functors 
$\prfhprB \rightarrow \prfhprA$ such 
that $(-) \ldertimes_{\B} M^\A$ and $(-) \ldertimes_{\B} M^\B$ 
are the left and right adjoints of $(-) \ldertimes_{\A} M$.

It also follows from 
\cite[Prop. 6.14]{Lunts-CategoricalResolutionOfSingularities}
that $\Aopp \otimes \B$ is the standard Morita enhancement of 
$D(X \times Y)$. Since $\A$ is smooth, we have by Cor. 
\ref{cor-for-smooth-A-perfect-means-perfect} a natural inclusion 
$\hproj^{\Bperf}(\AbimB) \subset \prfhpr(\AbimB)$. 
This identifies each quasi-functor $\prfhprA \rightarrow \prfhprB$ 
with an object $E \in D(X \times Y)$ in such a way that
the underlying functor $D(X) \rightarrow D(Y)$
is isomorphic to the Fourier-Mukai transform defined by $E$.
\end{exmpl}

\section{Spherical DG-functors} 
\label{section-spherical-DG-functors}

\subsection{Spherical bimodules and spherical quasi-functors}
\label{section-spherical-bimodules}

Let $\A$ and $\B$ be two small DG-categories and
$S \in D(\AbimB)$ be $\A$- and $\B$-perfect. 
Denote by $R$ and $L$ the derived duals 
$S^\lderB$ and $S^\lderA$ in $D(\BbimA)$. 
Let 
$$s \colon D(\A) \rightarrow D(\B)$$
be the exact functor $(-) \ldertimes_\A S$ and 
$$r, l \colon D(\B) \rightarrow D(\A)$$
be the exact functors $(-) \ldertimes_\B S^\lderB$ and $(-)
\ldertimes_\B S^\lderA$.  By 
Cor.~\ref{cor-derived-M-and-M^A-and-M^B-adjunction} 
$r$ and $l$ are right and left adjoint to $s$.  

As per Section \ref{section-DG-enhancements} the objects of e.g. $D(\AbimB)$ 
represent continuous quasi-functors $\hprojA \rightarrow \hprojB$. 
The functors $s$, $r$ and $l$ are the exact functors underlying 
the quasi-functors $S$, $R$ and $L$. 
Accordingly, we introduce the following notation. 
Given e.g. $S \in D(\AbimB)$ and $R \in
D(\BbimA)$ we write $SR$ for the object $R \ldertimes_\A S \in
D(\BbimB)$. The exact functor underlying the quasi-functor $SR$ 
is then $sr$. 

\begin{defn}
Define:
\begin{itemize}
\item the \em twist $T$ \rm of $S$ is  
$\cone\left(SR \xrightarrow{\trace} \B\right)$ in 
$D(\BbimB)$.
\item the \em dual twist $T'$ \rm of $S$ is 
$\cone\left(\B \xrightarrow{\action} SL\right)[-1]$ 
in $D(\BbimB)$.
\item the \em cotwist $F$ \rm of $S$ is 
$\cone\left(\A \xrightarrow{\action} RS\right)[-1]$
in $D(\AbimA)$.
\item the \em dual cotwist $F'$ \rm of $S$ is 
$\cone\left(LS \xrightarrow{\trace} \A\right)$
in $D(\AbimA)$.
\end{itemize}
\end{defn}

Thus we have the following natural exact triangles in 
$D(\BbimB)$ and $D(\AbimA)$ 
\begin{align}
\label{eqn-twist-exact-triangle-DG}
SR \xrightarrow{\trace} &\B \rightarrow T, \\
\label{eqn-dual-twist-exact-triangle-DG}
T' \rightarrow & \B \xrightarrow{\action} SL, \\
\label{eqn-cotwist-exact-triangle-DG}
F \rightarrow & \A \xrightarrow{\action} RS, \\
\label{eqn-dual-cotwist-exact-triangle-DG}
LS \xrightarrow{\trace} &\A \rightarrow F'. 
\end{align}
  
Let $t,t'\colon D(\B) \rightarrow D(\B)$ and 
$f, f'\colon D(\A) \rightarrow D(\A)$ be the corresponding exact functors. 
By Cor.~\ref{cor-derived-M-and-M^A-and-M^B-adjunction} 
the functorial exact triangles of functors $D(\B) \rightarrow D(\B)$
and $D(\A) \rightarrow D(\A)$ induced by 
\eqref{eqn-twist-exact-triangle-DG}-\eqref{eqn-dual-cotwist-exact-triangle-DG} 
are
\begin{align}
\label{eqn-twist-exact-triangle}
sr \xrightarrow{\adjcounit} &\id_{D(\B)} \rightarrow t 
\\
\label{eqn-dual-twist-exact-triangle}
t' \rightarrow &\id_{D(\B)} \xrightarrow{\adjunit} sl 
\\
\label{eqn-cotwist-exact-triangle}
f \rightarrow &\id_{D(\A)} \xrightarrow{\adjunit} rs  
\\
\label{eqn-dual-cotwist-exact-triangle}
ls \xrightarrow{\adjcounit} &\id_{D(\A)} \rightarrow f', 
\end{align}
i.e. $t$ and $f[1]$ are functorial cones of the counit and the unit of the
adjoint pair $(s,r)$, while $t'[1]$ and $f'$ are functorial cones of the
unit and the counit of the adjoint pair $(l,s)$. 

Finally, consider the compositions
\begin{align}
\label{eqn-lt(-1)-to-r-DG}
LT[-1] \xrightarrow{\eqref{eqn-twist-exact-triangle-DG}}  
LSR 
\xrightarrow{\trace}
R
\\
\label{eqn-r-to-fl(1)-DG}
R \xrightarrow{\action}  
RSL 
\xrightarrow{\eqref{eqn-cotwist-exact-triangle-DG}}
FL[1]. 
\end{align}
and the induced natural transformations 
\begin{align}
\label{eqn-lt(-1)-to-r}
lt[-1] \xrightarrow{\eqref{eqn-twist-exact-triangle}}  
lsr 
\xrightarrow{\adjcounit}
r
\\
\label{eqn-r-to-fl(1)}
r \xrightarrow{\adjunit}  
rsl 
\xrightarrow{\eqref{eqn-cotwist-exact-triangle}}
fl[1]. 
\end{align}

\begin{defn}
An object $S \in D(\AbimB)$ is \em spherical \rm if 
it is $\A$- and $\B$-perfect and the following holds:
\begin{enumerate}
\item $t$ and $t'$ are quasi-inverse autoequivalences of $D(\B)$
\item $f$ and $f'$ are quasi-inverse autoequivalences of $D(\A)$
\item $lt[-1] \xrightarrow{\eqref{eqn-lt(-1)-to-r}} r$ is an
isomorphism of functors (``the twist identifies the adjoints'').  
\item $r \xrightarrow{\eqref{eqn-r-to-fl(1)}} fl[1]$ is an
isomorphism of functors (``the co-twist identifies the adjoints'').  
\end{enumerate}

We say that an $\AbimB$-bimodule is \em spherical \rm if its image in
$D(\AbimB)$ is spherical. 
\end{defn}

The following is the main theorem of this section:

\begin{theorem}
\label{theorem-main-theorem}
Let $S$ be an $\A$- and $\B$-perfect object of $D(\AbimB)$. If any 
two of the following conditions hold:
\begin{enumerate}
\item 
\label{item-main-theorem-the-twist-is-an-equivalence}
$t$ is an autoequivalence of $D(\B)$
(``the twist is an equivalence'').  
\item 
\label{item-main-theorem-the-cotwist-is-an-equivalence}
$f$ is an autoequivalence of $D(\A)$
(``the cotwist is an equivalence'').  
\item 
\label{item-main-theorem-the-twist-identifies-adjoints}
$lt[-1] \xrightarrow{\eqref{eqn-lt(-1)-to-r}} r$ is an
isomorphism of functors (``the twist identifies the adjoints'').  
\item 
\label{item-main-theorem-the-cotwist-identifies-adjoints}
$r \xrightarrow{\eqref{eqn-r-to-fl(1)}} fl[1]$ is an
isomorphism of functors (``the cotwist identifies the adjoints'').  
\end{enumerate}
then all four hold and $S$ is \em spherical\rm.
\end{theorem}

To prove this result we lift everything to the DG-enhancements 
$\hproj(\AbimA), \hproj(\BbimB), \hproj(\AbimB)$ and 
$\hproj(\BbimA)$ and work with twisted complexes over them. 
As these DG-categories are strongly pre-triangulated the canonical
convolution functors send twisted complexes over them
to (the Yoneda embeddings of) these categories themselves. 
Given e.g. a twisted complex 
$E_0 \rightarrow \dots \rightarrow E_n$ over $\hproj(\AbimA)$
we write $\left\{E_0 \rightarrow \dots \rightarrow  E_n\right\}$
for its convolution in $\hproj(\AbimA)$ 
 
Recall that $\rderhom_{\cts}(\hproj(\A), \hproj(\B))$ is represented 
in $\HoDGCatV$ by $\hproj(\AbimB)$, cf. Example
\ref{example-quasi-functors-between-enhancements-of-QCoh}. 
Similarly, \em Morita quasi-functors \rm $\A \rightarrow \B$, 
the morphisms from $\A$ to $\B$ in $\Morita(\DGCat)$, 
are in $1$-to-$1$ correspondence with ordinary
quasi-functors $\prfhprA \rightarrow \prfhprB$ and
$\rderhom(\prfhprA, \prfhprB)$ is represented in 
$\HoDGCat$ by $\hproj^{\Aperf}(\AbimB)$, cf. Example 
\ref{example-quasi-functors-between-enhancements-of-DBCoh}. 
Define a quasi-functor $\hproj(\A) \rightarrow \hproj(\B)$ or 
a Morita quasi-functor $\A \rightarrow \B$ to be 
\em spherical \rm if the corresponding element of $D(\AbimB)$ 
is spherical. 

Let $M = \barA \otimes_\A S \otimes_\B \barB$, with $S$ here viewed
as the corresponding bimodule in $\AmodB$. Then $M$ is an $h$-projective
resolution of $S$ in $\AmodB$. We now 
make use of the homotopy adjunction theory set up in
\S\ref{section-duals-and-adjoints}, and in particular of
$h$-projective resolutions $\MhdA$ and $\MhdB$ 
of $\MddA$ and $\MddB$. 

Below, we use the following shorthand:
$\tau$ denotes the map which consists of applying all possible instances
of the canonical maps $\barA \rightarrow \A$, $\barB \rightarrow \B$, 
e.g. $\MhdA \xrightarrow{\tau} M^\A$
or $M \otimes_\B \MhdB \xrightarrow{\tau} M \otimes_B M^\B$.

In the diagrams below we also use the following convention: 
the maps of degree $0$ are denoted by solid arrows and
the maps of degree $-1$ are denoted by dashed arrows. 

By Defn.~\ref{defn-homotopy-action-maps}
of the homotopy action maps, the following two diagrams 
commute up to homotopy:
\begin{align}
\label{eqn-homotopy-and-ordinary-action-maps-squares}
\vcenter{
\xymatrix{
\barB
\ar[r]^<<<<<<{\action}
\ar[d]_{\tau}
 &
\MhdA \otimes_\A M 
\ar[d]^{\eqref{eqn-otimesM^A-to-hom-M-*-map} \circ \tau}
\\
\B
\ar[r]^<<<<<{\action}
 &
\homm_\A(M,M)
}
}
\quad\quad
\vcenter{
\xymatrix{
\barA 
\ar[r]^<<<<<<{\action}
\ar[d]_{\tau}
&
M \otimes_\B \MhdB
\ar[d]^{\eqref{eqn-otimesM^B-to-hom-M-*-map} \circ\tau}
\\
\A
\ar[r]^<<<<<{\action}
&
\homm_\B(M,M).
}
}
\end{align}
Fix once and for all $\theta_\B \in
\homm^{-1}_{\BbimB}\left(\barB, \homm_\A(M,M)\right)$ and 
$\theta_\A \in \homm^{-1}_{\AbimA}\left(\barA, \homm_\B(M,M)\right)$
such that 
\begin{align*}
\eqref{eqn-otimesM^A-to-hom-M-*-map} \circ \tau \circ \action 
&= 
\action \circ \tau + d\theta_\B 
\\
\eqref{eqn-otimesM^B-to-hom-M-*-map} \circ \tau \circ \action 
&= 
\action \circ \tau + d\theta_\A
\end{align*}
i.e. the squares in \eqref{eqn-homotopy-and-ordinary-action-maps-squares}
commute up to $d\theta_\B$ and $d\theta_\A$. 

To establish our homotopy adjunctions we've proved in Prop. 
\ref{prps-M^B-and-M^A-are-homotopy-adjoints-of-M-via-S-SRS-S-maps}
that the four compositions
\eqref{eqn-mb-mb_m_mb-mb-homotopic-to-the-identity-map}-\eqref{eqn-m-m_ma_m-m-homotopic-to-the-identity-map}
are homotopic to identity. 
We can now make this more precise: let 
\begin{align*}
\chi_\A =\quad  &
M \xrightarrow{ \id \otimes \theta_\B} M \otimes_\B \homm_\A(M,M) 
\xrightarrow{\ev} \M &\in \homm^{-1}_{\AbimB}(M,M),
\\
\chi_\B = \quad  & M \xrightarrow{\theta_\A \otimes \id} \homm_\B(M,M) \otimes_\A M
\xrightarrow{\ev} M &\in \homm^{-1}_{\AbimB}(M,M), 
\\
\xi_\A = \quad  & 
\MhdA \xrightarrow{\id\!\otimes\theta_\B \otimes\id} 
\barB\!\otimes_\B\!\homm_\A(M,M)\!\otimes_\B\!M^\A\!\otimes_\A\!\barA\!
\xrightarrow{\id\!\otimes (-\circ -) \otimes \id} \MhdA \quad 
& \in \homm^{-1}_{\BbimA}(\MhdA, \MhdA),
\\
\xi_\B = \quad  & 
\MhdB \xrightarrow{\id\!\otimes\theta_\A \otimes\id} 
\barB\!\otimes_\B\!M^\B \otimes_\A \homm_\B(M,M)\!\otimes_\A\!\barA\!
\xrightarrow{\id\!\otimes (-\circ -) \otimes \id} \MhdB \quad 
& \in \homm^{-1}_{\BbimA}(\MhdB, \MhdB).
\end{align*}
The compositions 
\eqref{eqn-mb-mb_m_mb-mb-homotopic-to-the-identity-map}-\eqref{eqn-m-m_ma_m-m-homotopic-to-the-identity-map}
equal $\id + d\xi_\B$ and $\id + d\chi_\B$, $\id + d\xi_\A$ and
$\id+d\chi_\A$, respectively.

By construction, the homotopy action and trace maps are isomorphic
in $D(\AbimA)$ and $D(\BbimB)$ to their derived counterparts. 
We therefore have
\begin{align*}
T &\simeq  \convol\bigl\{\MhdB \otimes_\A M \xrightarrow{\trace} 
\underset{\degzero}{\barB} \bigr\}
\quad \text{ in } D(\BbimB),\\
T' &\simeq  \convol\bigl\{\underset{\degzero}{\barB}
\xrightarrow{\action} \MhdA \otimes_\A M \bigr\}
\quad \text{ in } D(\BbimB),\\
F &\simeq  \convol\bigl\{\underset{\degzero}{\barA} \xrightarrow{\action}
M \otimes_\B \MhdB \bigr\}
\quad \text{ in }  D(\AbimA),\\
F' &\simeq  \convol\bigl\{M \otimes_\B \MhdA \xrightarrow{\trace}
\underset{\degzero}{\barA} \bigr\}
\quad \text{ in }  D(\AbimA).
\end{align*}

\begin{prps}
\label{prop-t'-and-f'-are-left-adjoints-of-t-and-f-respectively}
We have 
$$T^{l\tilde{\B}} \simeq T' \quad \quad \text{ in } D(\BbimB)$$ 
$$(F')^{r\tilde{\A}} \simeq F \quad \quad \text{ in } D(\AbimA).$$ 
Consequently, $t'$ is the left adjoint of $t\colon D(B) \rightarrow D(B)$ 
and $f'$ is the left adjoint of $f\colon D(\A) \rightarrow D(\A)$.   
\end{prps}
\begin{proof}
By definitions of $T'$ and $T$ we have exact triangles
$$T' \rightarrow \B \xrightarrow{\action} 
\MddA \ldertimes_\A M$$
$$\MddB \ldertimes_\A M \xrightarrow{\trace} \B \rightarrow T.$$
in $D(\BbimB)$. Applying the functor $(-)^{l\tilde{\B}}$ to the latter 
one we obtain an exact triangle
$$T^{l\tilde{\B}} \rightarrow \B \xrightarrow{\trace^{l\tilde{B}}} 
(\MddB \ldertimes_\A M)^{l\tilde{\B}}.$$ 
Lemma \ref{lemma-derived-trace-and-action-maps-are-isomorphic}
produces an isomorphism 
$\MddA \ldertimes_\A M \xrightarrow{\sim} (\MddB \ldertimes_\A M)^{l\B}$
which makes the diagram
\begin{align*}
\xymatrix{
\B 
\ar[r]^{\action}
\ar@{=}[d]
&
\MddA \ldertimes_\A
\ar[d]^{\sim}
\\
\B
\ar[r]^<<<<<<{\trace^{l\tilde{B}}}
&
(\MddB \ldertimes_\A M)^{l\B}
}
\end{align*}
commute. Thus there exists $T' \simeq T^{l\tilde{\B}}$ 
which completes the above to an isomorphism of exact triangles. 

An identical argument produces an isomorphism 
$(F')^{r\tilde{\A}} \simeq F$ in $D(\AbimA)$. 
The final assertion then follows since by 
Cor. \ref{cor-derived-M-and-M^A-and-M^B-adjunction} the functors
$(-) \ldertimes_\B T^{l\tilde{\B}}$ and $(-) \ldertimes_\A ((F')^{r\tilde{\A}}$
are left and right adjoint to $t$ and $f'$, respectively. 
\end{proof}
Thus, if $t$ is an auto-equivalence of $D(\B)$ then $t'$ is always
its quasi-inverse, and similarly for $f$ and $f'$. 

Denote by 
$\B \xrightarrow{\action} TT'$ and $T'T \xrightarrow{\trace} \B$
the maps in $D(\BbimB)$ which the isomorphism $T' \simeq T^{l\tilde{\B}}$
of Prop. \ref{prop-t'-and-f'-are-left-adjoints-of-t-and-f-respectively}
identifies with the derived action and trace maps for $T$. 
By construction of the $(t',t)$ adjunction these maps induce its unit and co-unit.

\begin{prps}
\label{prps-unit-and-counit-of-t't-via-twisted-complexes}
The maps
$\B \xrightarrow{\action} TT'$ and $T'T \xrightarrow{\trace} \B$ 
are isomorphic in 
$D(\BbimB)$ to the maps 
\begin{align}\label{eqn-adjunction-unit-t't}
\underset{\degzero}{\barB} \rightarrow 
\Bigl(\MhdB \otimes_\A M \xrightarrow{ \trace \oplus (\action\otimes\id)} 
\underset{\degzero}{\barB \oplus
\bigl(\MhdA \otimes_\A M \otimes_\B \MhdB \otimes_\A M\bigr)}
\xrightarrow{\action \oplus(- \id\otimes\trace)}
\MhdA \otimes_\A M \Bigr) \\
\label{eqn-adjunction-counit-t't}
\Bigl(\MhdB \otimes_\A M \xrightarrow{\trace 
\oplus(- \id \otimes \action)} 
\underset{\degzero}{\barB \oplus \bigl(\MhdB
\otimes_\A M\otimes_\B \MhdA \otimes_\A M\bigr)}
\xrightarrow{\action\oplus (\trace\otimes\id)}
\MhdA \otimes_\A M \Bigr) 
\rightarrow \underset{\degzero}{\barB}
\end{align}
of twisted complexes over $\BmodB$ given, respectively, by 
\begin{align}
\label{eqn-adjunction-unit-tt'-morphism}
&\barB 
\xrightarrow{\id \oplus \action}
\barB \oplus \bigl(\MhdA \otimes_\A M \bigr) 
\xrightarrow{\id \oplus (\id \otimes \action \otimes \id)}
\barB \oplus \bigl(\MhdA \otimes_\A M \otimes_\B \MhdB \otimes_\A M \bigr) 
\\
\label{eqn-adjunction-unit-tt'-morphism-deg-minus-1}
&\barB \xrightarrow{- (\id \otimes \chi_\B)\circ\action} \MhdA \otimes_\A M
\end{align}
and 
\begin{align}
\label{eqn-adjunction-counit-t't-morphism}
\barB \oplus \bigl(\MhdB \otimes_\A M \otimes_\B \MhdA \otimes_\A M \bigr) 
\xrightarrow{\id \oplus (\id \otimes \trace \otimes \id)}
\barB \oplus \bigl(M^\B \otimes_\A M \bigr) 
\xrightarrow{\id \oplus \trace}
&\barB
\\
\label{eqn-adjunction-counit-t't-morphism-deg-minus-1}
\MhdB \otimes_\A M
\xrightarrow{- \trace\circ(\id \otimes \chi_\A)} 
&\barB.
\end{align}
\end{prps}
\begin{proof}
We treat the case of the adjunction unit, the case of the counit is 
treated identically. 
It suffices to show that \eqref{eqn-adjunction-unit-t't} is isomorphic
in $D(\BbimB)$ to $\B \xrightarrow{\action} TT^{l\tilde{\B}}$. 
The latter is isomorphic to  
\begin{align}
\label{eqn-genuine-action-map-for-T}
\B \xrightarrow{\action} \homm_{l\B}
\left(\left\{M^\B \otimes_\A M \xrightarrow{\trace}
\underset{\degzero}{\B}\right\}, 
\left\{M^\B \otimes_\A M \xrightarrow{\trace} \underset{\degzero}{\B}\right\}
\right)
\end{align}
since $\left\{M^\B \otimes_\A M \xrightarrow{\trace}
\underset{\degzero}{\B}\right\}$ is
a left $\B$-$h$-projective bimodule homotopically equivalent to $T$.  

By the commutativity of \eqref{eqn-dual-trace-action-map-commutative-square}, 
the composition of $\B$-action map
with the quasi-isomorphism 
\begin{align}
\label{eqn-hom-M-M-to-M^B-otimes-M-dual}
\homm_{\A}(M,M) \xrightarrow{ }
\homm_{A}(M, M^{\B\B}) 
\xrightarrow{\text{adjunction}}
\left(M^{\B} \otimes_\A M\right)^{l\B} 
\end{align}
is the left dual of the $\B$-trace map. 
The following is a chain of quasi-isomorphisms of twisted complexes:
\begin{tiny}
\begin{align}
\label{eqn-reduction-of-TT'-to-Hom-T-T^lB}
\vcenter{
\xymatrix{
\MhdB\!\otimes_\A\!M
\ar[d]_{\tau}
\ar@{-->}[drr]^{0 \oplus (\theta_\B \otimes \tau)}
\ar[rr]^<<<<<<<<<<<<<<<<{\trace \oplus
(\action\otimes\id)} 
&&
\overset{\degzero}{\barB \oplus
\bigl(\MhdA\!\otimes_\A\!M\!\otimes_\B\!\MhdB\!\otimes_\A\!M\bigr)}
\ar[d]^{\tau \oplus (\ev\otimes\id) \circ \tau}
\ar@{-->}[drr]^{- \theta_\B \oplus 0}
\ar[rr]^<<<<<<<<<<<<<<<{\action \oplus(- \id\otimes\trace)}
&&
\MhdA\!\otimes_\A\!M
\ar[d]^{\ev \circ \tau}
\\
M^\B\!\otimes_\A\!M
\ar[d]_{\id}
\ar[rr]^<<<<<<<<<<<<<<<<{\trace \oplus (\action\otimes\id)} 
&&
\B \oplus
\bigl(\homm_\A(M,M)\!\otimes_\B\!M^\B\!\otimes_\A\!M\bigr)
\ar[d]^{\id \oplus \eqref{eqn-hom-M-M-to-M^B-otimes-M-dual}}
\ar[rr]^<<<<<<<<<<<<<<<{\action \oplus(- \id\otimes\trace)}
&&
\homm_\A(M,M)
\ar[d]^{\eqref{eqn-hom-M-M-to-M^B-otimes-M-dual}}
\\
M^\B\!\otimes_\A\!M
\ar[rr]^<<<<<<<<<<<<<<<<{\trace \oplus
(\trace^{l\B}\otimes\id)} 
\ar@{~}@<0.5ex>[d]
\ar@{-}@<-0.5ex>[d]
&&
\B \oplus
\bigl(\left(M^\B\!\otimes_\A\!M\right)^{l\B}\!\otimes_\B\!M^\B\!\otimes_\A\!M\bigr)
\ar[rr]^<<<<<<<<<<<<<<<{\trace^{l\B} \oplus(- \id\otimes\trace)}
\ar[d]^{\action \oplus \ev}
&&
\left(M^\B\!\otimes_\A\!M\right)^{l\B}
\ar@{=}[d]
\\
\homm_{l\B}\left(\B, M^\B\!\otimes_\A\!M)\right)
\ar[rr]^<<<<<<<<<{\trace \circ (-) \oplus
(-) \circ \trace } 
&&
\homm_{l\B}(\B,\B) \oplus
\homm_{l\B}\bigl(M^\B\!\otimes_\A\!M, M^\B\!\otimes_\A\!M\bigr)
\ar[rr]^<<<<<<<<<<{(-)\circ\trace \oplus - \trace\circ(-)}
&&
\homm_{l\B}\left(M^\B\!\otimes_\A\!M,\B\right)
}
}
\end{align}
\end{tiny}

By Lemma
\ref{lemma-convolutions-and-the-dualizing-functor}\eqref{item-trace-and-action-maps-via-twisted-complexes} the map \eqref{eqn-genuine-action-map-for-T}
is isomorphic to the map 
\begin{tiny}
\begin{align}
\label{eqn-genuine-action-map-for-T-expanded}
\vcenter{
\xymatrix{
&&
\B
\ar[d]^{\action \oplus \action}
&&
\\
\left\{
\homm_{l\B}\left(\B, M^\B\!\otimes_\A\!M)\right)
\ar[rr]^<<<<<<<<<{\trace \circ (-) \oplus
(-) \circ \trace } 
\right.
&&
\homm_{l\B}(\B,\B) \oplus
\homm_{l\B}\bigl(M^\B\!\otimes_\A\!M, M^\B\!\otimes_\A\!M\bigr)
\ar[rr]^<<<<<<<<<<{(-)\circ\trace \oplus - \trace\circ(-)}
&&
\left.
\homm_{l\B}\left(M^\B\!\otimes_\A\!M,\B\right)
\right\}
}
}
\end{align}
\end{tiny}
To show that $\eqref{eqn-adjunction-unit-t't}$ is isomorphic
in $D(\BbimB)$ to \eqref{eqn-genuine-action-map-for-T-expanded}, 
it now suffices to show that 
$\eqref{eqn-reduction-of-TT'-to-Hom-T-T^lB} 
\circ \eqref{eqn-adjunction-unit-t't}$
is homotopic to $\eqref{eqn-genuine-action-map-for-T-expanded} \circ \tau$. 
It is a routine check of the kind we normally leave to the reader, but 
we write it out in detail once 
to give the flavor of the computations involved. 

The composition of \eqref{eqn-adjunction-unit-t't} 
with \eqref{eqn-reduction-of-TT'-to-Hom-T-T^lB} is the map 
\begin{tiny}
\begin{align}
\label{eqn-composition-of-adjunction-unit-t't-with-chain-of-quasi-isos}
\vcenter{
\xymatrix{
&&
\barB
\ar[d]_{\bigl(\action \circ \tau\bigr) \oplus \alpha_0}
\ar@{-->}[drr]+<+1ex,+2.5ex>^{\alpha_1}
&&
\\
\left\{
\homm_{l\B}\left(\B, M^\B\!\otimes_\A\!M)\right)
\ar[rr]^<<<<<<<<<{\trace \circ (-) \oplus
(-) \circ \trace } 
\right.
&&
\homm_{l\B}(\B,\B) \oplus
\homm_{l\B}\bigl(M^\B\!\otimes_\A\!M, M^\B\!\otimes_\A\!M\bigr)
\ar[rr]^<<<<<<<<<<{(-)\circ\trace \oplus - \trace\circ(-)}
&&
\left.
\homm_{l\B}\left(M^\B\!\otimes_\A\!M,\B\right)
\right\}
}
}
\end{align}
\end{tiny}
where 
\begin{tiny}
\begin{align*}
\alpha_1 =
\left(\barB \xrightarrow{- \theta_B} \homm_\A(M,M) 
\xrightarrow{\eqref{eqn-hom-M-M-to-M^B-otimes-M-dual}}
\homm_{l\B}(M^\B\!\otimes_\A\!M,\B)\right)  
+ 
\left(\barB \xrightarrow{\action} \MhdA \otimes_\A M 
\xrightarrow{- \id \otimes \chi_\B} \MhdA \otimes_\A M
\xrightarrow{\eqref{eqn-hom-M-M-to-M^B-otimes-M-dual}\circ\ev\circ\tau}
\homm_{l\B}(M^\B\!\otimes_\A\!M,\B)\right) 
\end{align*}
\end{tiny}
and the composition $\alpha_0$ can be computed 
by considering the following diagram
\begin{tiny}
\begin{align*}
\vcenter{
\xymatrix{
\barB 
\ar[r]^<<<<<<{\action}
\ar[ddd]_{\tau}
\ar@{}[dddr]|*+<10pt>[o][F-]{\text{\Large{A}}}
&
\MhdA\!\otimes_\A\!M 
\ar@{}[drr]|*+<10pt>[o][F-]{\text{\Large{B}}}
\ar[rr]^{\id\!\otimes\!\action\!\otimes\!\id}
\ar[d]_{\tau}
& & 
\MhdA\!\otimes_\A\!M\!\otimes_\B\MhdB\!\otimes_\A\!M 
\ar[d]_{\tau}
&
\\
& 
M^\A\!\otimes_\A\!M
\ar[dd]_{\ev}
\ar[r]^<<<<<<<{\id\!\otimes\!\action\!\otimes\!\id}
\ar[dr]_{\id\!\otimes\!\action\!\otimes\!\id\quad}
&
M^\A\!\otimes_\A\homm_\B(M,M)\!\otimes_\A\!M
\ar@{~>}[d]^{\id \otimes (-)^\B \otimes \id}
&
M^\A\!\otimes_\A\!M\!\otimes_\B\!M^\B\!\otimes_\A\!M
\ar[r]^<<<<<{\ev\!\otimes\!\id}
\ar[l]_<<<<<<{\id\!\otimes\!\ev\!\otimes\!\id}
\ar[d]^{\id \otimes \eqref{eqn-double-B-dual-transformation}\otimes \id}
&
\homm_\A(M,M)\!\otimes_\B\!M^\B\!\otimes_\A\!M
\ar[d]_{\bigl(\eqref{eqn-double-B-dual-transformation}\circ(-)\bigr)
\otimes \id}
\\
& 
&
M^\A\!\otimes_\A\!\homm_\B(M^\B,M^\B)\!\otimes_\A\!M
\ar@{~>}[d]^{\ev \circ \bigl(\id \otimes \eqref{eqn-tensor-product-and-hom-map} \bigr)}
&
M^\A\!\otimes_\A\!M^{\B\B}\!\otimes_\B\!M^\B\!\otimes_\A\!M
\ar[r]^<<<<{\ev\!\otimes\!\id}
\ar[l]_<<<<<{\id\!\otimes\!\ev\!\otimes\!\id}
&
\homm_{\B}(M,M^{\B\B})\!\otimes_\B\!M^\B\!\otimes_\A\!M
\ar[d]_{\ev \circ \bigl(\text{adjunction} \otimes \id\bigr)}
\\
\B
\ar[r]^<<<<{\action}
&
\homm_\A(M,M)
\ar@{~>}[r]^<<<<{\eqref{eqn-adjunction-unit-for-M-otimes-(-)} \circ (-)}
&
\homm_\A(M,\homm_{\B}(M^\B,M^\B\!\otimes\!M))
\ar@{~>}[rr]^{\text{adjunction}}
&
&
\homm_{l\B}(M^\B\!\otimes_\A\!M,M^\B\!\otimes_\A\!M)).
}
} 
\end{align*}
\end{tiny}
This diagram commutes except for the sections marked $(A)$ and $(B)$. These 
commute up to $d\theta_\B$ and $\tau \otimes d\theta_\A \otimes \id$,
respectively. The upper right border of this diagram
composes to $\alpha_0$, while its bottom line composes
to $\B \xrightarrow{\action}
\homm_{l\B}(M^\B\!\otimes_\A\!M,M^\B\!\otimes\!M))$. It follows 
that 
$$ 
\alpha_0 =
\action \circ\; \tau  
+ 
d\bigl(\beta_1 \circ \theta_\B
+ 
\beta_2 \circ \left(\tau \otimes \theta_\A \otimes \id \right)
\circ \action \bigr)
$$
where $\beta_1$ and $\beta_2$ are the corresponding compositions of 
the wavy arrows in the diagram. 

Thus $\eqref{eqn-composition-of-adjunction-unit-t't-with-chain-of-quasi-isos}$
is the sum of
$\eqref{eqn-genuine-action-map-for-T-expanded} \circ \tau$
and the map 
\begin{tiny}
\begin{align}
\label{eqn-null-homotopic-part-of-composition-of-adjunction-unit-t't-with-chain-of-quasi-isos}
\vcenter{
\xymatrix{
&&
\barB
\ar[d]_{0 \oplus d\bigl(\beta_1 \circ \theta_\B + 
\beta_2 \circ \left(\tau \otimes \theta_\A \otimes \id \right)
\circ \action \bigr)
  }
\ar@{-->}[drr]+<+1ex,+2.5ex>^{\alpha_1}
&&
\\
\left\{
\homm_{l\B}\left(\B, M^\B\!\otimes_\A\!M)\right)
\ar[rr]^<<<<<<<<<{\trace \circ (-) \oplus
(-) \circ \trace } 
\right.
&&
\homm_{l\B}(\B,\B) \oplus
\homm_{l\B}\bigl(M^\B\!\otimes_\A\!M, M^\B\!\otimes_\A\!M\bigr)
\ar[rr]^<<<<<<<<<<{(-)\circ\trace \oplus - \trace\circ(-)}
&&
\left.
\homm_{l\B}\left(M^\B\!\otimes_\A\!M,\B\right)
\right\}
}
}
\end{align}
\end{tiny}
and it remains to show that
\eqref{eqn-null-homotopic-part-of-composition-of-adjunction-unit-t't-with-chain-of-quasi-isos} is a boundary.

It suffices to show that 
$$\alpha_1 = - \bigl( \trace \circ (-) \bigr) \circ 
\bigl(\beta_1 \circ \theta_\B + 
\beta_2 \circ \left(\tau \otimes \theta_\A \otimes \id \right)
\circ \action \bigr).$$ 
By definition of $\alpha_1$ and $\chi_\B$, this would follow from 
$$ \eqref{eqn-hom-M-M-to-M^B-otimes-M-dual} \circ \theta_\B 
= \bigl( \trace \circ (-) \bigr) \circ \beta_1 \circ \theta_\B $$
$$ \eqref{eqn-hom-M-M-to-M^B-otimes-M-dual} \circ \ev \circ 
(\id \otimes \ev) \circ (\tau \otimes \theta_\A \otimes \id) \circ \action = 
\bigl( \trace \circ (-) \bigr) \circ \beta_2 \circ 
\bigl(\tau \otimes \theta_\A \otimes \id\bigr) \circ \action. $$
In fact, a stronger statement is true:
$\eqref{eqn-hom-M-M-to-M^B-otimes-M-dual} = 
\bigl( \trace \circ (-) \bigr) \circ \beta_1 $
and 
$\bigl( \trace \circ (-) \bigr) \circ \beta_1 \circ \ev \circ (\id \otimes \ev) = \bigl( \trace \circ (-) \bigr) \circ \beta_2$. 
It is equivalent to the commutativity of the following two diagrams
\begin{tiny}
\begin{align*}
\vcenter{
\xymatrix{
\homm_\A(M,M) 
\ar[drr]_{\eqref{eqn-double-B-dual-transformation}\circ(-)\quad}
\ar[rr]^{\eqref{eqn-adjunction-unit-for-M-otimes-(-)} \circ (-)\quad}
&&
\homm_\A(M,\homm_{\B}(M^\B,M^\B\!\otimes\!M))
\ar[rr]^{\text{adjunction}}
\ar[d]^{\bigl(\trace \circ (-) \bigr) \circ (-)}
&&
\homm_{l\B}(M^\B\!\otimes_\A\!M,M^\B\!\otimes_\A\!M)).
\ar[d]^{\trace \circ (-)}
\\
&&
\homm_\A\left(M,M^{\B\B}\right)
\ar[rr]^{\text{adjunction}}
&&
\homm_\A\left(M^\B\!\otimes_\A\!M,B\right)
}
\xymatrix{
M^\A\!\otimes_\A\homm_\B(M,M)\!\otimes_\A\!M
\ar[rr]^{\id \otimes (-)^\B \otimes \id\quad}
\ar[d]^{\id \otimes \ev}
\ar[drr]^{\ev}
&&
M^\A\!\otimes_\A\!\homm_\B(M^\B,M^\B)\!\otimes_\A\!M
\ar[rr]^{\ev}
&&
\homm_\A\left(M,\homm_\B(M^\B,M^\B)\!\otimes_\A\!M\right)
\ar[d]^{\eqref{eqn-tensor-product-and-hom-map} \circ (-)}
\\
M^\A\!\otimes_\A M
\ar[d]^{\ev}
&&
\homm_\A\left(M,\homm_\B(M,M)\!\otimes_\A\!M\right)
\ar[lld]_{\ev \circ (-)}
\ar[urr]^<<<<<<<<<{\bigl((-)^\B \otimes \id \bigr) \circ (-)}
&&
\homm_\A(M,\homm_{\B}(M^\B,M^\B\!\otimes\!M))
\ar[d]^{\bigl(\trace \circ (-) \bigr) \circ (-)}
\\
\homm_\A(M,M)
\ar[rr]^<<<<<<<<<<<<<<<{\eqref{eqn-adjunction-unit-for-M-otimes-(-)} \circ (-)}
&&
\homm_\A(M,\homm_{\B}(M^\B,M^\B\!\otimes\!M))
\ar[rr]^{\bigl(\trace \circ (-) \bigr) \circ (-)}
&&
\homm_\A(M,M^{\B\B})
}
}
\end{align*}
\end{tiny}
which is readily checked. 
\end{proof}

Let $\A \xrightarrow{\action} FF'$ and $F'F \xrightarrow{\trace} \A$
be the maps in $D(\AbimA)$ which the isomorphism $(F')^{r\tilde{\A}} \simeq F$
of Prop. \ref{prop-t'-and-f'-are-left-adjoints-of-t-and-f-respectively}
identifies with the derived action and trace maps for $F$.
The following proposition is proved in the same way:

\begin{prps}
\label{prps-unit-and-counit-of-f'f-via-twisted-complexes}
The maps
$\A \xrightarrow{\action} FF'$ and $F'F \xrightarrow{\trace} \A$ 
are isomorphic in $D(\AbimA)$ to the maps 
\begin{align}\label{eqn-adjunction-unit-f'f}
\underset{\degzero}{\barA} \rightarrow 
\Bigl(M\otimes_{\B} \MhdA \xrightarrow{\trace \oplus (-\id\otimes\action)} 
\underset{\degzero}{\barA \oplus \bigl(M\otimes_{B} \MhdA \otimes_\A
M\otimes_\B \MhdB \bigr)}
 \xrightarrow{\action\oplus(\trace\otimes\id)}
M\otimes_{\B} \MhdB \Bigr) \\
\label{eqn-adjunction-counit-f'f}
\Bigl(M\otimes_{\B} \MhdA \xrightarrow{\trace\oplus(\action\otimes\id)} 
\underset{\degzero}{\barA \oplus \bigl(M\otimes_{\B} \MhdB \otimes_\A
M\otimes_\B \MhdA \bigr)}
\xrightarrow{\action\oplus(-\id\otimes\trace)}
M\otimes_{\B} \MhdB \Bigr) 
\rightarrow \underset{\degzero}{\barA}
\end{align}
of twisted complexes over $\AmodA$ given, respectively, by 
\begin{align}
\label{eqn-adjunction-unit-f'f-morphism}
&\barA 
\xrightarrow{\id \oplus \action}
\barA \oplus \bigl(M \otimes_\B \MhdB \bigr) 
\xrightarrow{\id \oplus -(\id \otimes \action \otimes \id)}
\barA \oplus \bigl(M \otimes_\B \MhdA \otimes_\A M \otimes_\B \MhdB
\bigr) \\
\label{eqn-adjunction-unit-f'f-morphism-deg-minus-1}
&\barA \xrightarrow{- (\chi_\A \otimes \id) \circ \action} M \otimes_\B \MhdB 
\end{align}
\begin{align}
\label{eqn-adjunction-counit-f'f-morphism}
\barA \oplus \bigl(M \otimes_\B \MhdB \otimes_\A M \otimes_\B \MhdA \bigr) 
\xrightarrow{\id \oplus -(\id \otimes \trace \otimes \id)}
\barA \oplus \bigl(M \otimes_\B \MhdA \bigr) 
\xrightarrow{\id \oplus \trace}
&\barA
\\
\label{eqn-adjunction-counit-f'f-morphism-deg-minus-1}
M\otimes_{\B} \MhdA \xrightarrow{- \trace \circ (\chi_\B \otimes \id)} &\barA. 
\end{align}
\end{prps}

Consider the twisted $2$-cube over $\BmodA$
\begin{align}
\label{eqn-M^A-M-M^B-diagram-small}
\vcenter{
\xymatrix{
0 
\ar[r] 
\ar[d]
&
\MhdA
\ar[d]^{\id \otimes \action}
\\
\MhdB
\ar[r]_<<<<{\action \otimes \id}
&
\MhdA \otimes_{\A} M \otimes_{\B} \MhdB
} 
}.
\end{align}

By the Cube Lemma 
(Lemma \ref{lemma-the-cube-lemma}) the convolutions of 
the rows of \eqref{eqn-M^A-M-M^B-diagram-small} fit into a $1$-cube 
(i.e. a single morphism) whose convolution is the convolution of 
the total complex of the $2$-cube. And similarly
for the convolutions of the columns of \eqref{eqn-M^A-M-M^B-diagram-small}. 
This is formalised in the Cube Completion Lemma (Lemma
\ref{lemma-the-cube-completion-lemma}) which constructs for
us the diagram 
\begin{align}
\label{eqn-M^A-M-M^B-diagram-big}
\vcenter{
\xymatrix{
0 
\ar[r] 
\ar[d]
&
\MhdA
\ar[r]
\ar[d]^{\id \otimes \action}
&
\MhdA
\ar[d]
\ar[r]^{[1]}
&
\\
\MhdB
\ar[r]_<<<<<<<{\action \otimes \id}
\ar[d]
&
\MhdA \otimes_{\A} M \otimes_{\B} \MhdB 
\ar[r]
\ar[d]
&
\left\{
\MhdB \rightarrow \underset{\degzero}{\MhdA \otimes M \otimes \MhdB}
\right\}
\ar[d]
\ar[r]^<<<<<<{[1]}
&
\\
\MhdB
\ar[r]
\ar[d]^{[1]}
&
\left\{
\MhdA \rightarrow \underset{\degzero}{\MhdA \otimes M \otimes \MhdB}
\right\}
\ar[r]
\ar[d]^{[1]}
&
\left\{
\MhdA \oplus \MhdB
\rightarrow \underset{\degzero}{\MhdA \otimes M \otimes \MhdB}
\right\}
\ar[r]^<<<<{[1]}
\ar[d]^{[1]}
&
\\
& 
& 
} 
}
\end{align}
in $\BmodA$. The morphisms marked $[1]$ are morphisms of degree $1$ which 
``wrap around'' to the beginning of the corresponding row or column. 
We haven't labeled all the maps within twisted complexes
or the morphisms between their convolutions in
\eqref{eqn-M^A-M-M^B-diagram-big}, but the precise formulas can be found 
in Lemma \ref{lemma-the-cube-completion-lemma}. 

Let now $Q$ be the convolution of 
the $2$-cube \eqref{eqn-M^A-M-M^B-diagram-small} 
shifted by one to the right, that is
\begin{align*}
Q & \overset{\text{def}}{=}
\left\{
\underset{\degzero}{\MhdA \oplus \MhdB}
\xrightarrow{- \id \otimes \action - \action \otimes \id} 
\MhdA \otimes_\A M \otimes_\B \MhdB
\right\} \simeq  \\
&\simeq\; \cone 
\left(R \oplus L \xrightarrow{\id\!\otimes\!\action + \action\!\otimes\!\id } 
RSL \right)[-1]
 \quad \text{ in } D(\BbimA). 
\end{align*}
The diagram \eqref{eqn-M^A-M-M^B-diagram-big} 
descends to a commutative $3\times3$-diagram 
in $D(\BbimA)$ whose rows and columns are exact:
\begin{align}
\label{eqn-rsl-3x3-diagram}
\vcenter{
\xymatrix{
0 
\ar[r] 
\ar[d]
&
L
\ar@{=}[r]
\ar[d]^{(\action)L}
&
L
\ar[d]^{\alpha'}
\\
R
\ar[r]^<<<<<<{R(\action)}
\ar@{=}[d]
&
RSL 
\ar[r]^{\zeta}
\ar[d]^{\eta}
&
RT'[1]
\ar[d]
\\
R
\ar[r]^<<<<<{\alpha}
&
FL[1]
\ar[r]
&
Q[1].
} 
}
\end{align}
The connecting morphisms for the exact triangles are 
the images of the morphisms labeled $[1]$ 
in \eqref{eqn-M^A-M-M^B-diagram-big}. 

\begin{lemma}
\label{lemma-condition-cotwist-identifies-adjoints-is-equiv} 
The following are equivalent:
\begin{itemize}
\item $r \xrightarrow{\eqref{eqn-r-to-fl(1)}} fl[1]$ is an isomorphism
(the condition \eqref{item-main-theorem-the-cotwist-identifies-adjoints} of 
Theorem \ref{theorem-main-theorem}).
\item $r \oplus l \xrightarrow{r(\text{unit}) \oplus \text{unit}} rsl$
is an isomorphism.
\item $Q \simeq 0$ in $D(\BbimA)$.
\item $\alpha$ is an isomorphism in $D(\BbimA)$.
\item $\alpha'$ is an isomorphism in $D(\BbimA)$.
\end{itemize}
\end{lemma}
\begin{proof}
Denote by $q$ the functor $(-) \ldertimes Q$ from $D(\B)$ to $D(A)$.  
The morphisms $R \rightarrow RSL$ and 
$L \rightarrow RSL$ in \eqref{eqn-rsl-3x3-diagram} 
induce the natural transformations 
$r \xrightarrow{r(\text{unit})} rsl$ and 
$l \xrightarrow{\text{unit}} rsl$. Hence 
$R \xrightarrow{\alpha} FL[1]$ induces
the natural transformation \eqref{eqn-r-to-fl(1)}.
The functorial exact triangle 
$r \xrightarrow{\eqref{eqn-r-to-fl(1)}} fl[1] \rightarrow q[1]$ 
induced by the bottom row of \eqref{eqn-rsl-3x3-diagram} 
implies that $r \xrightarrow{\eqref{eqn-r-to-fl(1)}} fl[1]$ is
an isomorphism if and only if $q$ 
is the zero functor. Similarly, the exact triangle 
$R \oplus L \xrightarrow{\id\!\otimes\!\action +
\action\!\otimes\!\id } RSL \rightarrow Q[1]$
implies that
$r \oplus l \xrightarrow{r(\text{unit}) \oplus \text{unit}} rsl$
is an isomorphism if and only if $q[1]$ is the zero functor.

Clearly $Q \simeq 0$ implies that $q$ is the zero functor. 
On the other hand, if $q$ is the zero functor then it sends
all representable $\B$-modules to $0 \in D(\A)$. Thus
$\leftidx{_b}{Q}$ is an acyclic $\A$-module for all 
$b \in \B$, and hence $Q$ is acyclic $\BbimA$ bimodule. 
We conclude that $q$ is the zero functor if and only if
$Q \simeq 0$ in $D(\BbimA)$. 

Finally, $Q \simeq 0$ is equivalent to $\alpha$ (resp. $\alpha'$)
being an isomorphism by 
exactness of the bottom row (resp. right column)
of the diagram \eqref{eqn-rsl-3x3-diagram}. 
\end{proof}

Now define
\begin{align*}
Q' & \overset{\text{def}}{\;=\;}  
\left\{
\MhdB \otimes_{\A} M \otimes_{\B} \MhdA
\xrightarrow{- \trace \otimes \id - \id \otimes \trace} 
\underset{\degzero}{\MhdA \oplus \MhdB}
\right\} \simeq \\
& \simeq \;
\cone \left(LSR \xrightarrow{L\trace \oplus \trace} L \oplus R\right)[1]
\quad \text{ in } D(\BbimA).
\end{align*}

Then, in a similar way, the twisted $2$-cube  
\begin{align}
\label{eqn-M^B-M-M^A-diagram-small}
\vcenter{
\xymatrix{
\MhdB \otimes_{\A} M \otimes_{\B} \MhdA
\ar[r]^<<<<<{\trace \otimes \id} 
\ar[d]_{\id \otimes \trace}
&
\MhdA
\ar[d]
\\
\MhdB
\ar[r]
&
0 
} 
}
\end{align}
produces the following $3 \times 3$-diagram in $D(\BbimA)$ whose
rows and columns are exact triangles:
 \begin{align}
\label{eqn-lsr-3x3-diagram}
\vcenter{
\xymatrix{
Q'[-1]
\ar[r] 
\ar[d]
&
F'R[-1]
\ar[r]^<<<<<{\beta'}
\ar[d]
&
L
\ar@{=}[d]
\\
LT[-1]
\ar[r]
\ar[d]_{\beta}
&
LSR
\ar[r]^{L(\trace)}
\ar[d]^{(\trace)R}
&
L
\ar[d]
\\
R
\ar@{=}[r]
&
R
\ar[r]
&
0.
} 
}
\end{align}

Arguing as in the proof of Lemma
\ref{lemma-condition-cotwist-identifies-adjoints-is-equiv} 
we obtain:
\begin{lemma}
\label{lemma-condition-twist-identifies-adjoints-is-equiv} 
The following are equivalent:
\begin{itemize}
\item $lt[-1] \xrightarrow{\eqref{eqn-lt(-1)-to-r}} r$ is an isomorphism
(the condition \eqref{item-main-theorem-the-twist-identifies-adjoints} of 
Theorem \ref{theorem-main-theorem}). 
\item $lsr \xrightarrow{l(\text{counit})\oplus\text{counit}} l \oplus
r$ is an isomorphism. 
\item $Q' \simeq 0$ in $D(\BbimA)$.
\item $\beta$ is an isomorphism in $D(\BbimA)$.
\item $\beta'$ is an isomorphism in $D(\BbimA)$.
\end{itemize}
\end{lemma}

Consider now the twisted $2$-cube over $\pretriag(\BmodB)$: 
\begin{align}
\label{eqn-SL-SRT'-diagram-small}
\vcenter{
\xymatrix{
\underset{\degone}{\MhdA \otimes_{\A} M}
\ar[rr]^{\id} 
\ar[d]_{1\to1\colon \id\otimes\action\otimes\id}
\ar@{-->}[drr]^{\quad 1\to1\colon - \id \otimes \chi_\B}
& &
\underset{\degone}{\MhdA \otimes_{\A} M}
\ar[d]^{1\to1\colon \id}
\\
\left(\underset{\degzero}{\MhdB\otimes_{\A} M} \xrightarrow{-\action\otimes\id}%
\MhdA\otimes_{\A} M\otimes_{\B} \MhdB\otimes_{\A} M\right)
\ar[rr]^<<<<<<<<<{0\to0\colon \trace}_<<<<<<<<<{{1\to1\colon - \id\otimes\trace}}
& &
\left(\underset{\degzero}{\barB} \xrightarrow{\action} \MhdA\otimes_{\A} M\right)
} 
}
\end{align}
Apriori the total complex 
of a face of a twisted cube over $\pretriag(\BmodB)$
is an object of $\pretriag \pretriag(\BmodB)$.
However, there is a canonical equivalence 
which sends a twisted complex of twisted complexes:
$$ \pretriag \pretriag(\BmodB) \xrightarrow{\sim} \pretriag(\BmodB),$$
cf. \cite[\S2]{BondalKapranov-EnhancedTriangulatedCategories}. 
We implicitly use this equivalence wherever possible.

The Cube Completion Lemma constructs from
the $2$-cube \eqref{eqn-SL-SRT'-diagram-small}
a $3\times 3$ commutative diagram in $D(\BbimB)$ whose 
rows and columns are exact. We now compute this diagram. 

The left column of \eqref{eqn-SL-SRT'-diagram-small} is
the image under $(-) \otimes_\A M [-1]$ of the first 
map in the right column of \eqref{eqn-M^A-M-M^B-diagram-big}. 
It descends to the morphism 
$SL[-1] \xrightarrow{S\alpha'} SRT'$ in \eqref{eqn-rsl-3x3-diagram}
in $D(\BbimB)$ and its convolution is isomorphic to $SQ$.  

The diagonal bimodule $\barB$ is homotopy equivalent 
to the total complex of the right column of 
\eqref{eqn-SL-SRT'-diagram-small}:
\begin{align}
\label{eqn-barB-homotopy-equivalent-to-the-right-column}
\xymatrix{
\left(\underset{\degzero}{\barB}\right)
\ar[rr]^<<<<<<<<<<{0\to0\colon \id \oplus -\action}
&&
\left(\underset{\degzero}{\barB\oplus \MhdA\otimes_{\A} M}
\ar[r]^<<<<<{\action \oplus \id}\right.
&
\left.  \underset{~}{\MhdA \otimes_{\A} M}\right). 
}
\end{align}

The total complex of the top row of \eqref{eqn-SL-SRT'-diagram-small}
is the null-homotopic twisted complex 
$\left( \underset{\degzero}{\MhdA \otimes_\A M} 
\xrightarrow{\id} \MhdA \otimes_\A M\right)$, while the total complex of the 
bottom row is the twisted complex
\begin{align}
\label{eqn-t't-via-twisted-complexes}
\MhdB \otimes_\A M \xrightarrow{\trace\oplus (\action\otimes\id)} 
\underset{\degzero}{\barB \oplus \bigl(\MhdA \otimes_\A M \otimes_\B \MhdB \otimes_\A M\bigr)}
\xrightarrow{\action\oplus(- \id\otimes\trace)}
\MhdA \otimes_\A M 
\end{align}
which we've shown in Prps. 
\ref{prps-unit-and-counit-of-t't-via-twisted-complexes} to convolve to $TT'$. 

By the Cube Lemma, the total complex of the whole $2$-cube equals 
the total complex of the $1$-cube constructed from its rows. 
It is then clear that the total complex of \eqref{eqn-SL-SRT'-diagram-small}
is homotopy equivalent to \eqref{eqn-t't-via-twisted-complexes}:
\begin{tiny}
\begin{align}
\label{eqn-SL-SRT'-diagram-part-isomorphic-to-unit-TT'}
\vcenter{
\xymatrix{
\left(\!\!\underset{\degzero}{\MhdA\!\otimes_{\A}\!M}\!\xrightarrow{- \id}\!\MhdA\!\otimes_{\A}\!M\!\!\right)
\ar[rr]^<<<<<<<<<<{0\to0: 0\oplus\id\!\otimes\!\action\!\otimes\!\id}_<<<<<<<<<<{0\!\to\!1:
- \id\!\otimes\!\chi_\B, 1\!\to\!1: \id}
\ar@{-->}[drr]_>>>>>>>>>>>>>>>>>>>>>>>>>>>>>{1\to1: \id \otimes
\chi_\B\quad}^{\quad\quad 1\to0: 0 \oplus - \id\!\otimes\!\action\!\otimes\!\id}
&&
\left(\!\!\MhdB\!\otimes_\A\!M\!\!\xrightarrow{\trace\oplus
(\action\!\otimes\!\id)}\!\!\underset{\degzero}{\barB\!\oplus\!\bigl(\!\MhdA\!\otimes_\A\!M\!\otimes_\B\!\MhdB\!\otimes_\A\!M\!\bigr)}\!\!\xrightarrow{\action\oplus(- \id\!\otimes\!\trace)}\!\!\MhdA\!\otimes_\A\!M\!\!\right)
\ar@<-1ex>[d] \ar[d] \ar[d] \ar@<1ex>[d]^{\id}
\\
&&
\left(\!\!\MhdB\!\otimes_\A\!M\!\!\xrightarrow{\trace\oplus(\action\!\otimes\!\id)}\!\!
\underset{\degzero}{\barB\!\oplus\!\bigl(\!\MhdA\!\otimes_\A\!M\!\otimes_\B\!\MhdB\!\otimes_\A\!M\!\bigr)}
\!\!\xrightarrow{\action\oplus(- \id\!\otimes\!\trace)}\!\!\MhdA\!\otimes_\A\!M\!\!\right)\!\!.\!\!
}
}
\!\!
\end{align}
\end{tiny}

Consider the map which the Cube Lemma
constructs from the total complex of the right column 
of \eqref{eqn-SL-SRT'-diagram-small} to the total complex 
of the whole 2-cube. It composes with the homotopy equivalences 
\eqref{eqn-barB-homotopy-equivalent-to-the-right-column}
and \eqref{eqn-SL-SRT'-diagram-part-isomorphic-to-unit-TT'}
to give the map \eqref{eqn-adjunction-unit-t't}. 
The latter was proven 
in Prps. \ref{prps-unit-and-counit-of-t't-via-twisted-complexes} 
to be isomorphic in $D(\BbimB)$ to $\B \xrightarrow{\action} TT'$. 

Putting together all of the above, 
we see that the diagram constructed by the Cube Completion Lemmma
from \eqref{eqn-SL-SRT'-diagram-small} is isomorphic in $D(\BbimB)$ to:
\begin{align}
\label{eqn-id-t't-3x3-diagram}
\vcenter{
\xymatrix{
SL[-1]
\ar@{=}[r]
\ar[d]_{S(\alpha')}
&
SL[-1]
\ar[r]
\ar[d]
&
0
\ar[d]
\\
SRT'
\ar[r]
\ar[d]
&
T'
\ar[r]
\ar[d]
&
TT'
\ar@{=}[d]
\\
SQ
\ar[r]
&
\B
\ar[r]^{\action}
&
TT'.
} 
}
\end{align}

Similarly, the following twisted $2$-cube over $\pretriag(\BmodB)$
\begin{align}
\label{eqn-SLT-SR-diagram-small}
\vcenter{
\xymatrix@C=3cm{
({\MhdB\otimes_{\A} M} \xrightarrow{\trace} \underset{\degzero}{\barB})
\ar[r]^<<<<<<<<<<<<<<<{-1\to-1\colon - \id \otimes \action}_<<<<<<<<<<<<<<<{0\to0\colon \action}
\ar[d]_{-1\to-1\colon \id}
\ar@{-->}[dr]^{\quad\quad -1\to-1\colon - \id \otimes \chi_\A}
&
{({\MhdB\otimes_{\A} M\otimes_{\B} \MhdA\otimes_{\A} M}
\xrightarrow{- \trace\otimes\id}
\underset{\degzero}{\MhdA\otimes_\A M})}
\ar[d]^{-1\to-1\colon \id \otimes \trace \otimes \id}
\\
\underset{\degminusone}{\MhdB \otimes_{\A} M}
\ar[r]^{\id} 
&
\underset{\degminusone}{\MhdB \otimes_{\A} M}
} 
}
\end{align}
produces the following diagram in $D(\BbimB)$ with exact rows and columns:
\begin{align}
\label{eqn-t't-id-3x3-diagram}
\vcenter{
\xymatrix{
T'T
\ar@{=}[d]
\ar[r]^{\trace}
&
\B
\ar[d]
\ar[r]
&
SQ'
\ar[d]
\\
T'T
\ar[d]
\ar[r]
&
T
\ar[d]
\ar[r]
&
SLT
\ar[d]^{S(\beta)}
\\
0
\ar[r]
&
SR[1]
\ar@{=}[r]
&
SR[1].
} 
}
\end{align}

Similarly, we incorporate the maps 
$\A \xrightarrow{\action} FF'$ and $F'F \xrightarrow{\trace} \A$
into the following two $3\times 3$ diagrams in $D(\AbimA)$ 
with exact rows and columns:
\begin{align}
\label{eqn-id-f'f-3x3-diagram}
\vcenter{
\xymatrix{
RS[-1]
\ar@{=}[r]
\ar[d]_{(\alpha)S}
&
RS[-1]
\ar[d]
\ar[r]
&
0
\ar[d]
\\
FLS
\ar[r]
\ar[d]
&
F
\ar[d]
\ar[r]
&
FF'
\ar@{=}[d]
\\
QS
\ar[r]
&
\A
\ar[r]^{\action}
&
FF'
} 
}
\end{align}
\begin{align}
\label{eqn-f'f-id-3x3-diagram}
\vcenter{
\xymatrix{
F'F
\ar[r]^{\trace}
\ar@{=}[d]
&
\A
\ar[d]
\ar[r]
&
Q'S
\ar[d]
\\
F'F
\ar[r]
\ar[d]
&
F'
\ar[r]
\ar[d]
&
F'RS
\ar[d]^{(\beta')S}
\\
0
\ar[r]
&
LS[1]
\ar@{=}[r]
&
LS[1].
} 
}
\end{align}

 We obtain immediately:
\begin{prps}
\label{prop-cond-3-4-isom}
~

\begin{enumerate}
\item
\label{item-cond-3-isom}
If the natural transformation $lt[-1] \xrightarrow{\eqref{eqn-lt(-1)-to-r}}
r$ is an isomorphism (the condition 
\eqref{item-main-theorem-the-twist-identifies-adjoints} of 
Theorem \ref{theorem-main-theorem}) 
then the adjunction counits $t't \rightarrow  \id$ and $f'f \rightarrow \id$ 
are isomorphisms.
\item
\label{item-cond-4-isom}
If the natural transformation $r \xrightarrow{\eqref{eqn-r-to-fl(1)}}
fl[1]$ is an isomorphism
(the condition 
\eqref{item-main-theorem-the-cotwist-identifies-adjoints} of Theorem
\ref{theorem-main-theorem}) 
then the adjunction units $\id\to tt'$ and $\id\to ff'$ are isomorphisms.
\end{enumerate}
\end{prps}
\begin{proof}
We only prove the first claim. By Lemma \ref{lemma-condition-twist-identifies-adjoints-is-equiv} 
the condition \eqref{item-main-theorem-the-twist-identifies-adjoints} 
of Theorem \ref{theorem-main-theorem} is equivalent to $Q \simeq 0$ in 
$D(\BbimA)$. Therefore $SQ \simeq 0$ in $D(\BbimB)$ and since  
the bottom row of \eqref{eqn-id-t't-3x3-diagram} is exact  
$\B \xrightarrow{\action} TT'$ is an 
isomorphism. Thus $\id \xrightarrow{\text{unit}} tt'$ is 
an isomorphism. Similarly, 
$QS \simeq 0$ in $D(\AbimA)$ and by the exactness
of the bottom row of \eqref{eqn-id-f'f-3x3-diagram} the map 
$\A \xrightarrow{\action} FF'$ is an
isomorphism. Hence $\id \xrightarrow{\text{unit}} ff'$
is also an isomorphism. 
\end{proof}

\begin{lemma}
\label{lemma-alpha-beta-compositions}
Let $\alpha$, $\alpha'$, $\beta$ and $\beta'$ be as in diagrams 
\eqref{eqn-rsl-3x3-diagram} and \eqref{eqn-lsr-3x3-diagram}. Then 
\begin{enumerate}
\item 
The composition  
$R
\xrightarrow{(\action)R}
FF'R
\xrightarrow{F\beta'}
FL[1]
$ is the map $\alpha$. 
\item
The composition  
$LT[-1]
\xrightarrow{\alpha' T}
RT'T
\xrightarrow{R(\trace)}
R
$ is the map $\beta$. 
\item
The composition  
$L
\xrightarrow{L(\action)}
LTT'
\xrightarrow{\beta T'}
RT'[1]
$ is the map $\alpha'$. 
\item
The composition  
$F'R[-1]
\xrightarrow{F'\alpha}
F'FL
\xrightarrow{(\trace)L}
L
$ is the map $\beta'$. 
\end{enumerate}
\end{lemma}
\begin{proof}
We only prove the first claim, the other three are
proved analogously. Note also, that throughout the proof 
we omit labelling the internal twisted maps inside twisted
complexes, since they are not relevant to our argument. The results
we quote before stating each twisted complex identify these maps 
explicitly.  

By construction of \eqref{eqn-rsl-3x3-diagram} the 
map $R \xrightarrow{\alpha} FL[1]$ in $D(\BbimA)$ 
descends from the map of twisted complexes
\begin{tiny}
\begin{align}
\label{eqn-R->-FL1-for-twisted-complexes}
\vcenter{
\xymatrix@C=2cm{
& 
\left( \underset{\degzero}{\MhdB} \right)
\ar[d]^{\action \otimes \id}
\\
\Bigl( \MhdA 
\ar[r]
&
\underset{\degzero}{\MhdA \otimes_{\A} M \otimes_{\B} \MhdB} \Bigr)
}
}.
\end{align}
\end{tiny}

By Prps. \ref{prps-unit-and-counit-of-f'f-via-twisted-complexes} 
the map $ R \xrightarrow{(\action)R} FF'R$
descends from the map of twisted complexes 
\begin{tiny}
\begin{align}
\label{eqn-R-->--FF'R-for-twisted-complexes}
\vcenter{
\xymatrix{
& & 
\left( \underset{\degzero}{\MhdB} \right)
\ar[d]_{\id \oplus (\id \otimes \action)}
\ar@{-->}[ddrr]^{\quad \quad \id \otimes (-(\chi_\A \otimes
\id)\circ\action)}
& & 
\\
& &
\underset{\degzero}{\MhdB \oplus \bigl(\MhdB \otimes_{\A} M
\otimes_{\B}  \MhdB\bigr)}
\ar[d]_{\id \oplus (- \id \otimes \action \otimes \id)}
& &
\\
\Bigl( \MhdB \otimes_{\A} M \otimes_{\B} \MhdA 
\ar[rr]
\ar@/^/@{-->}[rrrr]
& &
\underset{\degzero}{\MhdB \oplus \bigl(\MhdB \otimes_{\A} M
\otimes_{\B} \MhdA \otimes_\A M \otimes_\B \MhdB \bigr)}
\ar[rr]
& &
\MhdB \otimes_{\A} M\otimes_{\B} \MhdB \Bigr)
} 
}.
\end{align}
\end{tiny}

Finally, $FF'R \xrightarrow{F\beta'} FL[1]$ descends  
from the map of twisted complexes which is computed as follows. 

By construction of the diagram \eqref{eqn-lsr-3x3-diagram} 
the map $F'R \xrightarrow{\beta'} L[1]$ descends from
\begin{tiny}
\begin{align}
\label{eqn-map-defining-F'R--beta'--L1}
\vcenter{
\xymatrix@C=2cm{
\Bigl(\MhdB \otimes_{\A} M \otimes_{\B} \MhdA
\ar[r]
\ar[d]^{-\trace \otimes \id}
& 
\underset{\degzero}{\MhdA} \Bigr)
\\ 
\left( \underset{\degminusone}{\MhdA} \right)
&
}
}.
\end{align}
\end{tiny}
On the other hand, $F$ is the convolution of 
$\left( 
\underset{\degzero}{\barA} \xrightarrow{\action} M \otimes_{\B} \MhdB 
\right)$. Thus the map $FF'R \xrightarrow{F\beta'} FL[1]$ is 
\begin{align*}
\left\{
{\MhdB \otimes_{\A} M \otimes_{\B} \MhdA}
\xrightarrow{\id \otimes \trace}
\underset{\degzero}{\MhdB} \right\} \otimes 
\left\{ 
\underset{\degzero}{\barA} \xrightarrow{\action} M \otimes_{\B} \MhdB
\right\}
\xrightarrow{\eqref{eqn-map-defining-F'R--beta'--L1} \otimes \id}
\underset{\degminusone}{\MhdA} \otimes 
\left\{ 
\underset{\degzero}{\barA} \xrightarrow{\action} M \otimes_{\B} \MhdB
\right\}.
\end{align*}
Lemma \ref{lemma-tensor-and-hom-of-twisted-complexes} tells us
how to take tensor product of twisted complexes in a way 
compatible with convolutions. It follows from it that 
$FF'R \xrightarrow{F\beta'} FL[1]$ descends from the map 
\begin{tiny}
\begin{align}
\label{eqn-FF'R->-FL1-for-twisted-complexes}
\vcenter{
\xymatrix{
\Bigr(\MhdB \otimes_{\A} M \otimes_{\B} \MhdA 
\ar[d]^{- \trace \otimes \id}
\ar[rr]
\ar@/^/@{-->}[rrrr]
& &
\underset{\degzero}{\MhdB \oplus \bigl(\MhdB \otimes_{\A} M
\otimes_{\B} \MhdA \otimes_\A M \otimes_\B \MhdB \bigr)}
\ar[rr]
\ar[d]^{0 \oplus ( - \trace \otimes \id)}
& &
\MhdB \otimes_{\A} M\otimes_{\B} \MhdB \Bigl)
\\
\Bigl( \MhdA 
\ar[rr]
& & 
\underset{\degzero}{\MhdA \otimes_{\A} M \otimes_{\B} \MhdB} \Bigr)
& & 
} 
}.
\end{align}
\end{tiny}

It remains to prove that 
the composition of
\eqref{eqn-R-->--FF'R-for-twisted-complexes} and 
\eqref{eqn-FF'R->-FL1-for-twisted-complexes} is homotopic 
to \eqref{eqn-R->-FL1-for-twisted-complexes}. This is equivalent 
to the following diagram commuting up to homotopy:
\begin{tiny}
\begin{align}
\label{eqn-triangle-square-commutative-diagram}
\vcenter{
\xymatrix{
\MhdB
\ar[rr]^<<<<<<<<<<{\id\otimes\action}
\ar[rrd]_{\id}
&
& 
\MhdB\otimes_{\A}M\otimes_{\B}\MhdB
\ar[rrr]^<<<<<<<<<<<<<<{\id\otimes\action\otimes\id}
\ar[d]^{\trace\otimes\id}
&
&
&
(\MhdB\otimes_{\A}M)\otimes_{\B}\MhdA\otimes_{\A}M\otimes_{\B}\MhdB
\ar[d]^{\trace\otimes\id}
\\
&
&
\MhdB
\ar[rrr]^{\action\otimes\id}
&
&
&
\MhdA\otimes_{\A}M\otimes_{\B}\MhdB.
}
}
\end{align}
\end{tiny}
This is clear: the square in
\eqref{eqn-triangle-square-commutative-diagram} 
commutes up to homotopy by the functoriality of the tensor product,
while 
the triangle commutes up to homotopy by
Prop.
\ref{prps-M^B-and-M^A-are-homotopy-adjoints-of-M-via-S-SRS-S-maps}.

\end{proof}

Let $\gamma: F'L[-1]\to LT'[1]$ be the map induced by 
the following morphism of twisted complexes
\begin{align}
\label{eqn-gamma-on-the-level-of-twisted complexes}
\vcenter{
\xymatrix{
 & 
\bigl( \underset{\degzero}{\MhdA\otimes_{\A} M \otimes_{\B} \MhdA}
\ar[r]^<<<<<{\id \otimes \trace}
\ar[d]^{\id}
&
\MhdA\bigr)
\\
\bigl( \MhdA
\ar[r]^>>>>>{\action \otimes \id}
&
\underset{\degzero}{\MhdA\otimes_{\A} M \otimes_{\B} \MhdA} \bigr).
}
}
\end{align}

\begin{lemma}
\label{lemma-gamma-is-an-isomorphism} 
The morphism \eqref{eqn-gamma-on-the-level-of-twisted complexes}
is a homotopy equivalence. Consequently, the map $\gamma$ is an
isomorphism. 
\end{lemma}
\begin{proof}
Consider the composition 
\begin{align}
\label{eqn-RSR--R--RSR-composition}
\MhdA\otimes_{\A} \left(M \otimes_{\B} \MhdA\right)
\xrightarrow{\id \otimes \trace}
\MhdA
\xrightarrow{\action \otimes \id}
\left(\MhdA\otimes_{\A} M \right) \otimes_{\B} \MhdA.
\end{align}
We claim that the homotopy inverse of 
\eqref{eqn-gamma-on-the-level-of-twisted complexes}
is the morphism 
\begin{align}
\label{eqn-homotopy-inverse-of-gamma-on-the-level-of-twisted complexes}
\xymatrix@C=2.5cm{
\Bigl( \MhdA
\ar@{-->}[dr]_{\quad - (\action \otimes \id) \circ \xi_\A \quad}
\ar[r]^<<<<<<<<<<<{\action \otimes \id}
\ar@{.>}[drr]|<<<<<<<<<<<<<<<<<<<<<<<<<<<<<<{\underset{~}{\xi_\A^2}}
&
\underset{\degzero}{\MhdA\otimes_{\A} M \otimes_{\B} \MhdA} \Bigr)  
\ar[d]^<<{\id - \eqref{eqn-RSR--R--RSR-composition}}
\ar@{-->}[dr]^{\quad - \xi_\A \circ (\id \otimes \trace)}
&
\\
&
\Bigl(\underset{\degzero}{\MhdA\otimes_{\A} M \otimes_{\B} \MhdA}
\ar[r]_<<<<<<<<<<<{\id \otimes \trace}
&
\MhdA\Bigr).
}
\end{align}

Indeed, the composition of 
\eqref{eqn-gamma-on-the-level-of-twisted complexes}
with 
\eqref{eqn-homotopy-inverse-of-gamma-on-the-level-of-twisted complexes}
is the morphism of twisted complexes 
\begin{align*}
\xymatrix@C=2cm{
\Bigl( \underset{\degzero}{\MhdA\otimes_{\A} M \otimes_{\B} \MhdA}
\ar[r]^<<<<<<<<<<<{\id \otimes \trace}
\ar[d]^{\id - \eqref{eqn-RSR--R--RSR-composition}}
\ar@{-->}[dr]^{\quad - \xi_\A \circ (\id \otimes \trace)}
&
\MhdA\Bigr)
\\
\Bigl( \underset{\degzero}{\MhdA\otimes_{\A} M \otimes_{\B} \MhdA}
\ar[r]^<<<<<<<<<<<{\id \otimes \trace}
&
\MhdA\Bigr)
}
\end{align*}
which differs from the identity morphism by
\begin{align*}
\xymatrix@C=2cm{
\Bigl( \underset{\degzero}{\MhdA\otimes_{\A} M \otimes_{\B} \MhdA}
\ar[r]^<<<<<<<<<<<{\id \otimes \trace}
\ar[d]^{\eqref{eqn-RSR--R--RSR-composition}}
\ar@{-->}[dr]^{\quad \xi_\A \circ (\id \otimes \trace)}
&
\MhdA\Bigr)
\ar[d]^{\id}
\\
\Bigl( \underset{\degzero}{\MhdA\otimes_{\A} M \otimes_{\B} \MhdA}
\ar[r]^<<<<<<<<<<<{\id \otimes \trace}
&
\MhdA\Bigr).
}
\end{align*}
This is null-homotopic because it is the differential of 
the following degree $-1$ morphism of twisted complexes:
\begin{align*}
\xymatrix@C=2cm{
\Bigl( \underset{\degzero}{\MhdA\otimes_{\A} M \otimes_{\B} \MhdA}
\ar[r]^<<<<<<<<<<<{\id \otimes \trace}
&
\MhdA\Bigr)
\ar[dl]^>>>>>>>>>>>>>>>>>>>>>{\quad \action \otimes \id}
\ar@{-->}[d]^{\xi_\A}
\\
\Bigl( \underset{\degzero}{\MhdA\otimes_{\A} M \otimes_{\B} \MhdA}
\ar[r]^<<<<<<<<<<<{\id \otimes \trace}
&
\MhdA\Bigr).
}
\end{align*}
Thus the composition of \eqref{eqn-gamma-on-the-level-of-twisted complexes}
with 
\eqref{eqn-homotopy-inverse-of-gamma-on-the-level-of-twisted complexes}
is homotopic to $\id$. 

The composition of \eqref{eqn-homotopy-inverse-of-gamma-on-the-level-of-twisted complexes}
and
\eqref{eqn-gamma-on-the-level-of-twisted complexes}
being homotopic to $\id$ is proved similarly.
\end{proof}

\begin{lemma}
\label{lemma-l-commutes-with-dual-twist-and-cotwist}
The composition $F'L[-1]\xrightarrow{F'\alpha'}
F'RT'\xrightarrow{\beta' T'} LT'[1]$ equals the map $\gamma$.
\end{lemma}
\begin{proof}
Arguing as in Lemma \ref{lemma-alpha-beta-compositions}
we see that $F'L[-1]\xrightarrow{F'\alpha'} F'RT'$
descends from the twisted complex map 
\begin{tiny}
\begin{align}
\label{eqn-F'L(-1)-->--F'RT'-for-twisted-complexes}
\vcenter{
\xymatrix{
&  
\Bigl( \underset{\degzero}{\MhdA \otimes_{\A} M \otimes_{\B} \MhdA} 
\ar[r]
\ar[d]^{0 \oplus (\id \otimes \action \otimes \id \otimes \id)}
&  
\MhdA \Bigr)
\ar[d]^{\id \otimes \action} 
\\
\Bigl(
\MhdB \otimes_{\A} M \otimes_{\B} \MhdA
\ar[r]
\ar@{-->}@/^/[rr]
& 
\underset{\degzero}{\MhdB \oplus \bigl(\MhdA \otimes_{\A} M
\otimes_{\B}  \MhdB \otimes_{\A} M \otimes_{\B} \MhdA \bigr)}
\ar[r]
& 
\MhdA \otimes_{\A} M \otimes_{\B} \MhdB \Bigr).
}
}
\end{align}
\end{tiny}
Once again we omit labeling the internal twisted maps inside twisted
complexes since they are not relevant to our argument. Similarly, 
 $F'RT'\xrightarrow{\beta' T'} LT'[1]$ descends  
from the twisted complex map 
\begin{tiny}
\begin{align}
\label{eqn-F'RT'-->--LT'(1)-for-twisted-complexes}
\vcenter{
\xymatrix{
\Bigl( \MhdB \otimes_{\A} M \otimes_{\B} \MhdA
\ar[r]
\ar[d]^{\trace \otimes \id}
\ar@{-->}@/^/[rr]
& 
\underset{\degzero}{\MhdB \oplus \bigl(\MhdA \otimes_{\A} M
\otimes_{\B}  \MhdB \otimes_{\A} M \otimes_{\B} \MhdA \bigr)}
\ar[r]
\ar[d]^{0 \oplus (\id \otimes \id \otimes \trace \otimes \id)}
& 
\MhdA \otimes_{\A} M \otimes_{\B} \MhdB \Bigr)
\\
\Bigl( \MhdA
\ar[r]
&  
\underset{\degzero}{\MhdA \otimes_{\A} M \otimes_{\B} \MhdA}\Bigr). 
&  
}
}
\end{align}
\end{tiny}

Hence the composition 
$F'L[-1]\xrightarrow{F'\alpha'} F'RT'\xrightarrow{\beta' T'} LT'[1]$
descends from 
\begin{tiny}
\begin{align}
\label{eqn-F'L(-1)-->--LT'(1)-for-twisted-complexes}
\vcenter{
\xymatrix{
&  
\Bigl( \underset{\degzero}{\MhdA \otimes_{\A} M \otimes_{\B} \MhdA} 
\ar[r]
\ar[d]^{\id \otimes \bigl((\id \otimes \trace) \circ (\action \otimes
\id) \bigr) \otimes \id}
&  
\MhdA \Bigr)
\\
\Bigl( \MhdA
\ar[r]
&  
\underset{\degzero}{\MhdA \otimes_{\A} M \otimes_{\B} \MhdA}\Bigr). 
&  
}
}
\end{align}
\end{tiny}
By Prop. \ref{prps-M^B-and-M^A-are-homotopy-adjoints-of-M-via-S-SRS-S-maps}
the composition
$$ M \xrightarrow{\action \otimes \id } M \otimes_{\B} \MhdB \otimes_{\A} M
\xrightarrow{\id \otimes \trace} M$$
is homotopic to $\id$, and thus
\eqref{eqn-F'L(-1)-->--LT'(1)-for-twisted-complexes} is homotopic to
the map $\gamma$. 
\end{proof}

\begin{lemma}
\label{lemma-lt-fl-2maps}
The following maps are equal:
\begin{enumerate}
\item 
\label{eqn-lemma-lt-fl-through-rsl}
$ LT[-1] \xrightarrow{\alpha\circ\beta} FL[1]$
\item 
\label{eqn-lemma-lt-fl-long-way}
$LT[-1] 
\xrightarrow{\eqref{eqn-adjunction-unit-f'f}LT}
FF'LT[-1]
\xrightarrow{F\gamma T}
FLT'T[1]
\xrightarrow{FL\eqref{eqn-adjunction-counit-t't}}
FL[1]$
\end{enumerate}
\end{lemma}
\begin{proof}
By Lemma \ref{lemma-alpha-beta-compositions} the composition 
$LT[-1] \xrightarrow{\alpha\circ\beta} FL[1]$ equals 
the composition 
\begin{align}
\label{eqn-alpha-beta-expanded}
LT[-1] \xrightarrow{\alpha' T} RT'T
\xrightarrow{R\eqref{eqn-adjunction-counit-t't}} R
\xrightarrow{\eqref{eqn-adjunction-unit-f'f}R} FF'R
\xrightarrow{F\beta'} FL[1]. 
\end{align}
By functoriality of tensor product the composition 
$\eqref{eqn-alpha-beta-expanded}$ equals the composition
\begin{align}
\label{eqn-alpha-beta-expanded-first-changeover}
LT[-1] \xrightarrow{\alpha' T} RT'T
\xrightarrow{\eqref{eqn-adjunction-unit-f'f}RT'T} FF'RT'T
\xrightarrow{FF'R\eqref{eqn-adjunction-counit-t't}} FF'R
\xrightarrow{F\beta'} FL[1],
\end{align}
which by functoriality of tensor product again  
equals the composition
\begin{align}
\label{eqn-alpha-beta-expanded-second-changeover}
LT[-1] \xrightarrow{\eqref{eqn-adjunction-unit-f'f}LT} FF'LT[-1]
\xrightarrow{FF'\alpha'T} FF'RT'T
\xrightarrow{F\beta'T'T} FLT'T[1]
\xrightarrow{FL\eqref{eqn-adjunction-counit-t't}} FL[1].
\end{align}
The claim now follows by applying Lemma \ref{lemma-l-commutes-with-dual-twist-and-cotwist}
to the two maps in the middle of
\eqref{eqn-alpha-beta-expanded-second-changeover}.
\end{proof}

Similarly, let $\gamma': RT[-1]\to FR[1]$ be the map induced by 
the following morphism of twisted complexes
\begin{align}
\label{eqn-gamma'-on-the-level-of-twisted complexes}
\xymatrix{
 & 
\Bigl(\underset{\degzero}{\MhdB\otimes_{\A} M \otimes_{\B} \MhdB}
\ar[r]^<<<<<{\trace \otimes \id}
\ar[d]^{\id}
&
\MhdB\Bigr)
\\
\Bigl( \MhdB
\ar[r]^>>>>>{\id \otimes \action}
&
\underset{\degzero}{\MhdB\otimes_{\A} M \otimes_{\B} \MhdB} \Bigr).
}
\end{align}
The following two results are proved identically to Lemmas
\ref{lemma-gamma-is-an-isomorphism} and \ref{lemma-lt-fl-2maps}:
\begin{lemma}
\label{lemma-gamma'-is-an-isomorphism} 
The morphism \eqref{eqn-gamma'-on-the-level-of-twisted complexes}
is a homotopy equivalence. Consequently, the map $\gamma'$ is an
isomorphism. 
\end{lemma}
\begin{lemma}
\label{lemma-f'r-rt'-2maps}
The following maps are equal: 
\begin{enumerate}
\item 
\label{eqn-lemma-f'r-rt'-through-rsl}
$ F'R[-1] \xrightarrow{\alpha'\circ\beta'} RT'[1]$
\item 
\label{eqn-lemma-f'r-rt'-long-way}
$F'R[-1] 
\xrightarrow{F'R\eqref{eqn-adjunction-unit-t't}}
F'RTT'[-1]
\xrightarrow{F'\gamma'T'}
F'FRT'[1]
\xrightarrow{\eqref{eqn-adjunction-counit-f'f}RT'}
RT'[1]$
\end{enumerate}
\end{lemma}

Thus, if the adjunction maps \eqref{eqn-adjunction-unit-f'f} and 
\eqref{eqn-adjunction-counit-t't} are isomorphisms,
the composition $\alpha\circ\beta$ is an isomorphism, and it filters
though the canonical map $RSL \xrightarrow{\eta} FL[1]$.

We are now in a position to prove the main theorem. Before we begin
the argument, recall that in a triangulated category all retracts
are split. More precisely, let $Z \xrightarrow{e} Y$ be a retract 
in a triangulated category, that is --- there exists 
$Y \xrightarrow{g} Z$ with $Z \xrightarrow{e} Y  \xrightarrow{g} Z$ 
being the identity. Then for any completion of $g$ to an exact
triangle $X \xrightarrow{f} Y  \xrightarrow{g} Z$, 
$X \oplus Z \xrightarrow{f\oplus e} Y$ is 
an isomorphism. Moreover, its inverse is of form $Y
\xrightarrow{h\oplus g} X \oplus Z$ for some morphism 
$Y \xrightarrow{h} X$. This can be established using only the axioms
of triangulated categories, though for enhanced triangulated categories 
one can see it very explicitly on the level of twisted complexes. 
\\

\begin{proof}[Proof of Theorem \ref{theorem-main-theorem}]
~
\vskip 0.25cm

\em $\eqref{item-main-theorem-the-twist-identifies-adjoints}
\text{ and }\eqref{item-main-theorem-the-cotwist-identifies-adjoints}
\Rightarrow \eqref{item-main-theorem-the-twist-is-an-equivalence}
\text{ and }
\eqref{item-main-theorem-the-cotwist-is-an-equivalence}$\rm:

Suppose that natural transformations 
$lt[-1] \xrightarrow{\eqref{eqn-lt(-1)-to-r}} r$
and $r \xrightarrow{\eqref{eqn-r-to-fl(1)}} fl[1]$ are functorial isomorphisms. 
In other words, the conditions 
\eqref{item-main-theorem-the-twist-identifies-adjoints} 
and \eqref{item-main-theorem-the-cotwist-identifies-adjoints} hold.
Then by the Proposition \ref{prop-cond-3-4-isom} 
the units and counits of both adjoint pairs $(t',t)$ and $(f',f)$ 
are isomorphisms. Hence $(t',t)$ and $(f',f)$ are pairs of mutually 
inverse equivalences, that is --  the conditions
\eqref{item-main-theorem-the-twist-is-an-equivalence} and 
\eqref{item-main-theorem-the-cotwist-is-an-equivalence} hold. 

\vskip 0.25cm

\em $\eqref{item-main-theorem-the-twist-is-an-equivalence}
\text{ and }\eqref{item-main-theorem-the-twist-identifies-adjoints}
\Rightarrow
\eqref{item-main-theorem-the-cotwist-identifies-adjoints}$

$\eqref{item-main-theorem-the-twist-is-an-equivalence}
\text{ and }\eqref{item-main-theorem-the-cotwist-identifies-adjoints}
\Rightarrow
\eqref{item-main-theorem-the-twist-identifies-adjoints}$ 

$\eqref{item-main-theorem-the-cotwist-is-an-equivalence}
\text{ and }\eqref{item-main-theorem-the-twist-identifies-adjoints}
\Rightarrow
\eqref{item-main-theorem-the-cotwist-identifies-adjoints}$

$\eqref{item-main-theorem-the-cotwist-is-an-equivalence}
\text{ and }\eqref{item-main-theorem-the-cotwist-identifies-adjoints}
\Rightarrow
\eqref{item-main-theorem-the-twist-identifies-adjoints}$\rm:

\vskip 0.25cm

We only prove the assertion 
$\eqref{item-main-theorem-the-twist-is-an-equivalence}
\text{ and }\eqref{item-main-theorem-the-twist-identifies-adjoints}
\Rightarrow
\eqref{item-main-theorem-the-cotwist-identifies-adjoints}$, the other
three are proved similarly.

Assume that the conditions \eqref{item-main-theorem-the-twist-is-an-equivalence}
and \eqref{item-main-theorem-the-twist-identifies-adjoints} hold.
The condition \eqref{item-main-theorem-the-twist-is-an-equivalence}
is $(t',t)$ being mutually inverse equivalences. In particular, 
the adjunction unit $\id \rightarrow tt'$ is an isomorphism. 
Therefore the morphism 
$\B \xrightarrow{\eqref{eqn-adjunction-unit-t't}}T'T$, 
which by Prop.  
\ref{prps-unit-and-counit-of-t't-via-twisted-complexes} 
induces this adjunction unit, is also an isomorphism. 
On the other hand, by 
Lemma \ref{lemma-condition-twist-identifies-adjoints-is-equiv} 
the condition \eqref{item-main-theorem-the-twist-identifies-adjoints}
is equivalent to the map $LT \xrightarrow{\beta} R[1]$
in the diagram $\eqref{eqn-lsr-3x3-diagram}$ being an isomorphism.

By Lemma \ref{lemma-condition-twist-identifies-adjoints-is-equiv} 
the condition 
\eqref{item-main-theorem-the-cotwist-identifies-adjoints}
is equivalent to the map 
$L \xrightarrow{\alpha'} RT'[1]$
in the diagram $\eqref{eqn-rsl-3x3-diagram}$ being an isomorphism.
By Lemma \ref{lemma-alpha-beta-compositions} the map
$L \xrightarrow{\alpha'} RT'[1]$
decomposes as $$L \xrightarrow{L\eqref{eqn-adjunction-unit-t't}} 
LTT' \xrightarrow{\beta T'} RT'[1].$$
By above, both the composants are isomorphisms. 
Hence $L \xrightarrow{\alpha'} RT'[1]$ is also an isomorphism, 
as desired. 

\vskip 0.25cm

$\eqref{item-main-theorem-the-twist-is-an-equivalence}
\text{ and }\eqref{item-main-theorem-the-cotwist-is-an-equivalence}
\Rightarrow
\eqref{item-main-theorem-the-cotwist-identifies-adjoints}$\rm:

\vskip 0.25cm

Assume the conditions \eqref{item-main-theorem-the-twist-is-an-equivalence} and
\eqref{item-main-theorem-the-cotwist-is-an-equivalence} hold. 
Then the maps $\id \xrightarrow{\eqref{eqn-adjunction-unit-f'f}} FF'$
and $T'T \xrightarrow{\eqref{eqn-adjunction-counit-t't}} \id$
are isomorphisms. By Lemma 
\ref{lemma-gamma-is-an-isomorphism} map $F'L[-1] \xrightarrow{\gamma} LT'[1]$
induced by \eqref{eqn-gamma-on-the-level-of-twisted complexes} 
is always an isomorphism. By Lemma \ref{lemma-lt-fl-2maps}
the map $LT[-1] \xrightarrow{\beta} R \xrightarrow{\alpha} FL[-1]$ 
decomposes as
$$ LT[-1] 
\xrightarrow{\eqref{eqn-adjunction-unit-f'f}LT}
FF'LT[-1]
\xrightarrow{F\gamma T}
FLT'T[1]
\xrightarrow{FL\eqref{eqn-adjunction-counit-t't}}
FL[1] $$
and is therefore an isomorphism.

This isomorphism $\alpha \circ \beta$ filters 
through the canonical map $RSL \xrightarrow{\eta} FL[1]$, 
thus $FL[1]$ is a retract of $RSL$. More specifically, 
denote by $\bareta$ the map 
$ FL[1] \xrightarrow{(\alpha \circ \beta)^{-1}} LT[-1]
\xrightarrow{\beta} R \xrightarrow{\Ract} RSL$,
so that 
$$ FL[1] \xrightarrow{\bareta} RSL \xrightarrow{\eta} FL[1] $$
is the identity map. Since all retracts in triangulated categories 
are split and since $L \xrightarrow{\actL} RSL
\xrightarrow{\eta}FL[1]$ is an exact triangle it follows that there
exists a map $RSL \xrightarrow{\baractL} L$ such that 
$$ L \oplus FL[1] \xrightarrow{(\actL)\oplus \bareta}
RSL \xrightarrow{(\baractL) \oplus \eta} L \oplus FL[1] $$
are mutually inverse isomorphisms. 
Similarly, since $F'F
\xrightarrow{\eqref{eqn-adjunction-counit-f'f}} \id$
and $\id \xrightarrow{\eqref{eqn-adjunction-unit-t't}} TT'$
are isomorphisms Lemmas \ref{lemma-gamma'-is-an-isomorphism} and 
\ref{lemma-f'r-rt'-2maps} imply that the map 
$F'R[-1] \xrightarrow{\alpha' \circ \beta'} RT'[1]$ is an isomorphism. 
Let $\barzeta$ be the map 
$ RT'[1] \xrightarrow{(\alpha' \circ \beta')^{-1}} 
F'R[-1] \xrightarrow{\beta'} L \xrightarrow{\actL} RSL$,
then there exists a map $RSL \xrightarrow{\barRact} R$ such 
that 
$$ R \oplus RT'[1] 
\xrightarrow{(\Ract)\oplus \barzeta} 
RSL
\xrightarrow{(\barRact)\oplus \zeta}
R \oplus RT'[1]
$$
are mutually inverse isomorphisms. 

Since $T'T \xrightarrow{\eqref{eqn-adjunction-counit-t't}} \B$ is an
isomorphism, it follows from the exactness of rows and columns in 
the diagram \eqref{eqn-t't-id-3x3-diagram} that 
$SLT[-1] \xrightarrow{S\beta} SR$ is an isomorphism. So is  
$SLT[-1] \xrightarrow{S\alpha \circ S\beta} SFL[1]$, and hence
so must also be $SR \xrightarrow{S\alpha} SFL[1]$. Then
$S(\alpha \circ \beta)^{-1} = (S\beta)^{-1} \circ (S\alpha)^{-1}$,
and hence the following diagram commutes 
\begin{align*}
\vcenter{
\xymatrix{
SFL[1] \ar[r]^{S\bareta}  
& 
SRSL. 
\\
SR 
\ar[u]^{S\alpha}_{\sim}
\ar[ur]_{S\Ract}
}
}
\end{align*}

Consider now the map $SFL[1] \rightarrow T'[1]$ which is adjoint to 
$FL[1] \xrightarrow{\zeta \circ \bareta} RT'[1]$. 
It filters through 
$$ SFL[1] \xrightarrow{S(\zeta \circ \bareta)} SRT'[1]$$
which we can re-write as  
$$ SFL[1] \xrightarrow{(S\alpha)^{-1}} SR \xrightarrow{S\Ract} SRSL
\xrightarrow{S\zeta} SRT'[1]$$
and $R \xrightarrow{\Ract} RSL \xrightarrow{\zeta} RT'[1]$ is the
zero map. We conclude that $FL[1] \xrightarrow{\zeta \circ \bareta} RT'[1]$ 
is adjoint to the zero map and hence itself is the zero map. 

Similarly, $S\alpha'$ and $S\beta'$ are isomorphisms and the following
diagram commutes
\begin{align*}
\vcenter{
\xymatrix{
SRT'[1] \ar[r]^{S\barzeta}  
& 
SRSL. 
\\
SL 
\ar[u]^{S\alpha'}_{\sim}
\ar[ur]_{S\actL}
}
}
\end{align*}
It follows, similarly, that 
$SL \xrightarrow{S(\barRact \circ \actL)} SR$ is the zero map and 
hence so is $L \xrightarrow{\barRact \circ \actL} R$. 

Observe now that the composition 
$$ L \oplus FL[1] \xrightarrow{(\actL)\oplus \bareta} RSL 
\xrightarrow{(\barRact)\oplus \zeta}
R \oplus RT'[1] $$
is an isomorphism and we have shown the compositions 
$$ L \xrightarrow{\actL} RSL \xrightarrow{\barRact} R 
\hskip 1cm \text{ and } \hskip 1cm
FL[1] \xrightarrow{\bareta} RSL \xrightarrow{\zeta} RT'[1] $$
to be the zero maps. It follows that the compositions 
$$ L \xrightarrow{\actL} RSL \xrightarrow{\zeta} RT'[1]
\hskip 1cm \text{ and } \hskip 1cm
FL[1] \xrightarrow{\bareta} RSL \xrightarrow{\barRact} R $$
are isomorphisms. The former composition is, by definition, 
the map $L \xrightarrow{\alpha'} RT'[1]$. It follows by
Lemma \ref{lemma-condition-twist-identifies-adjoints-is-equiv} 
that the condition \eqref{item-main-theorem-the-cotwist-identifies-adjoints}
holds, as desired. 
\end{proof}

\subsection{Applications to algebraic geometry}
\label{section-applications-to-algebraic-geometry}

In this section we interpret the results of Section  
\ref{section-spherical-bimodules} in the context of algebraic
geometry. 

Let $Z$ and $X$ be two separated schemes of finite type over $k$. 
Recall that for any $E \in D_{qc}(Z \times X)$ the 
\em Fourier-Mukai transform \rm $\Phi_E$ is the functor
$D_{qc}(Z) \rightarrow D_{qc}(X)$ defined by
$$ \rder\pi_{X *} \left( E \ldertimes \pi_Z^* (-) \right),$$
where $\pi_Z$ and $\pi_X$ are the projections from $Z \times X$
to $Z$ and $X$. Note that $\Phi_E$ 
doesn't apriori restrict to a functor $D(Z) \rightarrow D(X)$. 

As explained in Example 
\ref{example-quasi-functors-between-enhancements-of-DBCoh}
we can Morita enhance $D(Z)$ and $D(X)$ by smooth DG-algebras 
$\A$ and $\B$ whose classes in $\HoKcTrCat$ are the standard
enhancements of $D(Z)$ and $D(X)$. Moreover, $D(Z \times X)$
is Morita enhanced by the DG-algebra $\Aopp \otimes \B$ and
the following holds.  
Recall that Morita quasifunctors
$\A \rightarrow \B$ are identified naturally with the elements 
of $D^{\Bperf}(\AbimB)$. Since $\A$ is smooth, we have a natural inclusion 
$D^{\Bperf}(\AbimB) \hookrightarrow D_{c}(\AbimB)$. Thus 
to each Morita quasifunctor $\A \xrightarrow{F} \B$  
corresponds an element in $D_c(\AbimB)$ and so 
an element $E \in D(Z \times X)$. The 
Fourier-Mukai transform $\Phi_E$ restricts
to a functor $D(Z) \xrightarrow{\Phi_{E}} D(X)$
and this functor is isomorphic to the exact functor 
$D(Z) \rightarrow D(X)$ underlying $F$.

Similarly, 
$X \times Z$, $Z \times Z$ and $X \times X$
are Morita enhanced by
$\Bopp \otimes \A$, $\Aopp \otimes \A$ and $\Bopp \otimes \B$
with a similar 
correspondence between the Morita quasifunctors and the
Fourier-Mukai transforms. We identify implicitly 
$X \times Z$ with $Z \times X$ using the canonical isomorphism 
between the two. For any object $E$ in $D_c(\A)$, $D_c(\B)$, 
$D_c(\AbimB)$, etc. let $\overline{E}$ be the corresponding
object in $D(Z)$, $D(X)$, $D(Z \times X)$, etc.  

Let $\bar{S} \in D(Z \times X)$ be such that
the corresponding $S \in D_c(\AbimB)$ is $\A$- and $\B$-perfect. 
Let $L = S^{\tilde{\A}}$ and $R = S^{\tilde{\B}}$. 
These are $\A$-perfect and $\B$-perfect, respectively. 
Since $\A$ and $\B$ are smooth, $L$ and $R$
lie in $D_c(\BbimA)$ by
Cor. \ref{cor-for-smooth-A-perfect-means-perfect}. The corresponding objects $\bar{L}$ and $\bar{R}$ in $D(X \times Z)$
define Fourier-Mukai transforms $D(X) \xrightarrow{\Phi_{\bar{L}},
\Phi_{\bar{R}}} D(Z)$ which are left and right adjoint to 
$D(Z) \xrightarrow{\Phi_{\bar{S}}} D(X)$. The adjunction co-units and
units are the natural transformations of Fourier-Mukai transforms
induced by the derived trace and action maps
\begin{align}
\label{eqn-derived-trace-maps-for-bimodules}
SR \xrightarrow{\trace} \B  \quad &\text{ and } \quad  
LS \xrightarrow{\trace} \A  \\
\label{eqn-derived-action-maps-for-bimodules}
\B \xrightarrow{\action} SL \quad &\text{ and } \quad 
\A \xrightarrow{\action} RS.  
\end{align}

The co-twists $F,F' \in D(\AbimA)$ and the twists $T,T' \in D(\BbimB)$
of $S$ were defined in Section \ref{section-spherical-bimodules}
as the cones and the co-cones of the derived trace and action maps above. 
It follows from Cor. \ref{cor-for-smooth-A-perfect-means-perfect} 
that they are all compact objects. Hence we can define 
the \em co-twist \rm and the \em dual co-twist \rm of $\bar{S}$
to be the corresponding objects $\bar{F}$ and $\bar{F}' \in D(Z \times Z)$
and the \em the twist \rm and the \em dual twist \rm of $\bar{S}$ 
to be $\bar{T}$ and $\bar{T}'$ in $D(X \times X)$.
Finally, define 
\begin{align}
\label{eqn-r-to-fl(1)-FM}
\Phi_{\bar{R}} \rightarrow \Phi_{\bar{F}} \Phi_{\bar{L}} [1] \\
\label{eqn-lt(-1)-to-r-FM}
\Phi_{\bar{L}} \Phi_{\bar{T}}[-1] \rightarrow \Phi_{\bar{R}}
\end{align}
to be the natural transformations of Fourier-Mukai transforms
which correspond to the natural transformations 
$lt[-1] \xrightarrow{\eqref{eqn-lt(-1)-to-r}} r$
and $r \xrightarrow{\eqref{eqn-r-to-fl(1)}} fl[1]$ 
constructed in Section \ref{section-spherical-bimodules}. 

The algebras $\A$ and $\B$ are constructed as 
DG-$\eend$-algebras of $h$-injective strong 
generators $F_Z$ and $F_X$ of $D(Z)$ and $D(X)$. The functors
\begin{align}
\label{eqn-D(Z)-enhancement-structure}
\rder\homm_{D(Z)}(F_Z, -) \colon D(Z) \rightarrow D_c(\A) \\
\label{eqn-D(X)-enhancement-structure}
\rder\homm_{D(X)}(F_X, -) \colon D(X) \rightarrow D_c(\B) 
\end{align}
are the equivalences which give $\A$ and $\B$ the structure
of Morita enhancements of $D(Z)$ and $D(X)$. 
It follows from 
\cite[Theorem 6.3]{Lunts-CategoricalResolutionOfSingularities} 
that choosing a different generator $F'_X$ of e.g. $D(X)$ produces
a Morita-equivalent DG-algebra $\B'$ and
the Morita equivalence can be chosen so that the underlying exact 
equivalence $D_c(\B) \xrightarrow{\sim} D_c(\B')$ is compatible with 
enhancement equivalences \eqref{eqn-D(X)-enhancement-structure}
for $F_X$ and $F_X'$. More generally, it follows
that a different choice of generators $F'_Z$ and $F'_X$ produces 
Morita-equivalent DG-algebras 
$\A'$, $\B'$, $\A'\text{-}\B'$, $\A'\text{-}\A'$ and $\B'\text{-}\B'$
with Morita equivalences being compatible with the enhancement
equivalences as above. 

All the constructions from Section \ref{section-spherical-bimodules}
we used so far were defined entirely in terms of 
the derived duals $R$ and $L$ of $S$ and the derived trace and action
maps. One can check that the derived dualizing functors and 
the derived trace and action maps are preserved under Morita equivalences. 
Thus the objects $\bar{L}, \bar{R}$, $\bar{F}, \bar{F}'$, 
$\bar{T}, \bar{T}'$ and the natural transformations 
\eqref{eqn-r-to-fl(1)-FM}-\eqref{eqn-lt(-1)-to-r-FM} 
defined above depend only on $\bar{S}\in D(Z \times X)$ itself, and 
do not depend on our choice of generators $F_X$ and $F_Y$ 
of $D(Z)$ and $D(X)$. 

Though we have established that the above objects and maps 
are well-defined and are determined only by $\bar{S}\in D(Z \times X)$, 
to actually compute them in any practical scenario would require
explicit formulas for $\bar{L}$, $\bar{R}$ in terms of $\bar{S}$
as well as the explicit formulas for the maps in $D(X \times X)$
and $D(Z \times Z)$ which correspond to the derived trace and 
action maps. To this end we offer the following:
\begin{conj}
Let $\bar{S}\in D(Z \times X)$ be such that 
the corresponding $S \in D_c(\AbimB)$ is $\A$- and $\B$-perfect.
Then 
\begin{align*}
\bar{L} \simeq \rder\shhomm_{Z\times X}\left(\bar{S},
\pi^!_Z(\mathcal{O}_Z)\right) \\ 
\bar{R} \simeq \rder\shhomm_{Z\times X}\left(\bar{S},
\pi^!_X(\mathcal{O}_X)\right)
\end{align*}
and the maps in $D(Z \times Z)$ and $D(X \times X)$ which correspond 
to the derived trace and action maps 
\eqref{eqn-derived-trace-maps-for-bimodules}-\eqref{eqn-derived-action-maps-for-bimodules}
are isomorphic to the explicit maps written down in 
\cite{AnnoLogvinenko-OnTakingTwistsOfFourierMukaiFunctors}
and \cite{AnnoLogvinenko-OrthogonallySphericalObjectsAndSphericalFibrations}
which lift the adjunction co-units and units of Fourier-Mukai transforms
to the level of Fourier-Mukai kernels. 
\end{conj}

Finally, we need an intrinsic condition on 
$\bar{S} \in D(Z \times X)$ on the algebro-geometric side which 
ensures that 
the corresponding $S \in D_c(\AbimB)$ is $\A$- and $\B$-perfect. 

\begin{lemma}
Let $\bar{S} \in D(Z \times X)$. 
The Fourier-Mukai transform $\Phi_{\bar{S}}$ restricts to 
$D(Z) \rightarrow D(X)$ and this restriction has 
a left adjoint which is also a Fourier-Mukai transform
if and only if the corresponding object 
$S \in D_c(\AbimB)$ is $\A$- and $\B$-perfect.  
\end{lemma}
\begin{proof}
As explained above, $S \in D_c(\AbimB)$ is 
$\B$-perfect if and only if $\Phi_{\bar{S}}$ restricts 
to $D(Z) \rightarrow D(X)$. In such case 
$D_c(\A) \xrightarrow{(-) \ldertimes_\A S} D_c(\B)$
corresponds to 
$ D(Z) \xrightarrow{\Phi_{\bar{S}}} D(X)$. 

Suppose now $S$ is also $\A$-perfect. 
By 
Cor. \ref{cor-derived-M-and-M^A-and-M^B-adjunction}
the functor 
$(-) \ldertimes_\B S^{\tilde{\A}}$ is left adjoint to 
$(-) \ldertimes_\A S$. Moreover, since $S$ is $\A$-perfect, 
so is $S^{\tilde{\A}}$. Hence there exists  
an object in $D(X \times Z)$ which defines the Fourier-Mukai
transform $D(X) \rightarrow D(Z)$ which corresponds to 
$(-) \ldertimes_\B S^{\tilde{\A}}$. In particular, 
this Fourier-Mukai transform is left adjoint to $\Phi_{\bar{S}}$. 

Conversely, if there exists a Fourier-Mukai transform 
$D(X) \rightarrow D(Z)$ which is the left adjoint to $\Phi_{\bar{S}}$, 
let $L$ be the corresponding object of $D^{\Aperf}(\BbimA)$. 
Then $(-) \ldertimes_\B L$ is left adjoint to $(-) \ldertimes_\A S$ 
as functors between $D_c(\A)$ and $D_c(\B)$.  
But since derived tensor product commutes with infinite 
direct sums, these are, in fact, adjoint on the whole 
of $D(\A)$ and $D(\B)$. 

Since $L$ is $\A$-perfect, 
by Cor. \ref{cor-derived-M-and-M^A-and-M^B-adjunction}
the functor 
$(-) \ldertimes_\B L$ from $D(\A)$ to $D(\B)$
has a right adjoint 
$(-) \ldertimes_\A L^{\tilde{\A}}$.  
By uniqueness of adjoints we conclude that the functors 
$(-) \ldertimes_\A S$ and $(-) \ldertimes_\A L^{\tilde{\A}}$
are isomorphic. Since $L$ is $\A$-perfect, so is $L^{\tilde{\A}}$. 
Hence $(-) \ldertimes_\A L^{\tilde{\A}}$ takes compact objects
to compact objects, and hence so does $(-) \ldertimes_\A S$. 
We conclude that $S$ is also $\A$-perfect, as desired. 
\end{proof}

Theorem \ref{theorem-main-theorem} immediately implies
the following: 
\begin{theorem}
\label{theorem-main-theorem-for-coh-FM}
Let $\bar{S} \in D(Z \times X)$ be such that 
$\Phi_{\bar{S}}$ restricts to $D(Z) \rightarrow D(X)$
and this restriction has a left adjoint which is 
also a Fourier-Mukai transform.  

If any two of the following conditions hold:
\begin{enumerate}
\item 
\label{item-main-theorem-FM-the-twist-is-an-equivalence}
$\Phi_{\bar{T}}$ is an autoequivalence of $D(X)$
(``the twist is an equivalence'').  
\item 
\label{item-main-theorem-FM-the-cotwist-is-an-equivalence}
$\Phi_{\bar{F}}$ is an equivalence of $D(Z)$
(``the cotwist is an equivalence'').  
\item 
\label{item-main-theorem-FM-the-twist-identifies-adjoints}
$\Phi_{\bar{R}} \xrightarrow{\eqref{eqn-r-to-fl(1)-FM}}
\Phi_{\bar{F}} \Phi_{\bar{L}} [1]$
is an isomorphism of functors (``the twist identifies the adjoints'').  
\item 
\label{item-main-theorem-FM-the-cotwist-identifies-adjoints}
$\Phi_{\bar{L}} \Phi_{\bar{T}}[-1]
\xrightarrow{\eqref{eqn-lt(-1)-to-r-FM}}
\Phi_{\bar{R}}$ is an
isomorphism of functors (``the cotwist identifies the adjoints'').  
\end{enumerate}
then all four of them hold. If that happens, we say that 
$\bar{S}$ is \em spherical over $Z$\rm.
\end{theorem}

We can repeat all the arguments in this section 
using the framework of the Example
\ref{example-quasi-functors-between-enhancements-of-QCoh}
rather than the Example
\ref{example-quasi-functors-between-enhancements-of-DBCoh}. 
Thus we would work with large Morita enhancements 
of $D_{qc}(Z)$ and $D_{qc}(X)$, 
rather than with Morita enhancements of $D(Z)$ and $D(X)$. 
This yields a construction of \em twists \rm and \em co-twists \rm
as functors $D_{qc}(X) \rightarrow D_{qc}(X)$ and 
$D_{qc}(Z) \rightarrow D_{qc}(Z)$ and an analogue of
Theorem \ref{theorem-main-theorem-for-coh-FM}. 
However, we would have to impose the following condition
on the objects of $\bar{S} \in D_{qc}(Z \times X)$ which we work 
with: $\Phi_{\bar{S}}$ must have a left adjoint which is  
a Fourier-Mukai transform and they both
must send compact objects to compact objects. 


\section{Braiding criteria for spherical DG-functors} 
\label{section-braiding-criteria}

Let $\A_1$, $\A_2$ and $\B$ be small DG-categories and let 
$S_1 \in D(\AonebimB)$ and $S_2 \in D(\AtwobimB)$ be two 
spherical objects. We keep all the notation conventions of  
Section \ref{section-spherical-DG-functors}. E.g. $R_i$ denotes
$S_i^\lderB$, $S_i R_i$ denotes $R_i \ldertimes_{\A_i}
S_i$, $T_i$ denotes the cone of $S_i R_i \xrightarrow{\trace} \B$, etc.  

In particular, $M = \barAi \otimes_{\A_1} S_1 \otimes_{\B} \barB$ and 
$N = \barAj \otimes_{\A_2} S_2 \otimes_{\B} \barB$ are $h$-projective 
resolutions of $S_1$ and $S_2$. In this section it was possible 
to simplify a number of computations by replacing all homotopy 
trace maps 
$\MhdB \otimes_\A M \xrightarrow{\trace} \barB$
and $\NhdB \otimes_\A N \xrightarrow{\trace} \barB$
by their compositions with $\barB \xrightarrow{\tau} \B$. 
To keep the notation simple, we write  
$\MhdB \otimes_\A M \otimes_\B \xrightarrow{\trace} \B$
and $\NhdB \otimes_\A N \otimes_\B \xrightarrow{\trace} \B$
for these compositions throughout.


\subsection{Commutation}
\label{subsection-commutation}

By functoriality of the derived tensor product, the following
diagram commutes:
\begin{align}
\label{eqn-braiding-commutative-cube-in-H0-to-complete}
\vcenter{
\xymatrix{
&
S_2 R_2
\ar[rr]^{\id}
\ar[dd]_{\trace}
&
&
S_2 R_2
\ar[dd]^{\trace}
\\
S_1 R_1 S_2 R_2 
\ar[ur]^<<<<<<{\trace \otimes \id}
\ar[dd]_<<<<<<<{\id \otimes \trace}
& & 
S_2 R_2 S_1 R_1 
\ar[ur]^<<<<<<<{\id \otimes \trace}
\ar[dd]^<<<<<<<{\trace \otimes \id}
&
\\
&
\B
\ar[rr]_<<<<<<<<<<<<{\id}
&
&
\B
\\
S_1 R_1
\ar[rr]_<<<<<<<<<<<<{\id}
\ar[ur]^{\trace}
& &
S_1 R_1
\ar[ur]^{\trace}
&
}
}
\end{align}

The main result of this section is:
\begin{theorem}
\label{theorem-commutation-criterion}
Suppose there exists an isomorphism 
$$ S_1 R_1 S_2 R_2 \xrightarrow{\phi} S_2 R_2 S_1 R_1 $$
which makes the diagram
\eqref{eqn-braiding-commutative-cube-in-H0-to-complete} commute. 
Then 
$$ T_1 T_2 \simeq T_2 T_1. $$
\end{theorem}
\begin{proof}
 
By definition, $T_1 T_2$ is isomorphic in $D(\BbimB)$ to  
$$ \left\{ \NhdB \otimes_{\A_2} N \xrightarrow{\trace}
\underset{\degzero}{\B} \right\}
\otimes_\B \left\{ \MhdB \otimes_{\A_1} M \xrightarrow{\trace} 
\underset{\degzero}{\B}
\right\} $$
which by Lemma \ref{lemma-tensor-and-hom-of-twisted-complexes} is 
isomorphic to the convolution of
\begin{align}
\label{eqn-S1R1S2R2-twisted-complex}
\Bigl( \NhdB \otimes_{\A_2} N \otimes_\B \MhdB \otimes_{\A_1} M
\xrightarrow{\alpha}
\left( \NhdB \otimes_{\A_2} N \right) 
\oplus
\left( \MhdB \otimes_{\A_1} M \right)
\xrightarrow{\gamma}
\underset{\degzero}{\B}.\Bigr)
\end{align}
where $\alpha = (- \id \otimes \trace) \oplus (\trace\otimes \id)$
and $\gamma = \trace \oplus \trace$. 
Similarly, $T_2 T_1$ is isomorphic to the convolution of
\begin{align}
\label{eqn-S2R2S1R1-twisted-complex}
\Bigl( \MhdB \otimes_{\A_1} M \otimes_\B \NhdB \otimes_{\A_2} N
\xrightarrow{\beta}
\left( \NhdB \otimes_{\A_2} N \right)
\oplus
\left( \MhdB \otimes_{\A_1} M \right) 
\xrightarrow{\gamma}
\underset{\degzero}{\B}.\Bigr)
\end{align}
where $\beta = (\trace\otimes \id) \oplus (- \id \otimes \trace)$. 

By Theorem 
\ref{theorem-homotopy-equiv-of-twisted-complexes-differ-only-first-term}
to show that \eqref{eqn-S1R1S2R2-twisted-complex} and 
\eqref{eqn-S2R2S1R1-twisted-complex} are homotopy equivalent in
$\pretriag(\BmodB)$, and hence that $T_1T_2$ and $T_2T_1$ are
isomorphic in $D(\BbimB)$, it suffices to exhibit
\begin{align*}
f &\in \homm^0_{\BbimB}\left( 
\NhdB \otimes_{\A_2} N \otimes_\B \MhdB \otimes_{\A_1} M, 
\quad \MhdB \otimes_{\A_1} M \otimes_\B \NhdB \otimes_{\A_2} N
\right)\\
s_1 &\in \homm^{-1}_{\BbimB} \left(
\NhdB \otimes_{\A_2} N \otimes_\B \MhdB \otimes_{\A_1} M,  
\quad
\left( \NhdB \otimes_{\A_2} N \right)
\oplus
\left( \MhdB \otimes_{\A_1} M \right) 
\right)\\
s_2 &\in \homm^{-2}_{\BbimB} \left(
\NhdB \otimes_{\A_2} N \otimes_\B \MhdB \otimes_{\A_1} M,  
\quad
\B
\right)
\end{align*}
such that 
\begin{enumerate}
\item $f$ is a homotopy equivalence
\item $ds_1 = \alpha - \beta f$ 
\item $ds_2 = \gamma s_1$. 
\end{enumerate}

Since all the source bimodules are $h$-projective 
the $\homm^i$-spaces above are isomorphic to the 
$\ext^i$-spaces between the same objects in $D(\BbimB)$. 

In particular, we can lift the isomorphism 
$$ S_1 R_1 S_2 R_2 \xrightarrow{\phi} S_2 R_2 S_1 R_1 $$ 
in $D(\BbimB)$ to some homotopy equivalence
$$ f \in \homm^0_{\BbimB}\left( 
\NhdB \otimes_{\A_2} N \otimes_\B \MhdB \otimes_{\A_1} M, 
\quad \MhdB \otimes_{\A_1} M \otimes_\B \NhdB \otimes_{\A_2} N
\right).$$
The fact that $\phi$ makes 
\eqref{eqn-braiding-commutative-cube-in-H0-to-complete} commute
in $D(\BbimB)$ implies that $\alpha - \beta f$ vanishes 
in 
$$ \homm_{D(\BbimB}(S_1 R_1 S_2 R_2, S_1 R_1 \oplus S_2 R_2).$$
Hence we can find some 
$$ s_1 \in \homm^{-1}_{\BbimB} \left(
\NhdB \otimes_{\A_2} N \otimes_\B \MhdB \otimes_{\A_1} M,  
\quad
\left( \MhdB \otimes_{\A_1} M \right) 
\oplus
\left( \NhdB \otimes_{\A_2} N \right)
\right) $$
with $ds_1 = \alpha - \beta f$. But there is no apriori reason for
the class of $\gamma s_1$ to vanish in  
$\ext^{-1}_{D(\BbimB)}(S_1R_1S_2R_2, \B),$
which is what we need to warranty the existence of 
$$ s_2 \in \homm^{-2}_{\BbimB} \left(
\NhdB \otimes_{\A_2} N \otimes_\B \MhdB \otimes_{\A_1} M,  
\quad
\barB \right) $$
with $ds_2 = \gamma s_1$, whence as explained above the claim of 
this theorem would follow.

It suffices, however, to find 
$$ t_1 \in \homm^{-1}_{\BbimB} \left(
\NhdB \otimes_{\A_2} N \otimes_\B \MhdB \otimes_{\A_1} M,  
\quad
\left( \MhdB \otimes_{\A_1} M \right) 
\oplus
\left( \NhdB \otimes_{\A_2} N \right)
\right) $$
with $dt_1 = 0$ and $\gamma t_1 = \gamma s_1$ in 
$\ext^{-1}_{D(\BbimB)}(S_1R_1S_2R_2, \B)$. For if 
we then replace $s_1$ with $s_1 - t_1$ the condition 
$ds_1 = \alpha - \beta f$ would still hold, but the class of $\gamma s_1$ 
would now vanish in $\ext^{-1}_{D(\BbimB)}(S_1R_1S_2R_2, \B)$ as required. 
Thus it remains to show that the class $[\gamma s_1]$ in
$\ext^{-1}_{D(\BbimB)}\left(S_1R_1S_2R_2, \B\right)$ lifts 
with respect to
\begin{align}
\label{eqn-ext-minus-one-map-given-by-composit-with-gamma}
\ext^{-1}_{D(\BbimB)}\left(S_1R_1S_2R_2, S_1R_1 \oplus S_2 R_2\right)
\xrightarrow{\gamma(-)}
\ext^{-1}_{D(\BbimB)}\left(S_1R_1S_2R_2, \B \right)
\end{align}
to some class in 
$\ext^{-1}_{D(\BbimB)}\left(S_1R_1S_2R_2, S_1R_1 \oplus S_2 R_2\right)$. 

We claim that, in fact, 
\eqref{eqn-ext-minus-one-map-given-by-composit-with-gamma} is
surjective. Indeed, it follows from Prop. 
\ref{prps-M^B-and-M^A-are-homotopy-adjoints-of-M-via-S-SRS-S-maps}
via the usual adjunction-type argument that for any 
$N_1 \in D(\A_2\text{-}\A_1)$ and $N_2 \in D(\BbimB)$ the
map 
\begin{align}
\label{eqn-S_1-*-R_2-adjunction}
\ext^{-1}_{D(\BbimB)}(S_1 N_1 R_2, N_2) \longrightarrow
\ext^{-1}_{D(\BbimB)}(N_1, R_1 N_2 S_2) 
\end{align}
given by
$$ \alpha \quad \mapsto \quad 
N_1 \xrightarrow{\action N_1 \action} R_1 S_1 N_1 R_2 S_2
\xrightarrow{R_1 \alpha S_2} R_1 N_2 S_2 $$
is a functorial isomorphism. We thus have a commutative diagram
\begin{align}
\vcenter{
\xymatrix{
\ext^{-1}_{D(\BbimB)}(S_1R_1S_2R_2, S_1R_1 \oplus S_2 R_2) 
\ar[rr]^{\quad\eqref{eqn-ext-minus-one-map-given-by-composit-with-gamma}}
\ar[d]_{\sim}^{\eqref{eqn-S_1-*-R_2-adjunction}}
& &
\ext^{-1}_{D(\BbimB)}(S_1R_1S_2R_2, \B) 
\ar[d]_{\sim}^{\eqref{eqn-S_1-*-R_2-adjunction}}
\\
\ext^{-1}_{D(\A_2\text{-}\A_1)}(R_1S_2, R_1S_1R_1S_2 \oplus R_1S_2R_2S_2) 
\ar[rr]^>>>>>>>>>>{(R_1 \gamma S_2)(-)}
& &
\ext^{-1}_{D(\A_2\text{-}\A_1)}(R_1S_2, R_1 S_2) 
} 
}.
\end{align}
The map $R_1\gamma S_2$ is the map 
$$
R_1S_1R_1S_2 \oplus R_1S_2R_2S_2
\xrightarrow{R_1\trace S_2 \oplus R_1 \trace S_2}
R_1S_2
$$
and by Prop
\ref{prps-M^B-and-M^A-are-homotopy-adjoints-of-M-via-S-SRS-S-maps}
the map 
$$
R_1 S_2
\xrightarrow{\frac{1}{2}\action R_1S_2 \oplus \frac{1}{2} R_1S_2 \action}
R_1S_1R_1S_2 \oplus R_1S_2R_2S_2
$$
is its left inverse in $D(\A_2\text{-}\A_1)$. Therefore  
$$
\ext^{-1}_{D(\A_2\text{-}\A_1)}(R_1S_2, R_1S_1R_1S_2 \oplus R_1S_2R_2S_2) 
\xrightarrow{(R_1 \gamma S_2)(-)}
\ext^{-1}_{D(\A_2\text{-}\A_1)}(R_1S_2, R_1 S_2) 
$$
is surjective and hence so is
\eqref{eqn-ext-minus-one-map-given-by-composit-with-gamma} as
desired.
\end{proof}


\subsection{Braiding.}
\label{section-criterion-for-braiding}

Define 
\begin{align}
\label{eqn-definition-of-O_i}
O_i=F_i\{L_iS_jR_jS_i \xrightarrow{\trace\circ(L_i \trace S_i)}\A_i\}
\quad\quad \in D(\A_i\text{-}\A_i)
\end{align}
where $i,j\in\{1,2\}$, $i\ne j$. For spherical $S_1, S_2$ the natural
map $R_i[-1] \xrightarrow{\alpha} F_i L_i$ is an isomorphism and it
identifies the map in \eqref{eqn-definition-of-O_i}
with the map 
$$ R_i S_j R_j S_i[-1] \xrightarrow{R_i \trace S_i} R_i S_i[-1] \rightarrow F_i
$$
whose second composant comes from the exact triangle 
$F_i \rightarrow \A_i \rightarrow R_i S_i.$
Thus $O_1$ and $O_2$ are isomorphic to the convolutions of 
the twisted
complexes
\begin{align*}
\O_1 \overset{\text{def}}{=} 
\left(
\underset{\degzero}{
\left( M \otimes_{\B} \NhdB \otimes_{\A_2} N \otimes_{\B} \MhdB \right)
\oplus \barAi}
\xrightarrow{(\id \otimes \trace \otimes \id)\oplus(-\action)}
{M \otimes_\B \MhdB}
\right)
\\
\O_2 
\overset{\text{def}}{=}
\left(
\underset{\degzero}{
\left( N \otimes_{\B} \MhdB \otimes_{\A_1} M \otimes_{\B} \NhdB \right)
\oplus \barAj}
\xrightarrow{(\id \otimes \trace \otimes \id)\oplus(-\action)}
{N \otimes_\B \NhdB}
\right).
\end{align*}

There are natural maps
\begin{align}
\label{eqn-S1O1R1-to-S1R1S2R2-plus-S2R2S1R1}
S_1 O_1 R_1 \rightarrow S_1R_1S_2R_2 \oplus S_2R_2S_1R_1 
\\
\label{eqn-S2O2R2-to-S1R1S2R2-plus-S2R2S1R1}
S_2 O_2 R_2 \rightarrow S_1R_1S_2R_2 \oplus S_2R_2S_1R_1
\end{align}
where \eqref{eqn-S1O1R1-to-S1R1S2R2-plus-S2R2S1R1}
is the map induced by 
\begin{align}
\label{eqn-S_1O_1R_1-to-S_2R_2S_1R_1}
\MhdB \otimes_{\A_1}
M \otimes_{\B} \NhdB \otimes_{\A_2} N \otimes_{\B} \MhdB
\otimes_{\A_1} M
\xrightarrow{\id \otimes \id \otimes \trace \otimes \id \otimes \id } &
\MhdB \otimes_{\A_1} M \otimes_{\B} \NhdB \otimes_{\A_2} N 
\\
\label{eqn-S_1O_1R_1-to-S_1R_1S_2R_2} \MhdB \otimes_{\A_1}
M \otimes_{\B} \NhdB \otimes_{\A_2} N \otimes_{\B} \MhdB
\otimes_{\A_1} M
\xrightarrow{\trace \otimes \id \otimes \id \otimes \id \otimes \id } &
\NhdB \otimes_{\A_2} N \otimes_{\B} \MhdB \otimes_{\A_1} M.
\end{align}
and \eqref{eqn-S2O2R2-to-S1R1S2R2-plus-S2R2S1R1} is defined 
analogously. 

The main result of this section is:
\begin{theorem}
\label{theorem-braiding-criterion}
Suppose there exists an isomorphism 
$$ S_1 O_1 R_1 \xrightarrow{\phi} S_2 O_2 R_2 $$
which commutes with the maps
\eqref{eqn-S1O1R1-to-S1R1S2R2-plus-S2R2S1R1}
and \eqref{eqn-S2O2R2-to-S1R1S2R2-plus-S2R2S1R1}.
Then
$$ T_1 T_2 T_1 \simeq T_2 T_1 T_2. $$
\end{theorem}
\begin{proof}
$T_1T_2T_1$ is isomorphic
by the Cube Completion Lemma \ref{lemma-the-cube-completion-lemma}
to the convolution of the twisted cube
\begin{footnotesize}
\begin{align}\label{diag-origami-initial}
\vcenter{
\xymatrix@C=2cm{
 &
\MhdB\otimes M\otimes \NhdB\otimes N \otimes \MhdB\otimes M  
\ar[ld]_{\id \otimes \trace}
\ar[d]_{\trace \otimes \id}
\ar[rd]^{\quad - \id \otimes \trace \otimes \id}
& 
\\ 
\MhdB\otimes M\otimes \NhdB\otimes N 
\ar[d]_{\trace \otimes \id}
\ar[rd]_<<<<<<<<{- \id \otimes \trace}
&
\NhdB\otimes N\otimes \MhdB\otimes M 
\ar[ld]_<<<<<<<<<<<<{-\id \otimes \trace}
\ar[rd]^<<<<<<<<<<<<{\trace \otimes \id}
&
\MhdB\otimes M\otimes \MhdB\otimes M 
\ar[ld]^<<<<<<<<{- \id \otimes \trace}
\ar[d]^{\trace\otimes\id}
\\
\NhdB\otimes N 
\ar[rd]^{\trace}
&
\MhdB\otimes M 
\ar[d]^{\trace}
&
\MhdB\otimes M 
\ar[ld]_{\trace}
\\
&  \underset{\degzero}{\B}. & 
}
}
\end{align}
\end{footnotesize}

We now use the isomorphism 
$\left(\MhdB\otimes M\right) 
\oplus \left(\MhdB\otimes M\right)
\xrightarrow{
\left( 
\begin{smallmatrix}
1 & 1   \\
1 & -1 
\end{smallmatrix}
\right)
}
\left(\MhdB\otimes M\right)
\oplus \left(\MhdB\otimes M\right)$
to rewrite the total complex of \eqref{diag-origami-initial} as:

\begin{footnotesize}
\begin{align}\label{diag-origami-step1}
\vcenter{
\xymatrix@C=2cm{
 &
\MhdB\otimes M\otimes \NhdB\otimes N \otimes \MhdB\otimes M  
\ar[ld]_{\id \otimes \trace}
\ar[d]_{\trace \otimes \id}
\ar[rd]^{\quad - \id \otimes \trace \otimes \id}
& 
\\ 
\MhdB\otimes M\otimes \NhdB\otimes N 
\ar@{}[r]|{\bigoplus}
\ar[d]_{\trace \otimes \id}
\ar[rd]_<<<<<<<<{- \id \otimes \trace}
\ar[rrd]|<<<<<<<<<{\;- \id \otimes \trace\;}
&
\NhdB\otimes N\otimes \MhdB\otimes M 
\ar@{}[r]|{\bigoplus}
\ar[ld]_<<<<<<<<<<<<{-\id \otimes \trace}
\ar[rd]^<<<<<<<<<<<<{- \trace \otimes \id}
\ar[d]^<<{\trace \otimes \id}
&
\MhdB\otimes M\otimes \MhdB\otimes M 
\ar[d]^{-2 \left( \trace \otimes \id\right)}
\\
\NhdB\otimes N 
\ar@{}[r]|{\bigoplus}
\ar[rd]^{\trace}
&
\MhdB\otimes M 
\ar@{}[r]|{\bigoplus}
\ar[d]^{\trace}
&
\MhdB\otimes M 
\\
&  \underset{\degzero}{\B}. & 
}
}
\end{align}
\end{footnotesize}

Let $X$ and $Y$ be the full subcomplexes of \eqref{diag-origami-step1} 
which comprise its left two columns and its right column,
respectively.  
Since the right column has no outgoing arrows, its incoming arrows 
define a closed degree $0$ morphism $X \xrightarrow{\rho} Y$ 
whose total complex is \eqref{diag-origami-step1}. Let 
$ Y' = \left(\MhdB \otimes M 
\xrightarrow{-\frac{1}{2} \id \otimes \action \otimes \id} 
\underset{\text{deg.-2}}{\MhdB \otimes M \otimes \MhdB \otimes M}
\right)$.
Lemma \ref{lemma-gamma-is-an-isomorphism} yields with a homotopy
equivalence $Y \xrightarrow{\gamma'} Y'$. The total complex
of $X \xrightarrow{\rho} Y$ is then homotopy equivalent to 
the total complex of $X \xrightarrow{\gamma' \circ \rho} Y'$. 
Thus \eqref{diag-origami-step1} is homotopy equivalent to 
the twisted complex:
\begin{footnotesize}
\begin{align}\label{diag-origami-final1}
\vcenter{
\xymatrix@C=1.75cm{
&
\MhdB\otimes M\otimes \NhdB\otimes N \otimes \MhdB\otimes M  
\ar[ld]_{\id \otimes \trace}
\ar[d]_{\trace \otimes \id}
\ar[rd]^{\quad - \id \otimes \trace \otimes \id}
\ar@{}[r]|{\bigoplus}
& 
\MhdB\otimes M 
\ar[d]^{-\frac{1}{2} \id \otimes \action \otimes \id}
\\ 
\MhdB\otimes M\otimes \NhdB\otimes N 
\ar@{}[r]|{\bigoplus}
\ar[rd]_<<<<<<<<{- \id \otimes \trace}
\ar[d]_{\trace \otimes \id}
&
\NhdB\otimes N\otimes \MhdB\otimes M 
\ar@{}[r]|{\bigoplus}
\ar[ld]_<<<<<<<<<<<<{-\id \otimes \trace}
\ar[d]^<<{\trace \otimes \id}
&
\MhdB\otimes M\otimes \MhdB\otimes M 
\\
\NhdB\otimes N 
\ar@{}[r]|{\bigoplus}
\ar[rd]^{\trace}
&
\MhdB\otimes M 
\ar[d]^{\trace}
&
\\
&  \underset{\degzero}{\B}. & 
}
}
\end{align}
\end{footnotesize}

Now observe that $\MhdB \otimes \mathcal{O}_1 \otimes M[-3]$
is homotopy equivalent to the following initial subcomplex 
of \eqref{diag-origami-final1}:
\begin{footnotesize}
\begin{align}\label{diag-origami-final-piece}
\vcenter{
\xymatrix@C=1.75cm{
&
\MhdB\otimes M\otimes \NhdB\otimes N \otimes \MhdB\otimes M  
\ar[rd]^{\quad - \id \otimes \trace \otimes \id}
\ar@{}[r]|{\quad\quad\bigoplus}
& 
\MhdB\otimes M 
\ar[d]^{-\frac{1}{2} \id \otimes \action \otimes \id}
\\ 
&
&
\underset{\text{deg. -2}}{\MhdB\otimes M\otimes \MhdB\otimes M}.
}
}
\end{align}
\end{footnotesize}
By the same argument as above 
\eqref{diag-origami-final1} is homotopy equivalent to 
the twisted complex
\begin{equation}\label{eqn-braiding-twisted-complex1}
\MhdB\!\otimes \O_1 \otimes\! M
\xrightarrow{\alpha}
\left(\MhdB\!\otimes\!M\!\otimes\!\NhdB\!\otimes\!N\right)\oplus
\left(\NhdB\!\otimes\!N\!\otimes\!\MhdB\!\otimes\!M\right) 
\xrightarrow{\gamma}
\left(\MhdB\!\otimes\!M\right) \oplus \left(\NhdB\!\otimes\!N\right)
\xrightarrow{\delta}
\underset{\degzero}{\B}.
\end{equation}

Similarly, $T_2T_1T_2$ is isomorphic to the convolution of the twisted complex
\begin{equation}\label{eqn-braiding-twisted-complex2}
\NhdB\!\otimes \O_2 \otimes\! N
\xrightarrow{\alpha}
\left(\MhdB\!\otimes\!M\!\otimes\!\NhdB\!\otimes\!N\right)\oplus
\left(\NhdB\!\otimes\!N\!\otimes\!\MhdB\!\otimes\!M\right) 
\xrightarrow{\gamma}
\left(\MhdB\!\otimes\!M\right) \oplus \left(\NhdB\!\otimes\!N\right)
\xrightarrow{\delta}
\underset{\degzero}{\B}.
\end{equation}

The complexes \ref{eqn-braiding-twisted-complex1} and \ref{eqn-braiding-twisted-complex2}
descend to the following complexes of objects in $D(\BbimB)$:
\begin{equation}
S_1O_1R_1
\xrightarrow{\eqref{eqn-S1O1R1-to-S1R1S2R2-plus-S2R2S1R1}}
S_1R_1S_2R_2\oplus S_2R_2S_1R_1
\xrightarrow{
\left( 
\begin{smallmatrix}
S_1R_1\trace & - \trace S_1 R_1 \\
- \trace S_2 R_2 & S_2R_2\trace 
\end{smallmatrix}
\right)
}
S_1R_1\oplus S_2R_2
\xrightarrow{\trace \oplus \trace}
\B
\end{equation}
\begin{equation}
S_2O_2R_2
\xrightarrow{\eqref{eqn-S2O2R2-to-S1R1S2R2-plus-S2R2S1R1}}
S_1R_1S_2R_2\oplus S_2R_2S_1R_1
\xrightarrow{
\left( 
\begin{smallmatrix}
S_1R_1\trace & - \trace S_1 R_1 \\
- \trace S_2 R_2 & S_2R_2\trace 
\end{smallmatrix}
\right)
}
S_1R_1\oplus S_2R_2
\xrightarrow{\trace \oplus \trace}
\B.
\end{equation}

By Theorem 
\ref{theorem-homotopy-equiv-of-twisted-complexes-differ-only-first-term}
to show that \ref{eqn-braiding-twisted-complex1} and \ref{eqn-braiding-twisted-complex2}
are homotopy equivalent in
$\pretriag(\BmodB)$, and hence that $T_1T_2T_1$ and $T_2T_1T_2$ are
isomorphic in $D(\BbimB)$, it suffices to exhibit
\begin{align*}
f &\in \homm^0_{\BbimB}\left( 
\MhdB\otimes \O_1 \otimes M, 
\quad 
\NhdB\otimes \O_2 \otimes N
\right)\\
s_1 &\in \homm^{-1}_{\BbimB} \left(
\MhdB \otimes \O_1 \otimes M,  
\quad
\left(\MhdB\otimes M \otimes \NhdB \otimes N\right) 
\oplus 
\left(\NhdB\otimes N \otimes \MhdB \otimes M\right)
\right)\\
s_2 &\in \homm^{-2}_{\BbimB} \left(
\MhdB \otimes \O_1 \otimes M,  
\quad
\left(\MhdB\otimes M\right) \oplus
\left(\NhdB\otimes N\right)
\right)\\
s_3 &\in \homm^{-3}_{\BbimB} \left(
\MhdB \otimes \O_1 \otimes M,  
\quad
\B
\right)
\end{align*}
such that 
\begin{enumerate}
\item $f$ is a homotopy equivalence
\item $ds_1 = \alpha - \beta f$ 
\item $ds_2 = \gamma s_1$
\item $ds_3= - \delta s_2$. 
\end{enumerate}

As in the proof of Theorem \ref{theorem-commutation-criterion}
we can lift $\phi$ to some homotopy equivalence $f$ and 
the existence of some
$$
\tilde{s}_1 \in \homm^{-1}_{\BbimB} \left(
\MhdB \otimes \O_1 \otimes M,  
\quad
\left(\MhdB\otimes M \otimes \NhdB \otimes N\right) 
\oplus 
\left(\NhdB\otimes N \otimes \MhdB \otimes M\right)
\right)
$$
with $d\tilde{s}_1 = \alpha - \beta f$ 
is guaranteed by the commutation of $\phi$ with 
\eqref{eqn-S1O1R1-to-S1R1S2R2-plus-S2R2S1R1}-\eqref{eqn-S2O2R2-to-S1R1S2R2-plus-S2R2S1R1}.
Since $\gamma \alpha = \gamma \beta = 0$ we have  
$d(\gamma \tilde{s}_1) = 0$. Thus 
$\gamma \tilde{s}_1$ defines the class 
$
[\gamma \tilde{s}_1] \in \Ext^i_{D(\BbimB)}(S_1O_1R_1, S_1R_1\oplus S_2R_2)
$
and since $\delta \gamma = 0$ the composition 
$\delta[\gamma \tilde{s}_1]$ vanishes in 
$
\Ext^i_{D(\BbimB)}(S_1O_1R_1, \B).
$
By Cor. \ref{cor-braiding-existence-of-t1-and-t2} below 
there exists some
$$
t_1 \in \homm^{-1}_{\BbimB} \left(
\MhdB \otimes \O_1 \otimes M,  
\quad
\left(\MhdB\otimes M \otimes \NhdB \otimes N\right) 
\oplus 
\left(\NhdB\otimes N \otimes \MhdB \otimes M\right)
\right)
$$
such that $dt_1=0$ and $[\gamma \tilde{s}_1]=[\gamma t_1]$ in 
$
\Ext^i_{D(\BbimB)}(S_1O_1R_1, S_1R_1\oplus S_2R_2).
$
Set $s_1=\tilde{s}_1-t_1$. We still have $ds_1 = \alpha - \beta f$, but 
the class of $\gamma s_1$ in
$
\Ext^i_{D(\BbimB)}(S_1O_1R_1, S_1R_1\oplus S_2R_2)
$
is zero, so there exists
$$
\tilde{s}_2 \in \homm^{-2}_{\BbimB} \left(
\MhdB \otimes \O_1 \otimes M,  
\quad
\left(\MhdB\otimes M\right) \oplus
\left(\NhdB\otimes N\right)
\right)
$$
with $d\tilde{s}_2=\gamma s_1$. Since $\delta \gamma = 0$
we have $d(\delta \tilde{s}_2)=0$. Again, 
by Cor. \ref{cor-braiding-existence-of-t1-and-t2}
there exists 
$$
t_2 \in \homm^{-2}_{\BbimB} \left(
\MhdB \otimes \O_1 \otimes M,  
\quad
\left(\MhdB\otimes M\right) \oplus
\left(\NhdB\otimes N\right)
\right)
$$
with $dt_2=0$ and $[\delta t_2]=[\delta\tilde{s}_2]$. Set 
$s_2=\tilde{s}_2-t_2$. We still have $ds_2 = \gamma s_1$, 
but the class of $\delta s_2$ in $\Ext^i_{D(\BbimB)}(S_1O_1R_1, \B)$
is zero, so
there exists 
$$s_3\in \homm^{-3}_{\BbimB} \left(\MhdB\otimes\O_1\otimes M,
\B\right)$$ 
with $ds_3=-\delta s_2$.
\end{proof}

\begin{lemma}\label{lemma-diagram-of-ext-groups}
There is a diagram of $\Ext$ groups in $D(\BbimB)$
\begin{equation}\label{diagram-ext-groups-for-braiding}
\xymatrix{
& &
\Ext^i_{D(\BbimB)}(*,S_1R_1)
\ar@<2pt>[lld]^{\eta_1}
\ar@<2pt>[rd]^{\mu_1}
& & \\
\Ext^i_{D(\BbimB)}(*,\B)
\ar@<2pt>[rru]^{\kappa_1}
\ar@<-2pt>[rrd]_{\kappa_2}
& & & 
\Ext^i_{D(\BbimB)}(*, S_1R_1S_2R_2\oplus S_2R_2S_1R_1) 
\ar@<2pt>[lu]^{\nu_1}
\ar@<-2pt>[ld]_{\nu_2}
\\
& &
\Ext^i_{D(\BbimB)}(*, S_2R_2)
\ar@<-2pt>[llu]_{\eta_2}
\ar@<-2pt>[ru]_{\mu_2}
& & 
}\end{equation}
where $*$ can mean $S_1O_1R_1$ or $S_2O_2R_2$ 
(since they are isomorphic in the derived category).

Moreover, $\eta_i\kappa_i=\id$ and $\nu_i \mu_i=\id$, 
while 
$\nu_2\mu_1 = -\kappa_2\eta_1$, 
$\nu_1\mu_2  = -\kappa_1\eta_2$ and
$\eta_1\nu_1 = -\eta_2\nu_2$.
\end{lemma}
\begin{proof}
Let $\nu_1$ be the map
$S_1R_1S_2R_2\oplus S_2R_2S_1R_1
\xrightarrow{- S_1 R_1 \trace \oplus \trace S_1R_1}
S_1R_1$.
Similarly, let $\nu_2$ be the map
$S_1R_1S_2R_2\oplus S_2R_2S_1R_1
\xrightarrow{\trace S_2R_2 \oplus -S_2R_2\trace}
S_2R_2. $
Let $\eta_1$ and $\eta_2$ be the trace maps
$S_1R_1 \xrightarrow{\trace} \B$ and 
$S_2R_2 \xrightarrow{\trace} \B$. 

Let $\mu_1$ be the composition 
\begin{tiny}
\begin{align}
\label{eqn-construction-mu1}
& \Ext^i_{D(\BbimB)}(S_2O_2R_2,S_1R_1)
\xrightarrow{S_2R_2(-)S_2R_2}
\Ext^i_{D(\BbimB)}(S_2R_2S_2O_2R_2S_2R_2, S_2R_2S_1R_1S_2R_2)
\xrightarrow{{S_2 \action O_2 \action R_2}}\\
\nonumber
\longrightarrow
&
\Ext^i_{D(\BbimB)}(S_2O_2R_2, S_2R_2S_1R_1S_2R_2)
\xrightarrow{
-\frac{1}{2}\left(\trace S_1R_1S_2R_2 \right)
\oplus
\frac{1}{2}\left(S_2R_2S_1R_1\trace\right)
}
\Ext^i_{D(\BbimB)}(S_2O_2R_2,S_1R_1S_2R_2\oplus S_2R_2S_1R_1).
\end{align}
\end{tiny}
Let $\kappa_1$ be the composition 
\begin{tiny}
\begin{align}
& \Ext^i_{D(\BbimB)}(S_1O_1R_1,\B)
\xrightarrow{S_1 R_1 (-) S_1 R_1}
\Ext^i_{D(\BbimB)}(S_1R_1S_1O_1R_1S_1R_1,S_1R_1S_1R_1)
\xrightarrow{S_1 \action O_1 \action R_1}
\\
\nonumber
\longrightarrow
&
\Ext^i_{D(\BbimB)}(S_1O_1R_1,S_1R_1S_1R_1)
\xrightarrow{S_1R_1\trace}
\Ext^i_{D(\BbimB)}(S_1O_1R_1,S_1R_1).
\end{align}
\end{tiny}
The maps $\mu_2$ and $\kappa_2$ are defined analogously.

We have $\eta_1\nu_1=-\eta_2\nu_2$ by functoriality
of the tensor product. The relations $\eta_i\kappa_i=\id$ 
and $\nu_i\mu_i=\id$ are verified directly 
using Prop.
\ref{prps-M^B-and-M^A-are-homotopy-adjoints-of-M-via-S-SRS-S-maps}. 
Let us prove that $\nu_2\mu_1=-\kappa_2\eta_1$. 
Consider the composition
\begin{tiny}
\begin{align}
\label{eqn-another-kappa2}
& \Ext^i_{D(\BbimB)}(S_2O_2R_2,\B)
\xrightarrow{S_2R_2(-)S_2R_2}
\Ext^i_{D(\BbimB)}(S_2R_2S_2O_2R_2S_2R_2, S_2R_2S_2R_2)
\xrightarrow{S_2 \action O_2 \action R_2}
\\
\nonumber
\longrightarrow
& \Ext^i_{D(\BbimB)}(S_2O_2R_2, S_2R_2S_2R_2)
\xrightarrow{-\frac{1}{2}(\trace S_2 R_2)\oplus\frac{1}{2}(S_2R_2\trace)}
\Ext^i_{D(\BbimB)}(S_2O_2R_2,S_2R_2\oplus S_2R_2).
\end{align}
\end{tiny}
and the map 
\begin{tiny}
\begin{align}
\label{eqn-trS_2R_2-oplus-S_2R_2tr}
\Ext^i_{D(\BbimB)}(S_2O_2R_2,S_1R_1S_2R_2\oplus S_2R_2S_1R_1)
\xrightarrow{\trace S_2R_2 \oplus S_2R_2\trace}
\Ext^i_{D(\BbimB)}(S_2O_2R_2,S_2R_2\oplus S_2R_2).
\end{align}
\end{tiny}

Applying the map $S_1 R_1 \xrightarrow{\trace} \B$ to every 
composant of \eqref{eqn-construction-mu1} and using functoriality
we see that the square
\begin{align*}
\vcenter{
\xymatrix{
\Ext^i_{D(\BbimB)}(S_2O_2R_2, S_1R_1) 
\ar[rrr]^{\eta_1}
\ar[d]_{\mu_1}
& & &
\Ext^i_{D(\BbimB)}(S_2O_2R_2,\B)
\ar[d]^{\eqref{eqn-another-kappa2}}
\\
\Ext^i_{D(\BbimB)}(S_2O_2R_2,S_1R_1S_2R_2\oplus S_2R_2S_1R_1)
\ar[rrr]^<<<<<<<<<<<<<<{\eqref{eqn-trS_2R_2-oplus-S_2R_2tr}}
& & &
\Ext^i_{D(\BbimB)}(S_2O_2R_2,S_2R_2\oplus S_2R_2).
}
} 
\end{align*}
commutes. By inspection, the composition of  
$\eqref{eqn-another-kappa2}$ with the map 
\begin{align}
\label{eqn-Id-minus-Id-map}
\Ext^i_{D(\BbimB)}(S_2O_2R_2,S_2R_2\oplus S_2R_2)
\xrightarrow{\id \oplus -\id}
\Ext^i_{D(\BbimB)}(S_2O_2R_2,S_2R_2)
\end{align}
is $-\kappa_2$, while the composition of  
\eqref{eqn-trS_2R_2-oplus-S_2R_2tr} with
\eqref{eqn-Id-minus-Id-map} is $\nu_2$. 
It follows that $\nu_2\mu_1=-\kappa_2\eta_1$, as desired.
\end{proof}

\begin{cor}
\label{cor-braiding-existence-of-t1-and-t2}
The sequence 
\begin{align*}
\Ext^i_{D(\BbimB)}(S_1O_1R_1, S_1R_1S_2R_2\oplus S_2R_2S_1R_1)
\xrightarrow{\gamma}
\Ext^i_{D(\BbimB)}(S_1O_1R_1, S_1R_1\oplus S_2R_2)
\xrightarrow{\delta}
\Ext^i_{D(\BbimB)}(S_1O_1R_1, \B)
\end{align*}
is exact in its middle term and surjective onto its last term.
\end{cor}


\appendix
\section{On homotopy equivalences of twisted complexes}
\label{section-on-homotopy-equivalences-of-twisted-complexes}

Let $\C$ be a strongly pretriangulated DG-category. The example one
wants to keep in mind is $\hprojA$ for some DG-category $\A$, so 
that $H^0(\C) = D(\A)$. Let $(E_{i}, q_{ij})$ be a twisted complex 
over $\C$. The objects $E_i$ and the degree $0$ morphisms $q_{i(i+1)}$
form an ordinary differential complex over $H^0(\C)$:
$$ \dots \xrightarrow{q_{(i-2)(i-1)}} E_{i-1}  
         \xrightarrow{q_{(i-1)i}} E_{i}  
         \xrightarrow{q_{i(i+1)}} E_{i+1}
         \xrightarrow{q_{(i+1)(i+2)}} \dots $$ 

Let $(E_i, q_{ij})$ and $(F_i, r_{ij})$ be two twisted complexes over $\C$. 
We would like to know when their convolutions $\left\{E_i, q_{ij}\right\}$ 
and $\left\{F_i, r_{ij}\right\}$ are isomorphic in $H^0(\C)$. 
Since $\C$ was assumed to be strongly pretriangulated constructing 
isomorphism of $\left\{E_i, q_{ij}\right\}$ 
and $\left\{F_i, r_{ij}\right\}$ in $H^0(\C)$ is the same thing
as constructing a homotopy equivalence of $(E_i, q_{ij})$ and 
$(F_i, r_{ij})$ in $\pretriag(\C)$. 

Suppose that the underlying differential complexes of 
$(E_i, q_{ij})$ and $(F_i, r_{ij})$ are isomorphic, more specifically -- 
that we have a set of isomorphisms $E_i \xrightarrow{f_i} F_i$ in 
$H^0(\C)$ which gives an isomorphism of these differential complexes. 
This alone doesn't ensure that $\left\{E_i, q_{ij}\right\}$ 
and $\left\{F_i, r_{ij}\right\}$ are isomorphic in $H^0(\C)$, 
since the same differential complex over $H^0(\C)$ can, in general, 
be lifted to several non-homotopically equivalent twisted complexes 
over $\C$. Thus the question: \em what are the sufficient
conditions on $f_i$ for us to be able to cook up a homotopy
equivalence of $(E_i, q_{ij})$ and $(F_i, r_{ij})$ from them\rm ?

When trying to construct this homotopy equivalence even in simplest 
cases, one encounters a number of conditions which, at first glance,
seem unavoidable, but in fact are redundant:  
\begin{exmpl}
\label{exmpl-homotopy-equivalences-b1-example}
Let $E \xrightarrow{q} G$ and $F\xrightarrow{r} G$
be two twisted complexes over $\C$. Let $E \xrightarrow{f} F$ 
be a homotopy equivalence in $\C$, such that the square 
\begin{align}
\label{eqn-XYA-square}
\xymatrix{
E \ar[r]^{q} \ar[d]_f & G \ar[d]^{\id} \\
F \ar[r]^{r} & G
}
\end{align}
commutes in $H^0(\C)$. Since $H^0(\C)$ is triangulated, 
there exists an isomorphism $\cone(q) \rightarrow
\cone(r)$ which extends this square in $H^0(\C)$ to 
an isomorphism of exact triangles. 
It follows that we can extend $E \xrightarrow{f} F$ and 
$G \xrightarrow{\id} G$ to a homotopy equivalence in $\pretriag(\C)$
of the twisted complexes $E \xrightarrow{q} G$ 
and $F\xrightarrow{r} G$. 

If we actually try and construct this homotopy equivalence, 
we run into the following type of problems:

\smallskip

\em Claim: \rm Let $g \in \homm^0_\C(F,E)$ be a homotopy inverse of $f$. 
In other words, there exist $h \in \homm^{-1}_\C(E,E)$  and 
$h' \in \homm^{-1}_\C(F,F)$ such that $gf - \id = dh$ 
and $fg - \id = dh'$. 

Then there exist mutually inverse homotopy equivalences
\begin{align}
\label{eqn-hom_eq_two_terms}
\vcenter{
\xymatrix{
E \ar[rr]^q \ar[dd]_f \ar[rrdd]^<<<<<<<<<<<<{t}
& & 
G \ar[dd]^{\id} \\
\\
F \ar[rr]^{r} & & G
}
}
\quad \quad
\vcenter{
\xymatrix{
F \ar[rr]^{r}\ar[rrdd]^<<<<<<<<<<<<{t'} \ar[dd]_g 
& & G \ar[dd]^{\id} \\
\\
E \ar[rr]^{q} & & G
}
}
\quad \quad
\begin{minipage}[c][1in][c]{1.5in}
$t \in \homm^{-1}_\C(E,G),$\\
$t' \in \homm^{-1}_\C(F,G)$
\end{minipage}
\end{align}
of $E\xrightarrow{q} G$ and $F\xrightarrow{r} G$ 
if and only if $h$ and $h'$ can be chosen so that 
the following equivalent conditions hold: 
\begin{enumerate}
\item $r (fh - h' f) = ds$ for some 
$s \in \homm^{-2}_\C(E,G)$.
\item $q (gh' - h g) = ds'$ for some 
$s' \in \homm^{-2}_\C(F,G)$.
\end{enumerate}

\em Proof: \rm Straightforward verification. 

\smallskip

Apriori, there is no reason to expect a class 
like $r (fh - h' f)$ in $\homm^{-1}_\C(E,G)$ 
to be null-homotopic. In fact, for general $h$ and $h'$
it wouldn't be. So this may seem like a genuinely necessary 
condition. 

However, it turns out that we can always choose $h$ and $h'$ so 
that even $fh - h' f$ and $gh' - h g$ are
null-homotopic. Since $d q = d r = 0$, it would also 
imply the conditions above.

The explanation is: $fh - h' f$ and $gh' - h g$ 
are both killed by the differential, and thus define classes
$\xi \in \homm^{-1}_{H^0(\C)}(E,F)$ and 
$\xi \in \homm^{-1}_{H^0(\C)}(F,E)$, respectively. 
Since $f$ and $g$ are isomorphisms in $H^0(\C)$, they identify 
both $\homm^{-1}_{H^0(\C)}(E,F)$ and $\homm^{-1}_{H^0(\C)}(F,E)$ 
with $\homm^{-1}_{H^0(\C)}(E,E)$. Apriori, neither $\xi$, nor $\xi'$
are zero, however one can check that $\xi$ and $-\xi'$ give
the same class in $\homm^{-1}_{H^0(\C)}(E,E)$. We can therefore correct 
$h \in \homm^{-1}_{\C}(E,E)$ by this class and kill off both 
$\xi$ and $\xi'$, as required. 

It is not a calculation one would want to try and write down in a
larger, more complicated scenario. Fortunately, there turns out to be
a more conceptual argument. It requires us to consider 
$A_\infty$-categories and $A_\infty$-functors, 
see \cite{Keller-IntroductionToAInfinityAlgebrasAndModules} and
\cite[\S 8]{Lefevre-SurLesAInftyCategories} for the basics. 
In particular, we use the convention in 
\cite[\S 8]{Lefevre-SurLesAInftyCategories} for denoting 
$\A_\infty$-functors as $\left(\dot{\mathfrak{f}}, \mathfrak{f}_i\right)$ 
where $\dot{\mathfrak{f}}$ is the object map, $\mathfrak{f}_1$ is the
morphism map and $\mathfrak{f}_{i \geq 2}$ are the higher morphism maps. 

A choice of $h$ and $h'$ as
above and also of $j \in \homm^{-2}_{\C}(X,Y)$ and $j' \in
\homm^{-2}_{\C}(Y,X)$ such that 
$fh - h' f = dj$ and $ fh - h' f = dj'$
can readily be checked to be a part of precisely the data necessary 
to define a strictly unital $A_\infty$-functor
\begin{align*}
\begin{minipage}[c][1in][c]{1in}
$\psi \phi = \id_x,$ \\
$\phi \psi = \id_y,$ \\
$\beta \phi = \alpha$, \\
$\alpha \psi = \beta$, 
\end{minipage}
\vcenter{
\xymatrix{
\overset{x}{\bullet}
\ar[rd]^{\alpha}
\ar@/_/[dd]_{\phi}
& & & & & & 
\\
 &  \overset{a}{\bullet} &
\ar[rrr]^{\left(\dot{\mathfrak{f}}, \mathfrak{f}_i\right)} 
& & & & 
\C
\\
\overset{y}{\bullet}
\ar[ru]_{\beta} 
\ar@/_/[uu]_{\psi}
& & & & & & 
}
}
\end{align*}
which sends $x,y,a$ to $E,F,G$ and $\phi, \psi, \alpha, \beta$
to $f,g,q,r$. Here, the quiver on the left defines 
an additive $k$-category whose objects are the vertices of the quiver
and whose $\homm$-spaces are generated by the paths in the quiver,
modulo the indicated relations. The trivial path from a vertex to
itself correspond to its identity morphism. Denote this category 
by $\bar{\B}_1$, we think of it as of a DG-category concentrated 
in degree zero.
 
Conversely, any $A_\infty$-functor 
$\bar{\B}_1 \xrightarrow{\left(\dot{\mathfrak{f}},
\mathfrak{f}_i\right)} \C$ as above contains the data of 
homotopy equivalences \eqref{eqn-hom_eq_two_terms}. 
This is because $\left(\dot{\mathfrak{f}}, \mathfrak{f}_i\right)$ 
extends naturally to an $A_\infty$-functor 
$\pretriag(\bar{\B}_1)\xrightarrow{\left(\dot{\mathfrak{f}}, 
\mathfrak{f}_i\right)} \pretriag(\C)$. 
In $\pretriag(\bar{\B}_1)$ the twisted complexes 
$x \xrightarrow{\alpha} a$ and $y \xrightarrow{\beta} a$
are isomorphic. Specifically,
\begin{align*}
\vcenter{
\xymatrix{
x \ar[r]^{\alpha} \ar[d]_{\phi} 
& a \ar[d]^{\id} \\
y \ar[r]^{\beta} & a
}
}
\quad \quad
\vcenter{
\xymatrix{
y \ar[r]^{\beta} \ar[d]_{\psi}
&
a \ar[d]^{\id} \\
x \ar[r]^{\alpha} & a
}
}
\end{align*}
are mutually inverse isomorphisms. Their images under $\mathfrak{f}_1$
are the morphisms 
\begin{align*}
\vcenter{
\xymatrix{
E \ar[rr]^{q} \ar[dd]_{f} 
\ar[ddrr]^<<<<<<<<{\mathfrak{f}_2(\beta, \phi)}
& & G \ar[dd]^{\id} \\
& & 
\\
F \ar[rr]^{r} & & G
}
}
\quad \quad
\vcenter{
\xymatrix{
F \ar[rr]^{r} \ar[dd]_{g}
\ar[ddrr]^<<<<<<<<{\mathfrak{f}_2(\alpha, \psi)}
& &
G \ar[dd]^{\id} \\
& & 
\\
E \ar[rr]^{q} & & G
}
}
\end{align*}
in $\pretriag(\C)$. Since 
$\left(\dot{\mathfrak{f}}, H^0(\mathfrak{f}_1)\right)$ 
is an exact functor, these are become mutually inverse isomorphisms in 
$H^0(\pretriag(C))$. Thus, they are the mutually inverse 
homotopy equivalences \eqref{eqn-hom_eq_two_terms} we want.  

To construct a strictly unital $A_\infty$-functor 
$\bar{\B}_1 \xrightarrow{\left(\dot{\mathfrak{f}}, \mathfrak{f}_i\right)} \C$ 
it suffices to construct a strictly unital $A_\infty$-functor 
$\B_1 \xrightarrow{\left(\dot{\mathfrak{g}}, \mathfrak{g}_i\right)} \C$ 
where $\B_1$ is the category
\begin{align*}
\vcenter{
\xymatrix{
\overset{x}{\bullet}
\ar[rd]^{\alpha}
\ar[dd]_{\phi}
\\
&  \overset{a}{\bullet} 
\\
\overset{y}{\bullet}
\ar[ru]_{\beta} 
}
}
\quad\quad
\begin{minipage}[c][1in][c]{1in}
$\beta \phi = \alpha$. \\
\end{minipage}
\end{align*}
Roughly, this is because $\bar{\B}_1$ is the minimal $A_\infty$-structure 
of a certain DG-quotient of $\B_1$ whose universal properties
ensure that  
$\B_1 \xrightarrow{\left(\dot{\mathfrak{g}}, \mathfrak{g}_i\right)} \C$
filters through some 
$\bar{\B}_1 \xrightarrow{\left(\dot{\mathfrak{f}},
\mathfrak{f}_i\right)} \C$. We'll give the full argument in a greater 
generality later on in this section.

Thus we are reduced to constructing a strictly unital $A_\infty$-functor
$\B_1 \xrightarrow{\left(\dot{\mathfrak{g}}, \mathfrak{g}_i\right)} \C$ 
which sends $x,y,a$ to $E,F,G$ and $\phi, \alpha, \beta$
to $f, q, r$. The data of such functor is simply 
the choice of $\mathfrak{f}_2(\beta, \phi) \in
\homm_\C^{-1}(E,G)$ such that
$$ q - r f = \mathfrak{f}_2(\beta, \phi). $$ 
The existence of such class in $\homm_\C^{-1}(E,G)$ is precisely
the condition that \eqref{eqn-XYA-square} commutes in $H^0(\C)$. 

To sum up, a sufficient condition for the homotopy
equivalence $E \xrightarrow{f} F$ to
induce a homotopy equivalence 
\begin{align}
\label{eqn-X->A-to-Y->A-homotopy-equivalence}
\left\{E \xrightarrow{q} G\right\}
\rightarrow \left\{F \xrightarrow{r} G\right\}
\end{align}
is that $f$ must 
commute with $q$ and $r$ in $H^0(\C)$. 
This is also precisely the condition that a strictly unital
$A_\infty$-functor $\B_1 \rightarrow \C$ exists which
sends $x,y,a$ to $E,F,G$ and $\phi, \alpha, \beta$
to $f, q, r$. All the other conditions which seemingly 
arise when one naively tries to construct the homotopy equivalence 
\eqref{eqn-X->A-to-Y->A-homotopy-equivalence} are part of the data
necessary to lift this functor to a functor $\bar{\B}_1 \rightarrow \C$. 
Which gets done for us automatically by the universal properties of 
DG-quotients. 
\end{exmpl}

\smallskip 

The method outlined in Example \ref{exmpl-homotopy-equivalences-b1-example}
can be applied in full generality to any pair of twisted complexes
$(E_i, q_{ij})$, $(F_i, r_{ij})$ and any set of homotopy equivalences
$E_i \xrightarrow{f_i} F_i$ to answer the question posed in the
beginning of this subsection. In such a generality, however, 
the answer would not only look fearsome, but also quite 
obfuscating. 

Below, we only argue it in the generality we need for the proofs in 
Section \ref{section-braiding-criteria}. 

\begin{defn}
Denote by $\bar{\B}_n$ the category defined by  
\begin{align}
\label{eqn-quiver-defining-bar-B_n}
\begin{minipage}[c][1in][c]{1.5in}
$\psi \phi = \id_x,$ \\
$\phi \psi = \id_y,$ \\
$\beta \phi  = \alpha$, \\
$\alpha \psi  = \beta$, \\
$\gamma_1 \alpha  = 
\gamma_1 \beta = 0$, \\
$\gamma_{i+1} \gamma_{i} = 0$
\end{minipage}
\vcenter{
\xymatrix{
\overset{x}{\bullet}
\ar[rd]^{\alpha}
\ar@/_/[dd]_{\phi}
\\
& \overset{a_1}{\bullet} \ar[r]_{\gamma_1} 
& \overset{a_2}{\bullet} \ar[r]_{\gamma_2} 
& \overset{a_3}{\bullet} \ar[r]_{\gamma_3} 
& \dots \ar[r]_{\gamma_{n-2}}
& \overset{a_{n-1}}{\bullet} \ar[r]_{\gamma_{n-1}} 
& \overset{a_{n}}{\bullet} 
\\
\overset{y}{\bullet}
\ar[ru]_{\beta} 
\ar@/_/[uu]_{\psi}
}
}.
\end{align}
We consider it as a DG-category concentrated in degree $0$. 
Denote by $\B_n$ its subcategory defined by the same quiver
but with the arrow $\psi$ removed. 
\end{defn}

DG quotients were introduced by Drinfeld in 
\cite{Drinfeld-DGQuotientsOfDGCategories} where we refer the reader 
to for all the details. 

\begin{lemma}
\label{lemma-explicit-presentation-of-Bn-quot-cone-phi}
Let $\B^f_n$ be the full subcategory of the DG-quotient
$\pretriag(\B_n)/\cone(\phi)$ supported at the objects of $\B_n$. 
Then $\B^f_n$ is isomorphic to the DG category defined by   
\begin{align}
\label{eqn-quiver-presentation-of-Bn-quot-cone-phi}
\begin{minipage}[c][1in][c]{1.5in}
$\beta \phi  = \alpha$, \\
$\alpha \psi = \beta$, \\
$\gamma_1 \alpha  = 
\gamma_1 \beta = 0$, \\
$\gamma_{i+1} \gamma_{i} = 0$ \\
$d \theta_x = - \id_x + \psi \phi,$ \\
$d \theta_y = \id_y - \phi \psi,$ \\
$d \psi = 0,$ \\
$d \xi = -\phi \theta_x - \theta_y \phi.$ 
\end{minipage}
\vcenter{
\xymatrix{
\overset{x}{\bullet}
\ar@{.>}@(ul,ur)[]^{\theta_x}
\ar[rd]^{\alpha}
\ar@/_/[dd]_{\phi}
\ar@{-->}@/_2em/[dd]_{\xi}
\\
& \overset{a_1}{\bullet} \ar[r]_{\gamma_1} 
& \overset{a_2}{\bullet} \ar[r]_{\gamma_2} 
& \overset{a_3}{\bullet} \ar[r]_{\gamma_3} 
& \dots \ar[r]_{\gamma_{n-2}}
& \overset{a_{n-1}}{\bullet} \ar[r]_{\gamma_{n-1}} 
& \overset{a_{n}}{\bullet} 
\\
\overset{y}{\bullet}
\ar@{.>}@(dr,dl)[]^{\theta_y}
\ar[ru]_{\beta} 
\ar@/_/[uu]_{\psi}
}
}
\end{align}
where dotted arrows denote the morphisms of degree $-1$ and the dashed arrow
the morphism of degree $-2$.
\end{lemma}
\begin{proof}
In $\pretriag(\B_n)$ the cone of $\phi$ is the twisted complex
$x \xrightarrow{\phi} y$. As explained in  
\cite[\S3.1]{Drinfeld-DGQuotientsOfDGCategories} the DG quotient 
of $\pretriag(\B_n)$ by $x \xrightarrow{\phi} y$ is constructed by
adding a single endomorphism $\epsilon$ of $x \xrightarrow{\phi} y$
of degree $-1$ with $d\epsilon = \id$ and no other relations. 

As $\B_n$ is a subcategory of
\eqref{eqn-quiver-presentation-of-Bn-quot-cone-phi}, 
every twisted complex over $\B_n$ is a twisted complex over 
$\eqref{eqn-quiver-presentation-of-Bn-quot-cone-phi}$. 
Let $\A$ be the full subcategory of 
$\pretriag(\eqref{eqn-quiver-presentation-of-Bn-quot-cone-phi})$
consisting of all the objects in $\pretriag(\B_n)$. 
Define a functor from $\pretriag(\B_n)/(x \xrightarrow{\phi} y)$ to
$\A$ by sending $\epsilon$ to 
\begin{align*}
\vcenter{
\xymatrix{
x \ar[rr]^{\phi}
\ar[dd]_{\theta_x}
\ar[ddrr]^<<<<<<{\xi}
& &
y
\ar[dd]^{\theta_y}
\ar[ddll]_<<<<<<{\psi}
\\
\\
x \ar[rr]^{\phi}
& &
y
}
}.
\end{align*}
Define a functor in the opposite direction by sending
$\theta_x$, $\theta_y$, $\psi$ and $\xi$ to the compositions
\begin{align*}
\vcenter{
\xymatrix{
& & x \ar[dll]_{\id} 
\\
x \ar[rr]^{\phi} & \ar[d]^{\epsilon} & y 
\\
x \ar[drr]_{\id} \ar[rr]_{\phi} & & y 
\\
& & x
}
} 
\quad
\vcenter{
\xymatrix{
& & y \ar[d]_{\id} 
\\
x \ar[rr]^{\phi} & \ar[d]^{\epsilon} & y 
\\
x \ar[rr]_{\phi} & & y \ar[d]_{\id}
\\
& & y 
}
} 
\quad
\vcenter{
\xymatrix{
& & y \ar[d]_{\id} 
\\
x \ar[rr]^{\phi} & \ar[d]^{\epsilon} & y 
\\
x \ar[drr]_{\id} \ar[rr]_{\phi} & & y 
\\
& & x
}
} 
\quad
\vcenter{
\xymatrix{
& & x \ar[dll]_{\id} 
\\
x \ar[rr]^{\phi} & \ar[d]^{\epsilon} & y 
\\
x \ar[rr]_{\phi} & & y \ar[d]_{\id} 
\\
& & y
}
} 
\end{align*}
in $\pretriag(\B_n)/(x \xrightarrow{\phi} y)$, respectively. One
can readily check that these functors are mutualy inverse. 
Hence $\pretriag(\B_n)/(x \xrightarrow{\phi} y)$ is isomorphic to 
$\A$, and the result follows.
\end{proof}

Recall that an $A_\infty$-category is called \em minimal \rm if
it has $m_1 = 0$. Let $\A$ be an $A_\infty$-category. 
The \em minimal model \rm of $\A$ is a minimal $A_\infty$-category
$\A'$ together with an $\A_\infty$-quasi-isomorphism $\A' \rightarrow \A$. 
Such model always exists and is unique up to an $\A_\infty$-isomorphism, 
see \cite[\S1.4.1]{Lefevre-SurLesAInftyCategories} and
\cite[S6.4]{KontsevichSoibelman-HomologicalMirrorSymmetryAndTorusFibrations}. 

\begin{lemma}
\label{lemma-barBn-is-the-minimal-model-of-B_n-quot-cone-phi}
There exists a strictly unital $\A_\infty$-quasi-isomorphism 
\begin{align}
\label{eqn-b_n-bar-to-the-dg-quotient-of-b_n-functor}
\bar{\B}_n \xrightarrow{(\dot{\mathfrak{g}}, \mathfrak{g}_i)} \B_n^f, 
\end{align}
which gives $\bar{B}_n$ the structure of the minimal model of $\B_n^f$. 
\end{lemma}
\begin{proof}
Recall that $\bar{\B}_n$ is an ordinary category considered 
as an $A_\infty$-category concentrated in degree $0$. In particular, 
$\bar{\B}_n$ can be identified with its own graded homotopy category 
$H^\bullet(\bar{\B}_n)$. 

The category $\bar{\B}_n$ is defined by the quiver 
\eqref{eqn-quiver-defining-bar-B_n}, while 
Lemma \ref{lemma-explicit-presentation-of-Bn-quot-cone-phi} identifies 
$\B_n^f$ with the category defined by the DG-quiver  
\eqref{eqn-quiver-presentation-of-Bn-quot-cone-phi}. 
Forgetting the relations, identifying vertices and arrows which 
have the same labels gives the quiver \eqref{eqn-quiver-defining-bar-B_n} 
the structure of a subquiver of 
\eqref{eqn-quiver-presentation-of-Bn-quot-cone-phi}. This structure
defines a map $\dot{\mathfrak{g}}$ from the set of objects of $\bar{\B}_n$ 
to the set of objects of $\B_n^f$ and a map $\mathfrak{g}_1$ of morphism
spaces of $\bar{\B}_n$ into the morphisms spaces of $\B_n^f$. These
are compatible with differentials, but are not compatible with
compositions. 

By inspection, $(\dot{\mathfrak{g}}, \mathfrak{g}_1)$ does 
define an isomorphism 
$$ \bar{\B}_n \xrightarrow{\sim} H^\bullet(\B^f_n) $$
of graded homotopy categories. We can therefore apply the procedure 
described in
\cite[\S6.4]{KontsevichSoibelman-HomologicalMirrorSymmetryAndTorusFibrations}. 
It can be readily checked that it constructs $\mathfrak{g}_{\geq 2}$
which extend $\dot{\mathfrak{g}}$ and $\mathfrak{g}_1$ 
to a strictly unital $A_\infty$-quasi-isomorphism $\bar{\B}_n
\xrightarrow{(\dot{\mathfrak{g}}, \mathfrak{g}_i)} \B_n^f$, as required. 
\end{proof}
 
Before we proceed, we need to state the following well-known fact: 
\begin{lemma}
Let $\A$ be a DG-category, let $m \leq n$ be two integers and 
let $A_m, \dots, A_n$ be objects of $\A$. 
The one-sided twisted complexes 
\begin{align*}
(E_i, q_{ij}) \in \pretriag(\A) 
\quad\quad \text{ with } \quad \quad
\begin{cases}
E_i = A_i \quad \text{ for } m \leq i \leq n \\
E_i = 0 \quad \text{ otherwise }
\end{cases}
\end{align*}
are in $1$-to-$1$ correspondence with the strictly unital $A_\infty$-functors 
\begin{align*}
\gamma_i \gamma_{i + 1} = 0 \quad \quad 
\vcenter{
\xymatrix{
\overset{a_m}{\bullet} \ar[r]_{\gamma_m} 
& \overset{a_{m+1}}{\bullet} \ar[r]_{\gamma_{m+1}} 
& \dots \ar[r]_{\gamma_{n-2}}
& \overset{a_{n-1}}{\bullet} \ar[r]_{\gamma_{n-1}} 
& \overset{a_{n}}{\bullet} 
& 
\ar[rr]^{\left(\dot{\mathfrak{f}}, \mathfrak{f}_i\right)} 
& & &  
\C
}
}
\end{align*}
with $\dot{\mathfrak{f}}(a_i) = A_i$. 
\end{lemma}
\begin{proof}
Mutually inverse maps between the two sets can be defined by setting 
$$\mathfrak{f}_k(\gamma_{i+k-1}, \gamma_{i+k-2}, \dots, \gamma_{i})  
= (-1)^{i-1} q_{i(i+k)}
\quad \quad
\forall\; i \in \left\{m, \dots, n\right\} 
\; \text{ and } \; 
k \in \left\{ 1, \dots, n-i \right\} $$
and vice versa. 
\end{proof}

Let $\C$ be a strongly pretriangulated category and let 
$(A_i, g_{ij})$ be a one-sided twisted complex over $\C$
concentrated in degrees $1, \dots, n$. Let 
$(E_i, q_{ij})$ and $(F_i, r_{ij})$ be one-sided twisted complexes 
over $\C$ concentrated in degrees $0, \dots, n$ whose 
twisted subcomplexes supported in degrees $1, \dots, n$ 
are both equal to $(A_i, g_{ij})$. 

Let $A$ denote the convolution of $(A_i, g_{ij})$. 
Consider the closed degree $1$ morphisms $(q_{0j})$ and $(r_{0j})$ 
from $E_0$ and $F_0$ to $(A_i, g_{ij})$ in $\pretriag(\C)$. 
Denote by $E_0 \xrightarrow{q_0} A$ and $F_0 \xrightarrow{r_0} A$ 
the corresponding morphisms in $\C$. 

Recall that $\B_n$ is the category defined by
\begin{align}
\label{eqn-quiver-presentation-of-Bn}
\begin{minipage}[c][1in][c]{1.5in}
$\beta \phi = \alpha$, \\
$\gamma_1 \alpha  = 
\gamma_1 \beta = 0$, \\
$\gamma_{i+1} \gamma_{i} = 0$
\end{minipage}
\vcenter{
\xymatrix{
\overset{x}{\bullet}
\ar[rd]^{\alpha}
\ar@/_/[dd]_{\phi}
\\
& \overset{a_1}{\bullet} \ar[r]_{\gamma_1} 
& \overset{a_2}{\bullet} \ar[r]_{\gamma_2} 
& \overset{a_3}{\bullet} \ar[r]_{\gamma_3} 
& \dots \ar[r]_{\gamma_{n-2}}
& \overset{a_{n-1}}{\bullet} \ar[r]_{\gamma_{n-1}} 
& \overset{a_{n}}{\bullet} 
\\
\overset{y}{\bullet}
\ar[ru]_{\beta} 
}
}.
\end{align}

\begin{prps}
\label{prps-existence-of-an-A-infty-functor-B_n-->--C}
There exists a strictly unital $A_\infty$-functor 
$$\B_n \xrightarrow{\left(\dot{\mathfrak{f}}, \mathfrak{f}_i\right)} \C$$
whose restrictions to the full subcategories
of $\B_n$ supported at $x, a_1, \dots, a_n$ and $y, a_1, \dots, a_n$
correspond to the twisted complexes $(E_i, q_{ij})$ and 
$(F_i, r_{ij})$ if and only if the following two equivalent 
conditions hold:
\begin{enumerate}
\item 
\label{item-f-and-higher-morphisms-s_k-exist}
There exist $f \in \homm_\C^0(E_0,F_0)$ and 
$s_i \in \homm_\C^{-k}(E_0, A_k)$ for $k \in
\left\{1,\dots,n\right\}$ such that 
\begin{align}
\label{eqn-relations-on-f-and-s_k}
q_{0k} - r_{0k}f = 
\sum_{1 \leq j \leq k-1} q_{jk} s_j + (-1)^{k} ds_k.
\end{align}
\item 
\label{item-f-exists-and-commutes-with-q_0-and-r_0}
There exists $f \in \homm_{H^0(\C)}(E_0,F_0)$ such that 
\begin{align*}
\vcenter{
\xymatrix{
E_0 \ar[r]^{q_0} \ar[d]_{f} & A[1] \\
F_0 \ar[ur]_{r_0}
}
}
\end{align*}
commutes in $H^0(\C)$. 
\end{enumerate}
\end{prps}
\begin{proof}
\em The existence of $\left(\dot{\mathfrak{f}}, \mathfrak{f}_i\right)
\Leftrightarrow 
\eqref{item-f-and-higher-morphisms-s_k-exist}$: \rm

The condition that  $\left(\dot{\mathfrak{f}}, \mathfrak{f}_i\right)$
restricts on $x, a_1, \dots, a_n$ and $y, a_1, \dots, a_n$ to 
the functors corresponding to $(E_i, q_{ij})$ and $(F_i, r_{ij})$ 
determines $\dot{\mathfrak{f}}$ and all the values of $\mathfrak{f}_i$ 
other than
\begin{align}
\label{eqn-higher-morphism-maps-f-involving-phi}
\mathfrak{f}_1(\phi),\; 
\mathfrak{f}_2(\beta,\phi), \;
\mathfrak{f}_3(\gamma_1, \beta, \phi),\; 
\dots,\;
\mathfrak{f}_{n+1}(\gamma_{n-1}, \dots, \gamma_1, \beta, \phi). 
\end{align}
One can readily verify that if we set these to 
$f$, $s_1$, \dots, $s_n$, then
the standard relations which
$\eqref{eqn-higher-morphism-maps-f-involving-phi}$  
must satisfy according to the definition of an $A_\infty$-functor
\cite[\S3.4]{Keller-IntroductionToAInfinityAlgebrasAndModules}
become precisely the equations \eqref{eqn-relations-on-f-and-s_k}. 
And vice versa. 

\em $\eqref{item-f-and-higher-morphisms-s_k-exist}
\Leftrightarrow 
\eqref{item-f-exists-and-commutes-with-q_0-and-r_0}$: \rm

Let $s_k \in \homm_\C^{-k}(E_0, A_k)$ for $k \in
\left\{1,\dots,n\right\}$. Consider 
the degree $0$ morphism $E_0 \xrightarrow{(s_k)} (A_i,g_{ij})$ in 
$\pretriag(C)$. It is a straightforward verification that $d(s_k)$ is 
the morphism $E_0 \rightarrow (A_i,g_{ij})$ whose
component in $\homm_\C^{-k+1}(E_0,A_k)$ is precisely the RHS
of $\eqref{eqn-higher-morphism-maps-f-involving-phi}$. 

On the other hand, for any $f \in \homm_\C^0(E_0,F_0)$ 
the LHS of $\eqref{eqn-higher-morphism-maps-f-involving-phi}$
is the component in $\homm_\C^{-k+1}(E_0,A_k)$ of
the morphism $E_0 \xrightarrow{(q_{0j}) - (r_{0j})f} (A_i,g_{ij})$ 
in $\pretriag(\C)$. 

We conclude that \eqref{item-f-and-higher-morphisms-s_k-exist} is equivalent
to existence of $f \in \homm_\C^0(E_0,F_0)$ and $s \in
\homm_\C^0(E_0, A)$ such that $q_0 - r_0 f = ds$. This is precisely
the claim of \eqref{item-f-exists-and-commutes-with-q_0-and-r_0}. 
\end{proof}

The following is the main result of this section:

\begin{theorem}
\label{theorem-homotopy-equiv-of-twisted-complexes-differ-only-first-term}
Let $E_0 \xrightarrow{f} F_0$ be a homotopy equivalence
satisfying the equivalent conditions of 
Prop. \ref{prps-existence-of-an-A-infty-functor-B_n-->--C}. 
Then there exists a homotopy equivalence 
$$ (E_i, q_{ij}) \xrightarrow{(f_{ij})} (F_i, r_{ij}) $$
in $\pretriag(\C)$.  
\end{theorem}
\begin{proof}
By Prop. \ref{prps-existence-of-an-A-infty-functor-B_n-->--C}
there exists a strictly unital $A_\infty$-functor
$$\B_n \xrightarrow{\left(\dot{\mathfrak{f}}, \mathfrak{f}_i\right)}
\C$$
with $\mathfrak{f}_1(\phi) = f$. 
It extends naturally to a strictly unital $A_\infty$-functor
$$
\pretriag(\B_n) 
\xrightarrow{\left(\dot{\mathfrak{f}}, \mathfrak{f}_i\right)}
\pretriag(\C).
$$

By \cite[\S4.3]{Keller-OnDifferentialGradedCategories} there exists 
a corresponding quasi-functor 
$$ \pretriag(\B_n) \xrightarrow{\Phi}  \pretriag(\C)$$
in $\HoDGCat$ with $H\Phi \simeq H\mathfrak{f}$ as functors 
$H^0(\pretriag(\B_n)) \rightarrow H^0(\pretriag(\C))$. 
Since $H\mathfrak{f}(\phi) = f$ and since $f$ is an isomorphism in $H^0(\C)$, 
it follows that $H\Phi(\cone(\phi)) = 0$. By the universal 
property of DG-quotients 
\cite[Theorem 1.6.2]{Drinfeld-DGQuotientsOfDGCategories} 
quasi-functor $\Phi$ lifts to a quasi-functor 
$$\pretriag(\B_n) / \cone(\phi) \xrightarrow{\Phi'} \pretriag(\C)$$
such that $\Phi = \Phi' Q$ where $Q$ is 
the quotient quasi-functor 
$\pretriag(\B_n) \rightarrow \pretriag(\B_n) / \cone(\phi)$. Denote
by 
$$\pretriag(\B_n) / \cone(\phi)
\xrightarrow{\left(\dot{\mathfrak{f}}', \mathfrak{f}'_i\right)}
\pretriag(\C)$$
the corresponding strictly unital $A_\infty$-functor. We have
$\left(\dot{\mathfrak{f}}, \mathfrak{f}_i\right) = 
\left(\dot{\mathfrak{f}}', \mathfrak{f}'_i\right)Q$ and hence
restricting
to the full subcategory $\B_n^f$ of $\pretriag(\B_n) / \cone(\phi)$
consisting of objects of $\B_n$ we obtain a strictly unital 
$A_\infty$-functor
$$\B_n^f 
\xrightarrow{\left(\dot{\mathfrak{f}}', \mathfrak{f}'_i\right)}
\C.$$

Recall that in
Lemma \ref{lemma-barBn-is-the-minimal-model-of-B_n-quot-cone-phi}
we have constructed a strictly unital $A_\infty$-quasi-isomorphism
$\bar{\B}_n \xrightarrow{\left(\dot{\mathfrak{g}}, \mathfrak{g}_i\right)}
\B_n^f$ which gives $\bar{\B}_n$ the structure of the minimal model 
of $\B_n^f$. Taking the composition of
$\bar{\B}_n \xrightarrow{\left(\dot{\mathfrak{g}}, \mathfrak{g}_i\right)}
\B_n^f  \xrightarrow{\left(\dot{\mathfrak{f}}', \mathfrak{f}'_i\right)} \C$
we obtain the strictly unital $A_\infty$-functor denoted 
$$ 
\bar{\B}_n \xrightarrow{\left(\dot{\mathfrak{h}}, \mathfrak{h}_i\right)} \C. 
$$

We claim that 
$\bar{\B}_n \xrightarrow{\left(\dot{\mathfrak{h}}, \mathfrak{h}_i\right)} \C$
restricts on the full subcategory $\B_n \hookrightarrow
\bar{\B}_n$ to
$\B_n \xrightarrow{ \left(\dot{\mathfrak{f}}, \mathfrak{f}_i\right) }
\C$. As $\left(\dot{\mathfrak{f}}, \mathfrak{f}_i\right) = 
\left(\dot{\mathfrak{f}}', \mathfrak{f}'_i\right)Q$ this reduces
to the following diagram being commutative
\begin{align}
\vcenter{
\xymatrix{
\B_n 
\ar@{^{ (}->}[d] 
\ar[rd]^{Q}
\\
\bar{\B}_n 
\ar[r]_{\left(\dot{\mathfrak{g}}, \mathfrak{g}_i\right)}
&
\B_n^f
} 
}.
\end{align}
This is a straightforward check. On one hand, in Lemma 
\ref{lemma-explicit-presentation-of-Bn-quot-cone-phi}
we have constructed an explicit isomorphism between $\B_n^f$ and 
the category defined by
\eqref{eqn-quiver-presentation-of-Bn-quot-cone-phi}. 
One can check that it identifies the DG-quotient functor $\B_n
\xrightarrow{Q} \B_n^f$ 
with the functor induced by the inclusion of 
\eqref{eqn-quiver-presentation-of-Bn} into 
\eqref{eqn-quiver-presentation-of-Bn-quot-cone-phi} as 
quivers with relations. On the other hand, in Lemma 
\ref{lemma-barBn-is-the-minimal-model-of-B_n-quot-cone-phi}
we have used the above isomorphism between $\B_n^f$ and 
\eqref{eqn-quiver-presentation-of-Bn-quot-cone-phi}
to define $\dot{\mathfrak{g}}$ and $\mathfrak{g}_1$ 
by the quiver inclusion of 
\eqref{eqn-quiver-defining-bar-B_n}
into \eqref{eqn-quiver-presentation-of-Bn-quot-cone-phi} 
which ignores relations. However, 
restricted from \eqref{eqn-quiver-defining-bar-B_n} to
\eqref{eqn-quiver-presentation-of-Bn} 
this inclusion does respect the relations. 
Therefore $(\dot{\mathfrak{g}}, \mathfrak{g}_1)$ restricted
to $\B_n$ is a genuine functor. One can check that this forces 
$\mathfrak{g}_{\geq 3}$ constructed by
the procedure  
in \cite[\S6.4]{KontsevichSoibelman-HomologicalMirrorSymmetryAndTorusFibrations}
to be zero when restricted to $\B_n$. 
Thus $(\dot{\mathfrak{g}}, \mathfrak{g}_i)$ restricted
to $\B_n$ is just the functor $(\dot{\mathfrak{g}}, \mathfrak{g}_1)$, 
i.e. the functor defined by the inclusion of 
\eqref{eqn-quiver-presentation-of-Bn} into 
\eqref{eqn-quiver-presentation-of-Bn-quot-cone-phi}. The claim
follows. 

In $\pretriag(\bar{\B}_n)$ the 
twisted complexes $x \xrightarrow{\alpha} a_1  \dots
\xrightarrow{\gamma_{n-1}} a_n$ and 
$y \xrightarrow{\beta} a_1  \dots \xrightarrow{\gamma_{n-1}} a_n$
are isomorphic, for instance the following 
\begin{align}
\vcenter{
\xymatrix{
\overset{x}{\bullet} \ar[r]^{\alpha} \ar[d]_{\phi}
& \overset{a_1}{\bullet} \ar[r]_{\gamma_1} \ar[d]^{\id}
& \overset{a_2}{\bullet} \ar[r]_{\gamma_2} \ar[d]^{\id} 
& \overset{a_3}{\bullet} \ar[r]_{\gamma_3} \ar[d]^{\id}
& \dots \ar[r]_{\gamma_{n-2}}
& \overset{a_{n-1}}{\bullet} \ar[r]_{\gamma_{n-1}} \ar[d]^{\id}
& \overset{a_{n}}{\bullet} \ar[d]^{\id}
\\
\overset{y}{\bullet} \ar[r]^{\beta}
& \overset{a_1}{\bullet} \ar[r]_{\gamma_1} 
& \overset{a_2}{\bullet} \ar[r]_{\gamma_2} 
& \overset{a_3}{\bullet} \ar[r]_{\gamma_3} 
& \dots \ar[r]_{\gamma_{n-2}}
& \overset{a_{n-1}}{\bullet} \ar[r]_{\gamma_{n-1}} 
& \overset{a_{n}}{\bullet} 
}
}
.
\end{align}
is an isomorphism of twisted complexes. 
Hence they are also isomorphic in $H^0(\pretriag(\bar{\B}_n))$, 
and hence their images under $\left(\dot{\mathfrak{h}},
H^0(\mathfrak{h}_1)\right)$ 
are isomorphic in $H^0(\pretriag(\C))$. But by the claim above 
$\left(\dot{\mathfrak{h}}, \mathfrak{h}_i\right)$ and 
$\left(\dot{\mathfrak{f}}, \mathfrak{f}_i\right)$ agree 
on the subcategory $\pretriag(\B_n)$ of $\pretriag(\bar{\B}_n)$. 
Hence $\left(\dot{\mathfrak{h}},\mathfrak{h}_i\right)$
takes $x \xrightarrow{\alpha} a_1  \dots
\xrightarrow{\gamma_{n-1}} a_n$ and 
$y \xrightarrow{\beta} a_1  \dots \xrightarrow{\gamma_{n-1}} a_n$
to $(E_i, q_{ij})$ and $(F_i, r_{ij})$. The claim of the theorem
follows. 
\end{proof}

\bibliography{references}

\providecommand{\bysame}{\leavevmode\hbox to3em{\hrulefill}\thinspace}
\providecommand{\MR}{\relax\ifhmode\unskip\space\fi MR }
\providecommand{\MRhref}[2]{%
  \href{http://www.ams.org/mathscinet-getitem?mr=#1}{#2}
}
\providecommand{\href}[2]{#2}
\begin{thebibliography}{BvdB03}

\bibitem[Add11]{Addington-NewDerivedSymmetriesOfSomeHyperkaehlerVarieties}
Nicolas Addington, \emph{New derived symmetries of some hyperkaehler
  varieties}, arXiv:1112.0487, (2011).

\bibitem[AL]{AnnoLogvinenko-OnBraidingCriteriaForSphericalTwistsByFlatFibrations}
Rina Anno and Timothy Logvinenko, \emph{Braiding criteria for spherical
  fibrations}, (in preparation).

\bibitem[AL10]{AnnoLogvinenko-OrthogonallySphericalObjectsAndSphericalFibrations}
\bysame, \emph{Orthogonally spherical objects and spherical fibrations},
  arXiv:1011.0707, (2010).

\bibitem[AL12]{AnnoLogvinenko-OnTakingTwistsOfFourierMukaiFunctors}
\bysame, \emph{On adjunctions for {F}ourier-{M}ukai transforms}, Adv. in Math.
  \textbf{231} (2012), no.~3--4, 2069--2115, arXiv:1004.3052.

\bibitem[Ann07]{Anno-SphericalFunctors}
Rina Anno, \emph{Spherical functors}, arXiv:0711.4409, (2007).

\bibitem[BK90]{BondalKapranov-EnhancedTriangulatedCategories}
Alexei Bondal and Mikhail Kapranov, \emph{Enhanced triangulated categories},
  Mat. Sb. \textbf{181} (1990), no.~5, 669--683.

\bibitem[BLL04]{BondalLarsenLunts-GrothendieckRingofPretriangulatedCategories}
Alexei Bondal, Michael Larsen, and Valery Lunts, \emph{Grothendieck ring of
  pretriangulated categories}, Int. Math. Res. Not. \textbf{29} (2004),
  1461--1495, arXiv:math/0401009.

\bibitem[BP10]{BroomheadPloog-AutoequivalencesOfToricSurfaces}
Nathan Broomhead and David Ploog, \emph{Autoequivalences of toric surfaces},
  arXiv:1010.1717, (2010).

\bibitem[BPP13]{BroomheadPauksztelloPloog-DiscreteDerivedCategoriesIHomomorphismsAutoequivalencesAndTStructures}
Nathan Broomhead, David Pauksztello, and David Ploog, \emph{Discrete derived
  categories {I}: homomorphisms, autoequivalences and t-structures},
  arXiv:1312.5203, (2013).

\bibitem[Bri08]{Bridgeland-StabilityConditionsOnK3Surfaces}
Tom Bridgeland, \emph{Stability conditions on {K}3 surfaces}, Duke Math J.
  \textbf{141} (2008), 241--291, arXiv:math/0307164.

\bibitem[Bri09]{Bridgeland-StabilityConditionsAndKleinianSingularities}
\bysame, \emph{Stability conditions and {K}leinian singularities}, Int. Math.
  Res. Not. \textbf{21} (2009), 4142--4157, arXiv:math/0508257.

\bibitem[BvdB03]{BondalVanDenBergh-GeneratorsAndRepresentabilityOfFunctorsInCommutativeAndNoncommutativeGeometry}
Alexei Bondal and Michel van~den Bergh, \emph{Generators and representability
  of functors in commutative and noncommutative geometry}, Mosc. Math. J
  \textbf{3} (2003), no.~1, 1--36, arXiv:math/0204218.

\bibitem[Dri04]{Drinfeld-DGQuotientsOfDGCategories}
Vladimir Drinfeld, \emph{{DG} quotients of {DG} categories}, J. Algebra
  \textbf{272} (2004), no.~2, 643--691, arXiv:math/0210114.

\bibitem[DS13]{DonovanSegal-MixedBraidGroupActionsFromDeformationsOfSurfaceSingularities}
Will Donovan and Ed~Segal, \emph{Mixed braid group actions from deformations of
  surface singularities}, arXiv:1310.7877, (2013).

\bibitem[DW13]{DonovanWemyss-NoncommutativeDeformationsAndFlops}
Will Donovan and Michael Wemyss, \emph{Noncommutative deformations and flops},
  arXiv:1309.0698, (2013).

\bibitem[HLS13]{HalpernLeistnerShipman-AutoequivalencesOfDerivedCategoriesViaGeometricInvariantTheory}
Daniel Halpern-Leistner and Ian Shipman, \emph{Autoequivalences of derived
  categories via geometric invariant theory}, arXiv:1303.5531, (2013).

\bibitem[Hor05]{Horja-DerivedCategoryAutomorphismsFromMirrorSymmetry}
Richard~Paul Horja, \emph{Derived category automorphisms from mirror symmetry},
  Duke Math. J. \textbf{127} (2005), 1--34.

\bibitem[IU05]{IshiiUehara-AutoequivalencesOfDerivedCategoriesOnTheMinimalResolutionsOfA_nSingularitiesOnSurfaces}
Akira Ishii and Hokuto Uehara, \emph{Autoequivalences of derived categories on
  the minimal resolutions of ${A}_n$-singularities on surfaces}, J. Diff. Geom.
  \textbf{71} (2005), 385--435, arXiv:math/0409151.

\bibitem[Kel94]{Keller-DerivingDGCategories}
Bernhard Keller, \emph{Deriving {DG} categories}, Ann. Sci. {\'E}cole Norm.
  Sup. \textbf{27} (1994), no.~1, 63--102.

\bibitem[Kel01]{Keller-IntroductionToAInfinityAlgebrasAndModules}
\bysame, \emph{Introduction to ${A}$-infinity algebras and modules}, Homology
  Homotopy Appl. \textbf{3} (2001), no.~1, 1--35, arXiv:math/9910179.

\bibitem[Kel06]{Keller-OnDifferentialGradedCategories}
\bysame, \emph{On differential graded categories}, Proceedings of the
  International Congress of Mathematicians, Madrid, Spain, 2006, vol.~II, Eur.
  Math. Soc., Z{\"u}rich, 2006, arXiv:math/0601185, pp.~151--190.

\bibitem[KS01]{KontsevichSoibelman-HomologicalMirrorSymmetryAndTorusFibrations}
Maxim Kontsevich and Yan Soibelman, \emph{Homological mirror symmetry and torus
  fibrations}, Symplectic geometry and mirror symmetry ({S}eoul, 2000), World
  Sci. Publ., River Edge, NJ, 2001, arXiv:math/0011041, pp.~203--263.

\bibitem[KT07]{KhovanovThomas-BraidCobordismsTriangulatedCategoriesAndFlagVarieties}
Mikhail Khovanov and Richard Thomas, \emph{Braid cobordisms, triangulated
  categories, and flag varieties}, Homology, Homotopy Appl. \textbf{9} (2007),
  no.~2, 19--94, arXiv:math/0609335.

\bibitem[LH03]{Lefevre-SurLesAInftyCategories}
Kenji Lef{\`e}vre-Hasegawa, \emph{Sur les ${A}_\infty$-cat{\'e}gories}, Ph.D.
  thesis, Universit{\'e} Denis Diderot -- Paris 7, 2003, arXiv:math/0310337.

\bibitem[LO10]{LuntsOrlov-UniquenessOfEnhancementForTriangulatedCategories}
Valery~A. Lunts and Dmitri~O. Orlov, \emph{Uniqueness of enhancement for
  triangulated categories}, J. Amer. Math. Soc. \textbf{23} (2010), 853--908,
  arXiv:0908.4187.

\bibitem[Lun10]{Lunts-CategoricalResolutionOfSingularities}
Valery~A. Lunts, \emph{Categorical resolution of singularities}, J. Algebra
  \textbf{323} (2010), no.~10, 2977 -- 3003, arXiv:0905.4566.

\bibitem[Muk87]{Mukai-OnTheModuliSpaceOfBundlesOnK3SurfacesI}
Shigeru Mukai, \emph{On the moduli space of bundles on ${K}3$ surfaces {I}},
  Vector bundles on algebraic varieties, Oxford University Press, 1987,
  pp.~341--413.

\bibitem[Rou04]{Rouquier-CategorificationOfTheBraidGroups}
Raphael Rouquier, \emph{Categorification of the braid groups},
  arXiv:math/0409593, (2004).

\bibitem[Rou08]{Rouquier-DimensionsOfTriangulatedCategories}
\bysame, \emph{Dimensions of triangulated categories}, Journal of K-theory
  \textbf{1} (2008), 193--256, arXiv:math/0310134.

\bibitem[ST01]{SeidelThomas-BraidGroupActionsOnDerivedCategoriesOfCoherentSheaves}
Paul Seidel and Richard Thomas, \emph{Braid group actions on derived categories
  of coherent sheaves}, Duke Math. J. \textbf{108} (2001), no.~1, 37--108,
  arXiv:math/0001043.

\bibitem[Sze04]{Szendroi-ArtinGroupActionsOnDerivedCategoriesOfThreefolds}
Bal{\'a}zs Szendroi, \emph{Artin group actions on derived categories of
  threefolds}, J. Reine Angew. Math. \textbf{572} (2004), 139--166,
  arXiv:math/0210121.

\bibitem[Tab05]{Tabuada-UneStructureDeCategorieDeModelesDeQuillenSurLaCategorieDesDG-Categories}
Goncalo Tabuada, \emph{Une structure de cat{\'e}gorie de mod{\`e}les de
  {Q}uillen sur la cat{\'e}gorie des dg-cat{\'e}gories}, C. R. Acad. Sci.
  Paris, Ser. I \textbf{340} (2005), 15--19, arXiv:math/0407338.

\bibitem[Tod07]{Toda-OnACertainGeneralizationOfSphericalTwists}
Yukinobu Toda, \emph{On a certain generalization of spherical twists}, Bull.
  Soc. Math. France \textbf{135} (2007), no.~1, 119--134, arXiv:math/0603050.

\bibitem[To{\"e}07]{Toen-TheHomotopyTheoryOfDGCategoriesAndDerivedMoritaTheory}
Bertrand To{\"e}n, \emph{The homotopy theory of {\em dg}-categories and derived
  {M}orita theory}, Invent. Math. \textbf{167} (2007), no.~3, 615--667,
  arXiv:math/0408337.

\bibitem[To{\"e}11]{Toen-LecturesOnDGCategories}
\bysame, \emph{Lectures on {DG}-categories}, Topics in Algebraic and
  Topological K-Theory, Lecture Notes in Mathematics, vol. 2008, Springer
  Berlin Heidelberg, 2011, pp.~243--302.

\bibitem[Ver96]{Verdier-DesCategoriesDeriveesDesCategoriesAbeliennes}
Jean-Louis Verdier, \emph{Des cat{\'e}gories d{\'e}riv{\'e}es des
  cat{\'e}gories ab{\'e}liennes}, Ast{\'e}risque \textbf{239} (1996), xii+253
  pp. (1997).

\end{thebibliography}
\bibliographystyle{amsalpha}
\end{document}